\documentclass[12pt]{amsart}
\usepackage[utf8]{inputenc}
\usepackage{amsaddr}
\usepackage{comment}
\usepackage{amsmath,amsfonts,amssymb,amsthm}
\usepackage{xcolor} 
\usepackage{subcaption}
\usepackage{graphicx}
\usepackage{float}
\usepackage{bm}
\usepackage{enumitem}
\usepackage{mathrsfs}
\usepackage{multicol}
\usepackage{centernot}
\usepackage{mathtools}
\usepackage{ stmaryrd }
\usepackage{tikz}
\usepackage{tikz-cd}
\usepackage{adjustbox}
\usetikzlibrary {graphs} 
\usepackage{url}
\usepackage{hyperref}
\usepackage[margin = 1in]{geometry}
\usepackage{graphicx}
\usepackage{float}
\usepackage{array,multirow}
\usepackage{dynkin-diagrams}
\usepackage{caption} 
\usepackage{longtable}
\usepackage{changepage}
\newtheorem{thm}{Theorem}

\newtheorem{lemma}[thm]{Lemma}
\newtheorem{prop}[thm]{Proposition}

\theoremstyle{definition}
\newtheorem{defn}[thm]{Definition}
\newtheorem{example}[thm]{Example}
\newtheorem{remark}[thm]{Remark}
\newtheorem{non-example}[thm]{Non-Example}

\DeclareMathOperator{\Res}{Res}

\DeclareMathOperator{\Sym}{Sym}
\def\mf{\mathfrak}
\def\phi{\varphi}

\def\tilde{\widetilde}

\def\C{\mathbb{C}}
\def\arrvline{\hfil\kern\arraycolsep\vline\kern-\arraycolsep\hfilneg}

\tikzset{
    /Dynkin diagram,
        indefinite edge/.style={
            ultra thick,
            densely dashed
        },
        o/.style={
            ultra thick,
            fill=black!30!white,
            draw=black
        },
        edge length=2cm, 
        edge/.style={ultra thick}, 
        arrow style={black,
                length=3mm,
                width=5mm,
                line width=2pt
                },
        root radius=.2cm,
        mark=o}

\title{An explicit classification of dual pairs in exceptional Lie algebras}
\author{Marisa Gaetz}
\email{mgaetz@mit.edu}
\thanks{The author was supported by the NSF Graduate Research Fellowship Program under Grant No.~2141064, and by the Fannie \& John Hertz Foundation.}

\begin{document}
\renewcommand{\arraystretch}{1.5}

\maketitle

\begin{abstract}
The primary goal of this paper is to explicitly write down all semisimple \textit{dual pairs} in the exceptional Lie algebras. (A \textit{dual pair} in a reductive Lie algebra $\mf{g}$ is a pair of subalgebras such that each member equals the other's centralizer in $\mf{g}$.) In a 1994 paper, H.~Rubenthaler outlined a process for generating a complete list of candidate dual pairs in each of the exceptional Lie algebras. However, the process of checking whether each of these candidate dual pairs is in fact a dual pair is not easy, and requires several distinct insights and methods. In this paper, we carry out this process and explain the relevant concepts as we go. We also give plenty of examples with the hopes of making Rubenthaler's 1994 result not only more complete but more usable and understandable. 
\end{abstract}

\section{Introduction}

The notion of a \textit{reductive dual pair}\footnote{Italicized terms will be defined later in the paper.} in a complex reductive Lie algebra was first introduced by Howe in his seminal 1989 paper \cite{HoweRemarks}, in which he also classified the reductive dual pairs in the symplectic Lie algebra $\mf{sp}_{2m}$. Since then, reductive dual pairs have been widely used and studied. For example, certain reductive dual pairs are implicitly used in some of the well-known constructions of the exceptional Lie algebras: $(G_2^1, A_2^{2''})$ and $(G_2^1, F_4^1)$ correspond to the Freudenthal-Tits constructions of $E_6$ and $E_8$, respectively \cite{tits-construction}, and $(D_4, D_4)$ corresponds to the triality construction of $E_8$ \cite{allison, barton-sudbery}. Additionally, reductive dual pairs have some beautiful connections to the study of commuting nilpotent elements in complex semisimple Lie algebras as introduced by Ginzburg in \cite{Ginzburg}; see \cite{Panyushev} for an explanation of these connections. More recently, the real forms of certain reductive dual pairs in exceptional Lie algebras have also drawn some attention \cite{KovT, Kov}.   

Some of the earliest work on reductive dual pairs appears in Rubenthaler's 1994 paper \cite{rubenthaler}. In this paper, Rubenthaler outlines a classification of reductive dual pairs in complex reductive Lie algebras. While Rubenthaler's paper provides a helpful framework, writing down all of these pairs explicitly requires a lot of additional nontrivial work. The primary goal of this paper is to explicitly write down all of the semisimple dual pairs in the exceptional Lie algebras. Here is how we will proceed: 
\begin{itemize}
\item \textbf{Section \ref{sec:prelims}}: We reduce the problem of classifying reductive dual pairs in reductive Lie algebras to the problem of classifying \textit{semisimple} dual pairs in \textit{simple} Lie algebras.
\item \textbf{Section \ref{sec:max-reg}}: We classify dual pairs coming from \textit{maximal regular} subalgebras of \textit{maximal rank}.
\item \textbf{Section \ref{sec:admissible}}: We define and classify \textit{admissible subalgebras}.
\item \textbf{Section \ref{sec:S-irred}}: We introduce \textit{$S$-irreducible} dual pairs and \textit{admissible dual pairs}, and describe how they nearly all arise from admissible subalgebras.
\item \textbf{Section \ref{sec:lists-classical}}: We classify the semisimple dual pairs in the complex classical simple Lie algebras that are $S$-irreducible, admissible, or that arise from maximal regular subalgebras of maximal rank. We write out all such dual pairs for rank up to 6.
\item \textbf{Section \ref{sec:non-S-irred}}: We explain how to obtain lists of candidate non-$S$-irreducible dual pairs (including all \textit{actual} non-$S$-irreducible dual pairs). 
\item \textbf{Section \ref{sec:elim-and-confirm}}: We introduce methods for determining whether or not a candidate dual pair is actually a dual pair.
\item \textbf{Section \ref{sec:lists}}: We classify all ($S$-irreducible and non-$S$-irreducible) semisimple dual pairs in the exceptional Lie algebras.
\end{itemize}

\section{Preliminaries} \label{sec:prelims}

In this section, we introduce several important notions and conventions, and we explain why we are focusing on \textit{semisimple} dual pairs in \textit{simple} Lie algebras. 

\begin{defn}[{\cite[p.~550]{HoweRemarks}}]
Let $\mf{g}$ be a reductive Lie algebra. Then a pair $(\mf{a}, \mf{b})$ of subalgebras of $\mf{g}$ is a \textbf{reductive dual pair} if
\begin{enumerate}
\item $\mf{a}$ and $\mf{b}$ are reductive in $\mf{g}$, and
\item $\mf{a} = \mf{z}_{\mf{g}}(\mf{b})$ and $\mf{b} = \mf{z}_{\mf{g}} (\mf{a})$,
\end{enumerate}
where $\mf{z}_{\mf{g}}(\cdot)$ denotes taking the centralizer in $\mf{g}$.
\end{defn}

Since we will only be considering \textit{reductive} dual pairs in this paper, we will just use ``dual pair" to refer to a reductive dual pair.

\subsection{Why we are focusing on \textit{semisimple} dual pairs in \textit{simple} Lie algebras}

As discussed in \cite[Section 5]{rubenthaler}, we can reduce the study of reductive dual pairs in reductive Lie algebras to the study of semisimple dual pairs in simple Lie algebras. In particular, the following result reduces the problem to finding reductive dual pairs in semisimple Lie algebras:

\begin{lemma}[{\cite[Lemma 5.1]{rubenthaler}}]
Let $\mf{g}$ be a reductive Lie algebra. Write $\mf{g} = Z_{\mf{g}} \oplus [\mf{g}, \mf{g}]$ (where $Z_{\mf{g}}$ denotes the center of $\mf{g}$). If $(\mf{a}, \mf{b})$ is a dual pair in $\mf{g}$, then $\mf{a} = Z_{\mf{g}} \oplus \mf{a}_1$ and $\mf{b} = Z_{\mf{g}} \oplus \mf{b}_1$, where $(\mf{a}_1, \mf{b}_1)$ is a dual pair in $[\mf{g},\mf{g}]$. Conversely, any dual pair in $[\mf{g},\mf{g}]$ gives a dual pair of $\mf{g}$ in this way. 
\end{lemma}

Lemma \ref{lem:rub-levis} further shows that any reductive dual pair $(\mf{a}, \mf{b})$ in a semisimple Lie algebra $\mf{g}$ can be written in terms of $Z:= \mf{a} \cap \mf{b}$ and a semisimple dual pair in $[\ell_Z, \ell_Z]$, where $\ell_Z=\mf{z}_{\mf{g}}(Z)$ is a \textit{Levi subalgebra} of $\mf{g}$ with center $Z$. Before formally stating this result, let us introduce Levi subalgebras (following \cite[Section 1.2]{rubenthaler}):

Let $\mf{g}$ be a semisimple Lie algebra, and let $\mf{h}$ be a Cartan subalgebra of $\mf{g}$. Let $\Psi = (\alpha_1 , \ldots , \alpha_n)$ be a choice of simple roots associated to $(\mf{g}, \mf{h})$. For a subset $\theta$ of $\Psi$, let $H_{\theta}$ be the unique element of $\mf{h}$ defined by
  \[
    \left\{\begin{array}{ll}
        \alpha (H_{\theta}) = 0 & \text{if } \alpha \in \theta \\
        \alpha (H_{\theta}) = 2 & \text{if } \alpha \in \Psi \setminus \theta         \end{array}\right.
  \]  
For $p \in \mathbb{Z}$, define
$$d_p (\theta) = \{ X \in \mf{g} \; \vert \; [H_{\theta}, X] = 2pX \} \hspace{.5cm} \text{ so that } \hspace{.5cm} \mf{g} = \bigoplus_{p \in \mathbb{Z}} d_p (\theta).$$

\begin{defn}[{\cite[Section 1.2]{rubenthaler}}]
The reductive algebra $\ell_{\theta} := d_0 (\theta)$ is called the \textbf{standard Levi subalgebra associated to $\theta$}.
\end{defn}

\begin{lemma}[{\cite[Lemma 5.2]{rubenthaler}}] \label{lem:rub-levis}
Let $\mf{g}$ be a semisimple Lie algebra and let $(\mf{a}, \mf{b})$ be a dual pair in $\mf{g}$. If $\mf{a}$ has nontrivial center $Z_{\mf{a}}$, then $\mf{a} \cap \mf{b} = Z_{\mf{a}} = Z_{\mf{b}} =:Z$. Let $\ell_Z = \mf{z}_{\mf{g}} (Z)$. The algebra $\ell_{Z}$ is a Levi subalgebra of $\mf{g}$ whose center is $Z$. If we write $\ell_Z = Z \oplus [\ell_Z, \ell_Z]$, $\mf{a} = Z \oplus [\mf{a},\mf{a}]$, and $\mf{b} = Z \oplus [\mf{b},\mf{b}]$, then $([\mf{a},\mf{a}], [\mf{b}, \mf{b}])$ is a dual pair in $[\ell_Z, \ell_Z]$. Conversely, any dual pair in $[\ell_Z,\ell_Z]$ gives a dual pair of $\mf{g}$ in this way.
\end{lemma}

Finally, the following result further reduces the problem to finding semisimple dual pairs in simple Lie algebras:

\begin{lemma}[{\cite[Lemma 5.3]{rubenthaler}}]
Let $\mf{g}$ be a semisimple Lie algebra and let $\mf{g} = \mf{g}_1 \oplus \cdots \oplus \mf{g}_n$ be its decomposition into simple ideals. If for each $i \in \{ 1 , \ldots , n \}$ we have a dual pair $(\mf{a}_i, \mf{b}_i)$ in $\mf{g}_i$, then $\left ( \prod_{i=1}^n \mf{a}_i , \prod_{i=1}^n \mf{b}_i \right )$ is a dual pair in $\mf{g}$. Conversely, any dual pair in $\mf{g}$ is obtained in this way.
\end{lemma}

Note that we will often omit the adjective ``semisimple" in the following discussion of the semisimple dual pairs in the complex simple Lie algebras. (All subalgebras that follow can be assumed to be semisimple.)

\subsection{Defining a normalized symmetric form on \texorpdfstring{$\mf{g}$}{g}}

Let $\mf{g}$ be a complex simple Lie algebra, and fix a Cartan subalgebra $\mf{h}$. Let $\mathcal{R}$ denote the root system corresponding to $(\mf{g}, \mf{h})$. Let $\mf{g}_0$ be a compact real form of $\mf{g}$ with Cartan $\mf{h}_0$ (defined over $\mathbb{R}$) satisfying $(\mf{h}_0)_{\mathbb{C}} = \mf{h}$. Recall that any $\alpha \in \mathcal{R}$ is determined (by $\mathbb{C}$-linearity) by its values on $i\mf{h}_0$, on which it is real-valued. Therefore, we can view $\mathcal{R}$ as sitting inside $(i\mf{h}_0)^*$. We identify $(i\mf{h}_0)^*$ with a subset of $\mathbb{R}^n$ such that the standard Euclidean inner product $\langle \cdot, \cdot \rangle$ on $\mathbb{R}^n$ restricts to a Weyl group-invariant form on $(i\mf{h}_0)^*$. In particular, we use the root system realizations described in \cite[Plates I--IX]{bourbaki}, writing elements of $i\mf{h}_0$ in terms of the standard Euclidean basis $\{ e_1, \ldots , e_n \}$ with $e_1 = (1, 0 ,\ldots , 0)$, etc., and writing elements of $(i\mf{h}_0)^*$ in terms of the corresponding standard basis of linear functionals $\{ \varepsilon_1 , \ldots , \varepsilon_n \}$.

Following \cite[Section 2]{dynkin}, we normalize a non-degenerate invariant symmetric bilinear form $(,)_{\mf{g}}$ on $\mf{g}$ as follows: for $\beta^{\vee}$ a short coroot in $\mathcal{R}^{\vee}$, we require that $(\beta^{\vee}, \beta^{\vee})_{\mf{g}} = 2$. We also define a non-degenerate invariant symmetric bilinear form $(,)_{\mf{g}^*}$ on $\mf{g}^*$ by the condition $(\beta , \beta )_{\mf{g}^*} = 2$ for $\beta$ a long root in $\mathcal{R}$. (Note that $(,)_{\mf{g}^*}$ can also be defined as the dual form to the form on $\mf{g}$; this dual form automatically satisfies the aforementioned normalizing condition.)

\subsection{Dynkin's index notation}

We now introduce Dynkin's index notation (cf.~\cite[Section 2]{dynkin}). Let $\phi : \mf{a} \rightarrow \mf{g}$ be a homomorphism of complex simple Lie algebras. For $x,y \in \mf{a}$, the bilinear form $(x,y) \mapsto ( \phi(x), \phi(y) )_{\mf{g}}$ is proportional to $(x,y) \mapsto (x,y)_{\mf{a}}$. The \textbf{index $i_{\phi}$ of $\phi$} is defined by
$$(\phi (x), \phi (y))_{\mf{g}} = i_{\phi} \cdot (x,y)_{\mf{a}}, \hspace{.25cm} x,y \in \mf{a};$$
the index is always an integer \cite[Theorem 2.2]{dynkin}. In particular, if $\mf{a}$ is a simple subalgebra of $\mf{g}$, then the \textbf{index of $\mf{a}$ in $\mf{g}$} is given by
$$i_{\mf{a} \hookrightarrow \mf{g}} = \frac{(x,x)_{\mf{g}}}{(x,x)_{\mf{a}}}, \hspace{.25cm} x \in \mf{a}.$$
When we have a sequence of inclusions $\mf{a} \hookrightarrow \mf{a}' \hookrightarrow \mf{g}$, the index satisfies the following multiplicative property:
$$i_{\mf{a}\hookrightarrow \mf{g}} = i_{\mf{a} \hookrightarrow \mf{a}'} \cdot i_{\mf{a}'\hookrightarrow \mf{g}}.$$

The index of $\mf{a}$ in $\mf{g}$ often determines the conjugacy class of $\mf{a}$ (where the conjugation is by the adjoint group of $\mf{g}$). When the conjugacy class of $\mf{a}$ is fully specified by its type and index, we will follow Dynkin in writing $\text{Type}(\mf{a})^{ i_{\mf{a} \hookrightarrow \mf{g}} }$ to refer to the conjugacy class of $\mf{a}$. If, on the other hand, the conjugacy class of $\mf{a}$ is not specified by its type and index, we will add ``primes" to distinguish the possible conjugacy classes. For example, by \cite[Table 25]{dynkin}, there are two conjugacy classes of subalgebras of the complex Lie algebra $E_7$ that have type $A_2$ and index 2; we will refer to them as $A_2^{2'}$ and $A_2^{2''}$ (by order of appearance in \cite[Table 25]{dynkin}). Similarly, by Appendix B, there are two conjugacy classes of subalgebras of the complex Lie algebra $B_4$ that have type $B_2$ and index 1; we will refer to them as $B_2^{1'}$ and $B_2^{1''}$ (by order of appearance in Table \ref{table:B2-subalgebras}). To know whether we need to add these ``primes," we will continue to refer to \cite{dynkin} for subalgebras of the exceptional Lie algebras and to Appendices A and B for subalgebras of the classical simple Lie algebras.

\begin{example}
Consider $\mf{sp}_2 \oplus \mf{sp}_{2m} \subset \mf{so}_{4m}$. (Here, $\mf{sp}_2 \hookrightarrow \mf{so}_{4m}$ can be thought of as a diagonal embedding with $2m$ copies, and $\mf{sp}_{2m} \hookrightarrow \mf{so}_{4m}$ can be thought of as ``multiplying" each entry by the $2 \times 2$ identity matrix.) Let's compute the index of $\mf{sp}_2$ and $\mf{sp}_{2m}$ in $\mf{so}_{4m}$. Following \cite[Plate IV]{bourbaki}, $\mf{so}_{4m}$ has 
$$\{ \varepsilon_1 - \varepsilon_2 , \, \varepsilon_2 - \varepsilon_3 , \ldots , \, \varepsilon_{2m-1} - \varepsilon_{2m} , \, \varepsilon_{2m-1} + \varepsilon_{2m} \}$$ 
as its simple roots, with corresponding simple coroots
$$\{ e_1 - e_2 , \, e_2 - e_3 , \ldots , \, e_{2m-1} - e_{2m} , \, e_{2m-1} + e_{2m} \}.$$
Set $z:= e_1 - e_2$. This coroot corresponds to the diagonal matrix 
$$\text{diag}(1,-1,0,\ldots, 0 ; -1,1,0,\ldots ,0)$$
in $\mf{so}_{4m}$, so we see that the normalizing condition on $(\cdot , \cdot )_{\mf{so}_{4m}}$ is 
$$(z,z)_{\mf{so}_{4m}} = 2 = \frac{1}{2} \langle z , z \rangle,$$
where $\langle \cdot , \cdot \rangle$ denotes the standard Euclidean inner product. 

Now, following \cite[Plate III]{bourbaki}, $\mf{sp}_2$ has $2 \varepsilon_1$ as its simple root, and has $x:=e_1$ as the corresponding simple coroot. Similarly, $\mf{sp}_{2m}$ has simple roots 
$$\{ \varepsilon_1 - \varepsilon_2, \, \varepsilon_2 - \varepsilon_3 , \ldots , \, \varepsilon_{m-1} - \varepsilon_m, 2 \varepsilon_m \},$$
with corresponding simple coroots 
$$\{ e_1 - e_2, \, e_2 - e_3, \ldots , \, e_{m-1} - e_m, \, y:=e_m \}.$$
Thus, the normalizing conditions on $(\cdot , \cdot )_{\mf{sp}_2}$ and $(\cdot , \cdot )_{\mf{sp}_{2m}}$ are $(x, x)_{\mf{sp}_2} = 2$ and $(y,y)_{\mf{sp}_{2m}} = 2$.

Here, $x$ corresponds to the matrix $\text{diag}(1;-1)$ in $\mf{sp}_2$ and $y$ corresponds to the matrix $\text{diag}(0,\ldots , 0 , 1 ; 0,\ldots , 0 , -1)$ in $\mf{sp}_{2m}$. Since $\mf{sp}_2$ is embedded diagonally (with $2m$ copies) in $\mf{so}_{4m}$, we see that
\begin{align*}
     (x,x)_{\mf{so}_{4m}} = \frac{1}{2} \Big  \langle & \text{diag}(1,-1,\ldots , 1,-1 ; -1,1,\ldots , -1,1) , \\
     & \text{diag}(1,-1,\ldots , 1,-1 ; -1,1,\ldots , -1,1) \Big \rangle = 2m.
\end{align*}

Similarly, we see that 
\begin{align*}
    \displaystyle (y,y)_{\mf{so}_{4m}} = \frac{1}{2} \Big \langle & \text{diag} (0,\ldots,0,1,1;0,\ldots,0,-1,-1), \\
    & \text{diag} (0,\ldots,0,1,1;0,\ldots,0,-1,-1) \Big \rangle = 2.
\end{align*}

It follows that
$$i_{\mf{sp}_2 \hookrightarrow \mf{so}_{4m}} = \frac{(x,x)_{\mf{so}_{4m}}}{(x,x)_{\mf{sp}_2}} = \frac{2m}{2} = m \hspace{.5cm} \text{ and } \hspace{.5cm} i_{\mf{sp}_{2m} \hookrightarrow \mf{so}_{4m}} = \frac{ (y,y)_{\mf{so}_{4m}} }{ (y,y)_{\mf{sp}_{2m}} } = \frac{2}{2} = 1.$$
In this way, we see that $\mf{sp}_2 = C_1^{m} \simeq A_1^m$ (possibly with ``primes") and that $\mf{sp}_{2m} = C_m^1$ (possibly with ``primes"). This will appear in later examples of dual pairs; specifically, $(A_1^{2'},B_2^{1'})$, $(A_1^{2''}, B_2^{1''})$, and $(A_1^{2'''}, B_2^{1'''})$ are dual pairs in $D_4$, and $(A_1^{3'}, C_3^{1''})$ and $(A_1^{3''}, C_3^{1'})$ are dual pairs in $D_6$ (see Table \ref{table:dps-in-Dn}).  
\end{example}

\subsection{Dynkin's defining vector conventions and weighted diagrams}

Let $\mf{a}$ be a three-dimensional semisimple subalgebra of a complex simple Lie algebra $\mf{g}$. (Any such subalgebra is isomorphic to $\mf{sl}_2$.) A vector $H \in \mf{a}$ which can be supplemented with two vectors $X,Y \in \mf{a}$ to form an $\mf{sl}_2$-triplet, i.e.~a basis of $\mf{a}$ satisfying the commutator relations
\begin{align*}
[H,X] &= 2X \\
[H,Y] &= -2Y \\
[X,Y] &= H,
\end{align*}
is called a \textbf{defining vector} of $\mf{a}$. 

\begin{remark}
For a three-dimensional subalgebra $\mf{a}$ of $\mf{g}$ with defining vector $H$, the index can be computed as 
$$i_{\mf{a} \hookrightarrow \mf{g}} = \frac{(H,H)_{\mf{g}}}{(H,H)_{\mf{a}}} = \frac{(H,H)_{\mf{g}}}{2}.$$
\end{remark}

Fix a defining vector $H$ for $\mf{a}$. Choose a Cartan subalgebra $\mf{h}$ containing $H$, as well as a positive root system making $H$ weakly dominant (i.e.~$\alpha (H) \ge 0$ whenever $\alpha >0$). Let $\Pi$ be the corresponding system of simple roots. In \cite{dynkin}, Dynkin often writes the defining vector of $\mf{a}$ in terms of the simple coroots of $\mf{g}$ (appropriately normalized). In particular, let $\Pi_{\ell}$ denote the set of long simple roots in $\Pi$ and let $\Pi_{s}$ denote the set of short simple roots in $\Pi$. Set $r := ( \beta , \beta )_{\mf{g}} / (\gamma , \gamma)_{\mf{g}}$ for $\beta \in \Pi_{\ell}$ and $\gamma \in \Pi_s$ (so that $r=1$ if $\mf{g}$ has only one root length, $r=3$ if $\mf{g}$ is of type $G_2$, and $r=2$ in all other cases of two root lengths). Then Dynkin writes the defining vector $H$ of $\mf{a}$ as 
$$H = \sum_{ \beta \in \Pi_{\ell} } a_{\beta} \beta^{\vee} + \sum_{\gamma \in \Pi_s} a_{\gamma} \frac{\gamma^{\vee}}{r},$$
and he calls the coefficients $\{ a_{\beta}, a_{\gamma} \}_{ \beta \in \Pi_{\ell}, \, \gamma \in \Pi_s }$ the \textbf{coordinates of the defining vector} $H$.  

Now, consider the Dynkin diagram for $\mf{g}$. There is a unique correspondence between the nodes of this diagram and the elements of $\Pi$. Let us associate the number $\alpha (H)$ with the node corresponding to $\alpha$. The resulting diagram (with numbers written in) is called the \textbf{weighted diagram} of the subalgebra $\mf{a}$.  

\begin{figure}
\begin{tikzpicture}[scale=1.65]
    \begin{scope}[xshift=-1.6cm]
    \node at (.25,.65) {$G_2$};
    \node (dot2) at (0,0) [draw, circle, inner sep=2pt, fill=black] {};
    \node at (-.1,.2) {$\alpha_1$};
    \node (dot3) at (.5,0) [draw, circle, inner sep=2pt, fill=black] {};
    \node at (.6,.2) {$\alpha_2$};
    \draw (dot2.north) -- (dot3.north);
    \draw (dot2.south) -- (dot3.south);
    \draw (dot2) -- (dot3);
    \draw (.32,0) -- (.18,.12);
    \draw (.32,0) -- (.18,-.12);
    \end{scope}

    \begin{scope}[xshift=.5cm]
    \node at (.75,.65) {$F_4$};
    \node (dot1) at (0,0) [draw, circle, inner sep=1.5pt, fill=black] {};
    \node at (0,.2) {$\alpha_1$};
    \node (dot2) at (.5,0) [draw, circle, inner sep=1.5pt, fill=black] {};
    \node at (.5,.2) {$\alpha_2$};
    \node (dot3) at (1,0) [draw, circle, inner sep=1.5pt, fill=black] {};
    \node at (1,.2) {$\alpha_3$};
    \node (dot4) at (1.5,0) [draw, circle, inner sep=1.5pt, fill=black] {};
    \node at (1.5,.2) {$\alpha_4$};
    \draw (dot2.north) -- (dot3.north);
    \draw (dot2.south) -- (dot3.south);
    \draw (.82,0) -- (.68,.12);
    \draw (.82,0) -- (.68,-.12);
    \draw (dot1) -- (dot2);
    \draw (dot3) -- (dot4);
    \end{scope}

    \begin{scope}[xshift=3.4cm]

    \node at (1,.65) {$E_6$};
    \node (dot1) at (0,0) [draw, circle, inner sep=1.5pt, fill=black] {};
    \node at (0,.2) {$\alpha_1$};
    \node (dot2) at (.5,0) [draw, circle, inner sep=1.5pt, fill=black] {};
    \node at (.5,.2) {$\alpha_3$};
    \node (dot3) at (1,0) [draw, circle, inner sep=1.5pt, fill=black] {};
    \node at (1,.2) {$\alpha_4$};
    \node (dot4) at (1.5,0) [draw, circle, inner sep=1.5pt, fill=black] {};
    \node at (1.5,.2) {$\alpha_5$};
    \node (dot5) at (2,0) [draw, circle, inner sep=1.5pt, fill=black] {};
    \node at (2,.2) {$\alpha_6$};
    \node (dot6) at (1,-.5) [draw, circle, inner sep=1.5pt, fill=black] {};
    \node at (1,-.7) {$\alpha_2$};

    \draw (dot1) -- (dot2);
    \draw (dot2) -- (dot3);
    \draw (dot3) -- (dot4);
    \draw (dot3) -- (dot6);
    \draw (dot4) -- (dot5);
    \end{scope}

    \begin{scope}[xshift=-1.5cm,yshift=-1.5cm]
    \node at (1,.65) {$E_7$};
    \node (dot1) at (0,0) [draw, circle, inner sep=1.5pt, fill=black] {};
    \node at (0,.2) {$\alpha_1$};
    \node (dot2) at (.5,0) [draw, circle, inner sep=1.5pt, fill=black] {};
    \node at (.5,.2) {$\alpha_3$};
    \node (dot3) at (1,0) [draw, circle, inner sep=1.5pt, fill=black] {};
    \node at (1,.2) {$\alpha_4$};
    \node (dot4) at (1.5,0) [draw, circle, inner sep=1.5pt, fill=black] {};
    \node at (1.5,.2) {$\alpha_5$};
    \node (dot5) at (2,0) [draw, circle, inner sep=1.5pt, fill=black] {};
    \node at (2,.2) {$\alpha_6$};
    \node (dot7) at (2.5,0) [draw, circle, inner sep=1.5pt, fill=black] {};
    \node at (2.5,.2) {$\alpha_7$};
    \node (dot6) at (1,-.5) [draw, circle, inner sep=1.5pt, fill=black] {};
    \node at (1,-.7) {$\alpha_2$};

    \draw (dot1) -- (dot2);
    \draw (dot2) -- (dot3);
    \draw (dot3) -- (dot4);
    \draw (dot3) -- (dot6);
    \draw (dot4) -- (dot5);
    \draw (dot5) -- (dot7);
    \end{scope}

    \begin{scope}[xshift=2cm, yshift=-1.5cm]
    \node at (1,.65) {$E_8$};
    \node (dot1) at (0,0) [draw, circle, inner sep=1.5pt, fill=black] {};
    \node at (0,.2) {$\alpha_1$};
    \node (dot2) at (.5,0) [draw, circle, inner sep=1.5pt, fill=black] {};
    \node at (.5,.2) {$\alpha_3$};
    \node (dot3) at (1,0) [draw, circle, inner sep=1.5pt, fill=black] {};
    \node at (1,.2) {$\alpha_4$};
    \node (dot4) at (1.5,0) [draw, circle, inner sep=1.5pt, fill=black] {};
    \node at (1.5,.2) {$\alpha_5$};
    \node (dot5) at (2,0) [draw, circle, inner sep=1.5pt, fill=black] {};
    \node at (2,.2) {$\alpha_6$};
    \node (dot7) at (2.5,0) [draw, circle, inner sep=1.5pt, fill=black] {};
    \node at (2.5,.2) {$\alpha_7$};
    \node (dot8) at (3,0) [draw, circle, inner sep=1.5pt, fill=black] {};
    \node at (3,.2) {$\alpha_8$};
    \node (dot6) at (1,-.5) [draw, circle, inner sep=1.5pt, fill=black] {};
    \node at (1,-.7) {$\alpha_2$};

    \draw (dot1) -- (dot2);
    \draw (dot2) -- (dot3);
    \draw (dot3) -- (dot4);
    \draw (dot3) -- (dot6);
    \draw (dot4) -- (dot5);
    \draw (dot5) -- (dot7);
    \draw (dot7) -- (dot8);
    \end{scope}
    
\end{tikzpicture}
    \caption{Root labeling conventions for types $G_2$, $F_4$, $E_6$, $E_7$, and $E_8$.}
    \label{fig:root-labelings}
\end{figure}
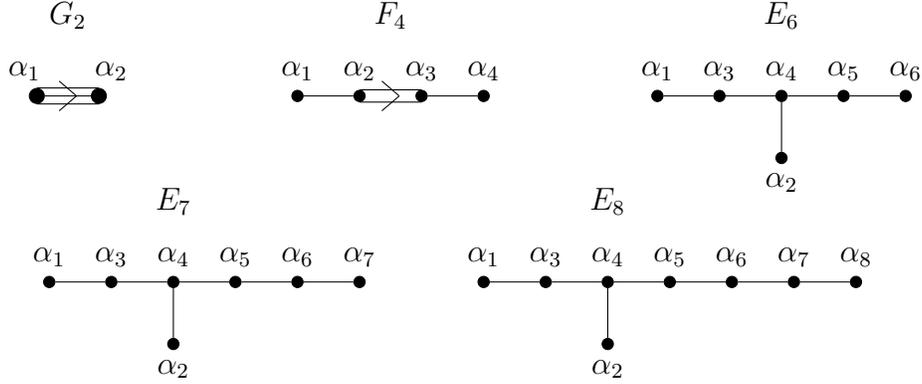

\begin{example}
To illustrate the above definitions, let us verify some of the entries in \cite[Table 16]{dynkin}, which states the index, weighted diagram, and defining vector for the three-dimensional subalgebras of the exceptional Lie algebra $\mf{g}$ of type $G_2$.

By \cite[Plate IX]{bourbaki}, we can realize this root system as the hyperplane in $\mathbb{R}^3$ defined by the equation $\xi_1 + \xi_2 + \xi_3 = 0$, with the following roots:
\begin{align*}
& \pm (\varepsilon_1 - \varepsilon_2), \hspace{.25cm} \pm (\varepsilon_1 - \varepsilon_3), \hspace{.25cm} \pm (\varepsilon_2 - \varepsilon_3),\\
 \pm  (2\varepsilon_1 & - \varepsilon_2 - \varepsilon_3), \hspace{.25cm} \pm (2 \varepsilon_2 - \varepsilon_1 - \varepsilon_3), \hspace{.25cm} \pm (2 \varepsilon_3 - \varepsilon_1 - \varepsilon_2).
\end{align*}
This system has basis $\alpha_1 = -2\varepsilon_1 + \varepsilon_2 + \varepsilon_3$ and $\alpha_2 = \varepsilon_1 - \varepsilon_2$ (cf.~Figure \ref{fig:root-labelings}), with corresponding coroots $\alpha_1^{\vee} = \frac{1}{3} (-2,1,1)$ and $\alpha_2^{\vee} = (1,-1,0)$. 

Let us consider the first subalgebra in \cite[Table 16]{dynkin}, with defining vector coordinates $(2 ,3)$. Then the defining vector is given by
$$H:= 2 \alpha_1^{\vee} + 3 \left ( \frac{\alpha_2^{\vee}}{3} \right ) = \left ( - \frac{1}{3}, - \frac{1}{3}, \frac{2}{3} \right ) .$$
So we see that 
\begin{align*}
\alpha_1 (H) &= \frac{2}{3} - \frac{1}{3} + \frac{2}{3} = 1, \text{ and } \\
\alpha_2 (H) &= - \frac{1}{3} + \frac{1}{3} + 0 = 0.
\end{align*}
Therefore, the weighted diagram of this three-dimensional subalgebra is
\begin{center}
\begin{tabular}{c c}
1 & 0
\end{tabular}
\end{center}
as expected. Since $\alpha_1^{\vee}$ is our short coroot, the normalizing condition on $(\cdot , \cdot)_{\mf{g}}$ is that $(\alpha_1^{\vee}, \alpha_1^{\vee})_{\mf{g}} = 2 = 3 \cdot \langle \alpha_1^{\vee}, \alpha_1^{\vee} \rangle$. Therefore, the index of this subalgebra in $G_2$ is given by
$$\frac{( H ,H )_{ \mf{g} }}{2} = \frac{3 \cdot (2/3) }{2} = 1,$$
as expected.
\end{example}

\section{Maximal and maximal-rank regular subalgebras} \label{sec:max-reg}

In this section, we introduce the notions of \textit{maximal regular subalgebras} and \textit{maximal-rank regular subalgebras}. We also classify the semisimple dual pairs coming from maximal regular subalgebras of maximal rank (see Proposition \ref{prop:maximal-reg-is-dp} and Table \ref{table:max-reg-dps}) and discuss briefly how maximal-rank regular subalgebras will play a crucial role later in the paper.

\begin{defn}[{\cite[Definition 5.6]{rubenthaler}}]
Let $\mf{g}$ be a semisimple Lie algebra. A subalgebra $\mf{u}$ of $\mf{g}$ is called \textbf{regular} if there exists a Cartan subalgebra $\mf{h}$ of $\mf{g}$ such that $\mf{u}$ is invariant under $\text{ad}_{\mf{g}} \mf{h}$ (i.e. $[\mf{h},\mf{u}]\subset \mf{u}$). 
\end{defn}

\begin{defn}
Let $\mf{g}$ be a semisimple Lie algebra.
\begin{itemize}
\item A \textbf{maximal regular subalgebra} of $\mf{g}$ is a proper regular semisimple subalgebra $\tilde{\mf{g}}$ of $\mf{g}$ for which no regular semisimple subalgebra $\mf{g}^*$ exists such that $\tilde{\mf{g}} \subsetneq \mf{g}^* \subsetneq \mf{g}$.
\item A \textbf{maximal-rank regular subalgebra} of $\mf{g}$ is a proper regular semisimple subalgebra with rank equal to the rank of $\mf{g}$.\footnote{In \cite{dynkin} and \cite{rubenthaler}, Dynkin and Rubenthalter do not require \textit{maximal regular subalgebras} or \textit{maximal-rank regular subalgebras} to be semisimple. Since the maximal/maximal-rank regular subalgebras we consider in this paper will all be semisimple, we include the ``semisimple" qualifier in the definition for convenience.} 
\end{itemize}
\end{defn}

As we will see below, a maximal-rank regular subalgebra need not be maximal, and a maximal regular subalgebra need not have maximal rank. It will be important to keep this in mind moving forward. In fact, mistakenly conflating these notions seems to have contributed to some errors in \cite{dynkin} and \cite{rubenthaler} that we will soon be discussing.

\subsection{Maximal-rank regular subalgebras}

Let $\mf{g}$ be a simple Lie algebra. It was shown in \cite[No.~17]{dynkin} that the maximal-rank regular subalgebras in $\mf{g}$ are obtained (up to conjugation) by applying a finite number of \textit{elementary operations} to the Dynkin diagram $D$ of $\mf{g}$. 

In each \textit{elementary operation}, we first pass from $D$ to the \textbf{extended Dynkin diagram} $\tilde{D}$ of $\mf{g}$, which is the diagram obtained from $D$ by adding a node corresponding to $\tilde{\alpha}$ (where $\tilde{\alpha}$ is the lowest root of $\mf{g}$). This new node gets connected to $D$ using the usual Dynkin diagram edge rules. The resulting diagram $\tilde{D}$ is not the Dynkin diagram of a semisimple algebra, but if we remove a node $\alpha_k$ from $\tilde{D}$, we obtain the Dynkin diagram $D_{\alpha_k}$ of a semisimple algebra. When $\alpha_k = \tilde{\alpha}$, we simply recover the Dynkin diagram $D$ of $\mf{g}$. When $\alpha_k \neq \tilde{\alpha}$, it is not hard to check that the corresponding semisimple algebra, $\mf{g}(\alpha_k)$ (i.e.~the subalgebra of $\mf{g}$ generated by the root spaces $\mf{g}^{\pm \alpha}$ for $\alpha \in D_{\alpha_k}$), is a maximal-rank regular subalgebra of $\mf{g}$. For $\alpha_k \neq \tilde{\alpha}$, the process of passing from $\mf{g}$ to $\mf{g}(\alpha_k)$ that we have just described is called an \textbf{elementary operation}. For $\mf{g}$ a semisimple Lie algebra, an \textbf{elementary operation} is the process of applying the above process to a single simple factor of $\mf{g}$.  

\begin{remark} \label{rmk:index-1}
It is not hard to see from this analysis that the simple factors of maximal-rank regular subalgebras almost always have index 1. As a result, we will often omit the index of these subalgebras for convenience. The exceptions to this rule are the subalgebras of rank 1 coming from a short root in $B_n$, $G_2$, or $F_4$ (which have index 2, 3, and 2, respectively), and the subalgebra of type $A_2$ coming from the short roots in $F_4$ (which has index 2). We will follow the convention of denoting these exceptions with a tilde (e.g.~$\tilde{A_2}$ denotes the aforementioned index-2 subalgebra of $F_4$). 
\end{remark}

\begin{example}
Consider the extended Dynkin diagram $\tilde{D}$ for $\mf{g} = E_6$ shown on the left (where the node corresponding to $\tilde{\alpha}$ is indicated with a triangle). Crossing out the node corresponding to $\alpha_2$, we obtain the diagram $D_{\alpha_2}$ shown on the right, which has type $A_5 \oplus A_1$. This corresponds to the semisimple algebra $\mf{g}(\alpha_2) = A_5 \oplus A_1$ of $E_6$ (where both $A_5$ and $A_1$ have index 1 in $E_6$). This is a maximal-rank regular subalgebra of $E_6$, which in fact turns out to be maximal regular as well. 

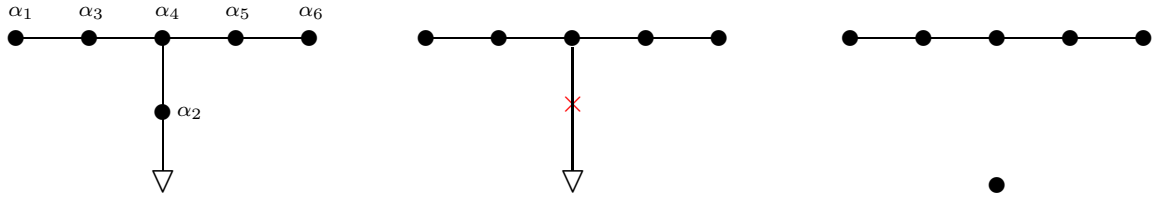
\setlength{\unitlength}{1.95cm}
\begin{picture}(3,2.0)(-2.5,-0.5)
  \multiput(-1.59,-.05)(.5,0){1}{\large $\triangledown$}
  \multiput(-1.5,.5)(.5,0){1}{\circle*{.10}}
  \multiput(-2.5,1.0)(.5,0){2}{\circle*{.10}}
  \multiput(-1.0,1.0)(.5,0){2}{\circle*{.10}}
  \multiput(-1.5,1.0)(.5,0){1}{\circle*{.10}}
  \multiput(-2.5, 1.0)(0,.5){1}{\line(1,0){1}}
  \multiput(-1.5,1.0)(0,.5){1}{\line(1,0){1}}
  \multiput(-1.5,0.1)(.5,0){1}{\line(0,1){.9}}
  \put(-1.4,.47){\tiny $\alpha_2$}
  \put(-1.55,1.15){\tiny $\alpha_4$}
  \put(-2.55,1.15){\tiny $\alpha_1$}
  \put(-2.07,1.15){\tiny $\alpha_3$}
  \put(-1.07,1.15){\tiny $\alpha_5$}
  \put(-0.57,1.15){\tiny $\alpha_6$}
\end{picture}
\setlength{\unitlength}{1.95cm}
\begin{picture}(3,2.0)(-2.2,-0.5)
  \multiput(-1.59,-.05)(.5,0){1}{\large $\triangledown$}
  \multiput(-1.58,.5)(.5,0){1}{\textcolor{red}{$\times$}}
  \multiput(-2.5,1.0)(.5,0){2}{\circle*{.10}}
  \multiput(-1.0,1.0)(.5,0){2}{\circle*{.10}}
  \multiput(-1.5,1.0)(.5,0){1}{\circle*{.10}}
  \multiput(-2.5, 1.0)(0,.5){1}{\line(1,0){1}}
  \multiput(-1.5,1.0)(0,.5){1}{\line(1,0){1}}
  \multiput(-1.5,0.1)(.5,0){1}{\line(0,1){.84}}
\end{picture}
\setlength{\unitlength}{1.95cm}
\begin{picture}(3,2.0)(-2,-0.5)
  \multiput(-1.5,0)(-.5,0){1}{\circle*{.10}}
  \multiput(-2.5,1.0)(.5,0){2}{\circle*{.10}}
  \multiput(-1.0,1.0)(.5,0){2}{\circle*{.10}}
  \multiput(-1.5,1.0)(.5,0){1}{\circle*{.10}}
  \multiput(-2.5, 1.0)(0,.5){1}{\line(1,0){1}}
  \multiput(-1.5,1.0)(0,.5){1}{\line(1,0){1}}
\end{picture}
\end{example}

\vspace{-2em}

Carrying out finitely many elementary operations like this, one can verify Table \ref{table:reg-max-rank}, which contains complete lists of the maximal-rank regular subalgebras of the classical simple and exceptional Lie algebras.     

\begin{table}[H]
\begin{center}
\begin{tabular}{| c | c l |} \hline
$\mf{g}$ & Subalgebra & \\ \hline
$B_n$ & $B_{m_0} \oplus D_{m_1} \oplus \cdots \oplus D_{m_r}$ & ($r \geq 1$, $m_1 \geq  \cdots \geq m_r >1$, $\sum_{i} m_i = n$) \\ 
$C_n$ & $C_{\ell_1} \oplus \cdots \oplus C_{\ell_r}$ & ($r \geq 2$, $\ell_1 \geq \cdots \geq \ell_r > 0$, $\sum_{i} \ell_i = n$) \\
$D_n$ & $D_{m_1} \oplus \cdots \oplus D_{m_r}$ & ($r \geq 2$, $m_1 \geq \cdots \geq m_r >1$, $\sum_{i} m_i = n$) \\ \hline
\end{tabular}

\vspace{.25cm}

\begin{tabular}{| c || c || c || c || c c c |} \hline
$G_2$ & $F_4$ & $E_6$ & $E_7$ & & $E_8$ & \\ \hline
$A_2$ & $B_4$ & $A_5 \oplus A_1$ & $D_6 \oplus A_1$ & $A_8$ & \arrvline & $E_7 \oplus A_1$ \\
$A_1 \oplus \tilde{A_1}$ & $A_3 \oplus \tilde{A_1}$ & $3A_2$ & $A_5^{''} \oplus A_2$ & $D_8$ & \arrvline & $D_6 \oplus 2A_1$ \\
 & $A_2 \oplus \tilde{A_2}$ & & $2A_3 \oplus A_1$ & $A_7^{'} \oplus A_1$ & \arrvline & $D_5 \oplus A_3$ \\
 & $C_3 \oplus A_1$ & & $A_7$ & $A_5 \oplus A_2 \oplus A_1$ & \arrvline & $2D_4$ \\
 & $D_4$ & & $D_4 \oplus 3A_1$ & $2A_4$ & \arrvline & $D_4 \oplus 4A_1$ \\
 & $B_2 \oplus 2A_1$ & & $7A_1$ & $4A_2$ & \arrvline & $2A_3 \oplus 2A_1$ \\
 & $4A_1$ & & & $E_6 \oplus A_2$ & \arrvline & $8A_1$ \\ \hline
\end{tabular}
\end{center}
\caption{Maximal-rank regular subalgebras of the simple Lie algebras (cf.~\cite[Tables 9, 11]{dynkin}). }
\label{table:reg-max-rank}
\end{table}

Note that there are two conjugacy classes of subalgebras with index 1 and type $A_5$ in $E_7$, as well as two conjugacy classes of subalgebras with index 1 and type $A_7$ in $E_8$. However, there is a \textit{unique} conjugacy class in $E_7$ with type $A_5 \oplus A_2$ and with both factors having index 1. Similarly, there is a \textit{unique} conjugacy class in $E_8$ with type $A_7 \oplus A_1$ and with both factors having index 1. In these cases, we still include ``primes" to clarify which conjugacy classes $A_5$ and $A_7$ belong to: $A_5^{''} \oplus A_2$ and $A_7^{'} \oplus A_1$ \cite[Table 25]{dynkin}.

Note also that for certain values of $n$, $m_i$, and $\ell_i$ as above, there are multiple conjugacy classes of subalgebras in $B_n$, $C_n$, or $D_n$ with index 1 and type $B_{m_0}$, $D_{m_i}$, or $C_{\ell_i}$. However, as in the exceptional case, the conjugacy class of each maximal-rank regular subalgebra is uniquely determined by type and index \cite[No.~17]{dynkin}. While ``primes" on simple factors are omitted from the classical portion of Table \ref{table:reg-max-rank}, we include them when describing the dual pairs in the classical simple Lie algebras of low rank in Section \ref{sec:lists-classical}.

\subsection{Maximal regular subalgebras of maximal rank}

Some of the maximal-rank regular subalgebras shown in Table \ref{table:reg-max-rank} are in fact maximal regular as well. By \cite[No.~17]{dynkin}, the regular subalgebras that are both maximal regular \textit{and} maximal-rank are all obtained by means of only one elementary operation. Moreover, all of the subalgebras obtained by a single elementary operation are maximal-rank regular subalgebras, and most of them are maximal. (Note that Dynkin in \cite[Theorem 5.5]{dynkin} mistakenly suggests that all maximal-rank regular subalgebras that are obtained by a single elementary operation are maximal. This seemingly led Dynkin to mistakenly include several non-maximal subalgebras in \cite[Table 12]{dynkin}; in particular, he mistakenly includes $A_3 \oplus \tilde{A_1}$ as maximal in $F_4$, includes $2A_3 \oplus A_1$ as maximal in $E_7$, and includes $A_7 \oplus A_1$, $A_5 \oplus A_2 \oplus A_1$, and $A_5 \oplus A_3$ as maximal in $E_8$.)\footnote{These errors in \cite[Table 12]{dynkin} were also noted in \cite[Section 3]{Tits}.} Fortunately, there is a straightforward way to check whether a maximal-rank regular subalgebra that is obtained via a single elementary operation is maximal regular as well:

\begin{lemma} \label{lemma:elementary-ops}
Let $\mf{g}$ be a simple Lie algebra, $\mf{h}$ a Cartan subalgebra of $\mf{g}$, and $\{ \alpha_1, \ldots , \alpha_n \}$ the set of simple roots corresponding to $(\mf{g}, \mf{h})$. Consider the extended Dynkin diagram of $\mf{g}$ with the node corresponding to $\alpha_k$ labeled with the coefficient $m_{k}$ of $\alpha_k$ in the highest root $\delta$, and the node corresponding to the lowest root $\tilde{\alpha}$ labeled 1. The regular subalgebras of $\mf{g}$ that are both maximal regular and maximal-rank are exactly the subalgebras $\mf{g}(\alpha_k)$ with $m_k$ prime. 
\end{lemma}

\begin{proof}
By \cite[No.~17]{dynkin}, the maximal regular subalgebras of maximal rank are all obtained by means of a single elementary operation, and any subalgebra obtained by elementary operations is a maximal-rank regular subalgebra. Therefore, it remains to show that for a simple root $\alpha_k$, $\mf{g}(\alpha_k)$ is maximal if and only if $m_{k}$ is prime.

To this end, fix $\alpha_k$ and let $\xi_{k}$ denote the corresponding fundamental coweight (so that $\xi_{k}(\alpha_k) = 1$ and $\xi_{k}(\alpha_j) = 0$ for $j \neq k$). Set 
$$t_{k} := \exp \left ( \frac{2 \pi i}{m_{k}} \xi_{k} \right ).$$
Then $t_{k}$ acts on $\mf{g}$ via the adjoint action, and hence defines an automorphism of $\mf{g}$ satisfying $t_{k}^{m_{k}} = 1$. We claim that $\mf{z}_{\mf{g}} (t_{k}) = \mf{g}(\alpha_k)$. It suffices to show that $\{ \tilde{\alpha} \} \cup \{  \alpha_j \}_{j \neq k}$ is the system of simple roots corresponding to $(\mf{z}_{\mf{g}}(t_{k}), \mf{h} )$.  

To start, note that $\alpha_j (t_{k}) = 1$ for any $j \neq k$. Additionally, note that $\alpha_k (t_{k}) = e^{ \frac{2 \pi i}{m_{k}} }$, which gives that $\tilde{\alpha}(t_{k}) = \alpha_k (t_{k})^{m_{k}} = e^{2\pi i} = 1$. From this, it's clear that $\mf{g}^{\pm \tilde{\alpha}}$ and $\mf{g}^{\pm \alpha_j}$ ($j \neq k$) are contained in $\mf{z}_{\mf{g}}(t_k)$. For the other direction, suppose we have some $\beta = \sum_{j=1}^{n} c_j \alpha_j$ in $\mf{z}_{\mf{g}}(t_k)$ (where $-m_j \leq c_j \leq m_j$ and where either $c_j \leq 0$ for all $j$ or $c_j \geq 0$ for all $j$). Then 
$$\beta (t_k) = e^{ \frac{2 \pi i c_k}{m_k} } = 1,$$
meaning $\frac{c_k}{m_k} \in \mathbb{Z}$. Since $-m_k \leq c_k \leq m_k$, we see that $c_k \in \{ -m_k, 0, m_k \}$. If $c_k = m_k$, then $c_j \geq 0$ for all $j$ and
$$\beta = \sum_j c_j \alpha_j = - \tilde{\alpha} - \sum_{j \neq k} (m_j - c_j) \alpha_j $$
is a non-positive linear combination of our proposed set of simple roots. Similarly, if $c_k = -m_{k}$, then $c_j \leq 0$ for all $j$ and
$$\beta = \sum_j c_j \alpha_j = \tilde{\alpha} + \sum_{j \neq k} (m_j + c_j) \alpha_j$$
is a non-negative linear combination of our proposed set of simple roots. Finally, if $c_k = 0$, there is nothing to show. It follows that $\mf{z}_{\mf{g}} (t_k) = \mf{g}(\alpha_k)$, as desired.

Now, suppose that $m_k$ is prime. We would like to show that $\mf{z}_{\mf{g}}(t_k) = \mf{g}(\alpha_k)$ is a maximal regular subalgebra. To this end, note that $t_k$ induces a $(\mathbb{Z}/m_k\mathbb{Z})$-grading on $\mf{g}$, where if $\beta (t_k) = e^{\frac{2\pi i r}{m_k}}$ (with $0 \leq r \leq m_k-1$), then $\mf{g}^{\beta}$ is in the $r$-th level of the grading. Let's call these levels $\mf{g}_0, \ldots , \mf{g}_{m_k-1}$. It's clear that $\mf{g}_0 = \mf{z}_{\mf{g}}(t_k) = \mf{g}(\alpha_k)$, and hence that each level is a module under the adjoint action of $\mf{g}(\alpha_k)$. Suppose we can show that each level is irreducible under this action. Then any subalgebra $\mf{r}$ of $\mf{g}$ properly containing $\mf{g}_0$ can be written as a sum of levels. Moreover, by the irreducibility of the levels, $[\mf{g}_i, \mf{g}_j] = \mf{g}_{i+j \pmod{m_k}}$. Since $m_k$ is prime, closure under Lie bracket implies that $\mf{r} = \mf{g}$. 

Thus, in the case of $m_k$ prime, we have reduced to showing that the levels of the $(\mathbb{Z}/m_k\mathbb{Z})$-grading on $\mf{g}$ induced by $t_k$ are irreducible as $\mf{g}(\alpha_k)$-modules. This is most straightforwardly proven by considering each case individually; we will not write out all of these cases, but instead will include an illustrative example. To this end, consider $\mf{g}(\alpha_5)$ in the case where $\mf{g}$ is of type $E_8$ (where $\alpha_5$ is as in Figure \ref{fig:root-labelings}). In this case, the extended Dynkin diagram with labels as described in the lemma statement is as follows:

\begin{center}
\begin{tikzcd}[row sep = small, column sep = small]
2 \arrow[dash]{r} & 4 \arrow[dash]{r} & 6 \arrow[dash]{r} \arrow[dash]{d} & 5 \arrow[dash]{r} & 4 \arrow[dash]{r} & 3 \arrow[dash]{r} & 2 \arrow[dash]{r} & 1 \\
 & & 3
\end{tikzcd}
\end{center}
From this, we see that $m_5 = 5$ and that the root vector $X_{\alpha_5}$ is the lowest weight for the irreducible representation $\wedge^2 \mathbb{C}^5 \otimes \mathbb{C}^5$ of $\mf{g}(\alpha_5)$. It follows that $\mf{g}_1$ has dimension at least $\binom{5}{2} \cdot 5 = 50$. Now, since $t_5$ represents the unique conjugacy class of elements of order 5 with centralizer $\mf{g}(\alpha_5)$ -- and since $t_5^2$, $t_5^3$, and $t_5^4$ have the same property -- we see that $t_5$, $t_5^2$, $t_5^3$, and $t_5^4$ are all conjugate. It follows that $\mf{g}_2$, $\mf{g}_3$, and $\mf{g}_4$ also have irreducible components of dimension at least 50. Additionally, $\mf{g}_0 = \mf{g}(\alpha_5) = \mf{sl}_5 \oplus \mf{sl}_5$ has dimension 48. Since $50 \cdot 4 + 48 = 248 = \dim \mf{g}$, this shows that these containments are in fact equalities, and hence that $\mf{g}_1$, $\mf{g}_2$, $\mf{g}_3$, and $\mf{g}_4$ are irreducible, as desired.

Finally, suppose that $m_k$ is not prime, and that $p \mid m_k$. It is clear that $\mf{g}(\alpha_k) = \mf{z}_{\mf{g}} (t_k) \subseteq \mf{z}_{\mf{g}} (t_k^p)$. Moreover, taking $\frac{m_k}{p}$ as the coefficient of $\alpha_k$, it is not difficult to construct roots of $\mf{z}_{\mf{g}}(t_k^p)$ that cannot be written as non-positive or non-negative linear combinations of $\{ \tilde{\alpha} \} \cup \{ \alpha_j \}_{j\neq k}$. It follows that $\mf{z}_{\mf{g}} (t_k) \subsetneq \mf{z}_{\mf{g}} (t_k^p)$. Additionally, $\mf{z}_{\mf{g}}(t_k^p)$ is clearly regular. Therefore, $\mf{g}(\alpha_k)$ is not a maximal regular subalgebra in this case. 
\end{proof}

\begin{example}
For example, here is the extended diagram for $E_7$, which has highest root $\delta = 2 \alpha_1 + 2 \alpha_2 + 3 \alpha_3 + 4 \alpha_4 + 3 \alpha_5 + 2 \alpha_6 + \alpha_7$:
\begin{center}
\begin{tikzcd}[row sep = small, column sep = small]
1 \arrow[dash]{r} & 2 \arrow[dash]{r} & 3 \arrow[dash]{r} & 4 \arrow[dash]{r} \arrow[dash]{d} & 3 \arrow[dash]{r} & 2 \arrow[dash]{r} & 1 \\
 & & & 2
\end{tikzcd}
\end{center}

In this case, $\mf{g}(\alpha_1) = \mf{g}(\alpha_6) = A_1 \oplus D_6$, which is maximal by Lemma \ref{lemma:elementary-ops} (since the relevant elementary operation involves removing a node with prime label 2). Similarly, Lemma \ref{lemma:elementary-ops} gives that $\mf{g}(\alpha_3) = \mf{g}(\alpha_5) = A_2 \oplus A_5^{''}$ is maximal. Removing either node with label 1, we get $\mf{g}(\tilde{\alpha}) = \mf{g}(\alpha_7) = E_7$, which is not a maximal regular subalgebra (since maximal regular subalgebras are defined to be proper). Finally, removing the node with label 4, we get $\mf{g}(\alpha_4) = A_1 \oplus A_3 \oplus A_3$, which Lemma \ref{lemma:elementary-ops} says is \textit{not} maximal regular. Indeed, since $A_3 \simeq D_3$, we see from Table \ref{table:maximal-regular} that $A_1 \oplus A_3 \oplus A_3 \simeq A_1 \oplus D_3 \oplus D_3 \subset A_1 \oplus D_6$. 
\end{example}

Computing $\mf{g}(\alpha_k)$ for simple $\mf{g}$ and simple roots $\alpha_k$ with prime label, we obtain the list of maximal regular subalgebras of maximal rank shown in Table \ref{table:maximal-regular}.

\begin{table}
\begin{tabular}{| c | l || c | l |} \hline
$\mf{g}$ & \hspace{1cm} Subalgebra & $\mf{g}$ & \hspace{2.3cm} Subalgebra \\ \hline
$B_n$ & $\mf{g} (\alpha_k)=D_k \oplus B_{n-k}$ & $E_6$ & $\mf{g}(\alpha_2) \simeq \mf{g}(\alpha_3) \simeq \mf{g}(\alpha_5) = A_1 \oplus A_5$ \\
 & ($k=2,3,\ldots,n$) & & $\mf{g}(\alpha_4) = 3A_2$ \\ \hline
$C_n$ & $\mf{g} (\alpha_k)=C_k \oplus C_{n-k}$ & $E_7$ & $\mf{g}(\alpha_1) \simeq \mf{g}(\alpha_6) = A_1 \oplus D_6$ \\ 
 & ($k=1,2,\ldots,n$) & & $\mf{g}(\alpha_2) = A_7$ \\ \cline{1-2}
$D_n$ & $\mf{g}(\alpha_k) = D_k \oplus D_{n-k} $ & & $\mf{g}(\alpha_3) \simeq \mf{g}(\alpha_5) = A_2 \oplus A_5^{''}$ \\  \cline{3-4}
 & ($k=2,3,\ldots,n-2$) & $E_8$ & $\mf{g}(\alpha_1) = D_8$ \\ \cline{1-2}
$F_4$ & $\mf{g}(\alpha_1) = A_1 \oplus C_3$ & & $\mf{g}(\alpha_2) = A_8$ \\ 
 & $\mf{g}(\alpha_2) = A_2 \oplus \tilde{A_2}$ & & $\mf{g}(\alpha_5) = 2A_4$ \\ 
 & $\mf{g}(\alpha_4) = B_4$ & & $\mf{g}(\alpha_7) = A_2 \oplus E_6$ \\ \cline{1-2}
$G_2$ & $\mf{g}(\alpha_1) = A_1 \oplus \tilde{A_1}$ & & $\mf{g}(\alpha_8) = A_1 \oplus E_7$ \\ 
  & $\mf{g}(\alpha_2) = A_2$ &  &  \\ \hline
\end{tabular}
\caption{Maximal regular subalgebras of maximal rank in simple Lie algebras (cf.~\cite[Table 12]{dynkin}). (In this table, the isomorphic subalgebras -- e.g.~$\mf{g}(\alpha_1) \simeq \mf{g}(\alpha_6)$ in $E_7$ -- are not only isomorphic but conjugate by an inner automorphism.)  } 
\label{table:maximal-regular}
\end{table}

\subsection{Dual pairs from maximal regular subalgebras of maximal rank}

It turns out that maximal-rank maximal regular subalgebras lead to an important class of dual pairs.  

\begin{prop}[{\cite[Proposition 5.15]{rubenthaler}}] \label{prop:maximal-reg-is-dp}
Let $\mf{g}$ be a simple Lie algebra. Let $\mf{g}_1$ and $\mf{g}_2$ be two semisimple subalgebras of $\mf{g}$ such that $\mf{g}_1 \oplus \mf{g}_2$ is a maximal regular subalgebra of maximal rank in $\mf{g}$. Then $(\mf{g}_1, \mf{g}_2 )$ is a dual pair in $\mf{g}$. 
\end{prop}

Note that the adjectives ``maximal-rank" and ``maximal" are both required here, although ``maximal-rank" was mistakenly omitted from the statement of \cite[Proposition 5.15]{rubenthaler}. The original statement of \cite[Proposition 5.15]{rubenthaler} would imply, for example, that there are dual pairs of the form $(A_k, A_{n-1-k})$ in $A_n$ (for $k=1,\ldots, n-2$) coming from the maximal regular subalgebras $A_k \oplus A_{n-1-k}$ of rank $n-1$ in $A_n$\footnote{These maximal regular subalgebras (along with other maximal regular subalgebras of non-maximal rank) are mistakenly omitted from \cite[Table 12]{dynkin}.}; however, it is not hard to check that these fail to be dual pairs. Similarly, removing the ``maximal" adjective would imply, for example, that $(A_1, 2A_3)$ is a dual pair in $E_7$. However, from Table \ref{table:maximal-regular}, it's clear that $\mf{z}_{E_7}(A_1) \supseteq D_6 \supsetneq 2A_3$. (In fact, Proposition \ref{prop:maximal-reg-is-dp} implies that $\mf{z}_{E_7}(A_1) = D_6$.) Therefore, by consulting Table \ref{table:maximal-regular}, we see that Proposition \ref{prop:maximal-reg-is-dp} gives us the dual pairs shown in Table \ref{table:max-reg-dps}.

\begin{table}
\begin{tabular}{| c | c |} \hline
$\mf{g}$ & Dual Pair \\ \hline 
$B_n$ & $(D_k, B_{n-k})$: $k=2,3,\ldots, n-1$ \\
 & $(A_1, A_1 \oplus B_{n-2})$ \\ \hline
$C_n$ & $(C_k, C_{n-k})$: $k=1,2, \ldots , n-1$ \\ \hline 
$D_n$ & $(D_k, D_{n-k})$: $k=2,3, \ldots , n-2$ \\
 & $(A_1, A_1 \oplus D_{n-2})$ \\ \hline
\end{tabular}

\vspace{.25cm}

\begin{tabular}{| c || c || c || c || c |} \hline
$G_2$ & $F_4$ & $E_6$ & $E_7$ & $E_8$ \\ \hline
$(A_1, \tilde{A_1})$ & $(A_1, C_3)$ & $(A_1, A_5)$ & $(A_1, D_6)$ & $(A_4, A_4)$ \\
 & $(A_2, \tilde{A_2})$ & $(A_2, 2A_2)$ & $(A_2, A_5^{''})$ & $(A_2, E_6)$ \\
 & & & & $(A_1, E_7)$ \\ \hline
\end{tabular}
\caption{Dual pairs coming from maximal regular subalgebras of maximal rank in simple Lie algebras.}
\label{table:max-reg-dps}
\end{table}

Even though the adjective ``maximal" is required for Proposition \ref{prop:maximal-reg-is-dp} to be correct, it \textit{can} happen that non-maximal maximal-rank regular subalgebras can lead to dual pairs. For example, we will later see that $(B_2, 2A_1)$ is a dual pair in $F_4$, that $(D_4, 3A_1)$ us a dual pair in $E_7$, and that $(A_5, A_2 \oplus A_1)$, $(D_6, 2A_1)$, $(D_5, A_3)$, $(D_4, D_4)$, and $(D_4 \oplus A_1, 3A_1)$ are dual pairs in $E_8$ (see Tables \ref{table:F4}, \ref{table:E7}, and \ref{table:E8}). We will later see that these dual pairs (along with the ones in Table \ref{table:max-reg-dps}) are examples of \textit{non-$S$-irreducible} dual pairs.  

Maximal-rank regular subalgebras will continue to play a crucial role in the classification of dual pairs throughout the remainder of the paper. In particular, it will turn out that a dual pair in a simple Lie algebra $\mf{g}$ is either \textit{$S$-irreducible} in $\mf{g}$ or \textit{$S$-irreducible} in a maximal-rank regular subalgebra of $\mf{g}$ (see Theorem \ref{thm:non-S-irred}).

\section{Admissible subalgebras} \label{sec:admissible} 

In this section, we introduce the notion of \textit{admissibility}, which will play a crucial role in the classification of \textit{$S$-irreducible} dual pairs. To this end, let $\mf{g}$ be a simple Lie algebra over $\mathbb{C}$, and let $\mf{h}$ be a Cartan subalgebra of $\mf{g}$. Let $\mathcal{R}$ denote the root system of $(\mf{g}, \mf{h})$. Let 
$$\Psi = \{ \alpha_1, \ldots , \alpha_n \}$$
be a basis (i.e.~set of simple roots) for $\mathcal{R}$. 

Let $\theta$ denote some subset of $\Psi$. Let $D$ denote the diagram obtained from the Dynkin diagram of $\Psi$ by circling the nodes corresponding to elements of $\Psi \setminus \theta$. We will say that $D$ is the \textbf{diagram associated with $(\Psi , \theta)$}. For each circled node $\gamma_i \in \Psi \setminus \theta$, let $D_i$ denote the maximal connected subdiagram of $D$ whose unique circled node is $\gamma_i$. Each $D_i$ is called an \textbf{irreducible component} of $D$. For example, consider the $D_9$ Dynkin diagram with $\Psi \setminus \theta = \{ \alpha_2, \alpha_5 \}$:

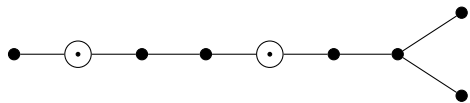
\vspace{.5cm}

\begin{center}

\end{center}

\newpage 

\noindent This diagram has the following irreducible components:
\vspace{.5em}
\begin{center}
\centering%
\end{center}

\begin{defn}[{\cite[Definition 3.1]{rubenthaler}}]
The couple $(\Psi, \theta)$ is said to be \textbf{admissible} if the irreducible components $D_i$ of the diagram $D$ associated with $(\Psi, \theta)$ all appear in Figure \ref{fig:table1}. In this case $D$ will also be called admissible. 
\end{defn}

\begin{remark}
The condition of admissibility corresponds to a certain associated prehomogeneous space being regular. See \cite[Section 2]{rubenthaler} for more details. 
\end{remark}

\begin{figure}
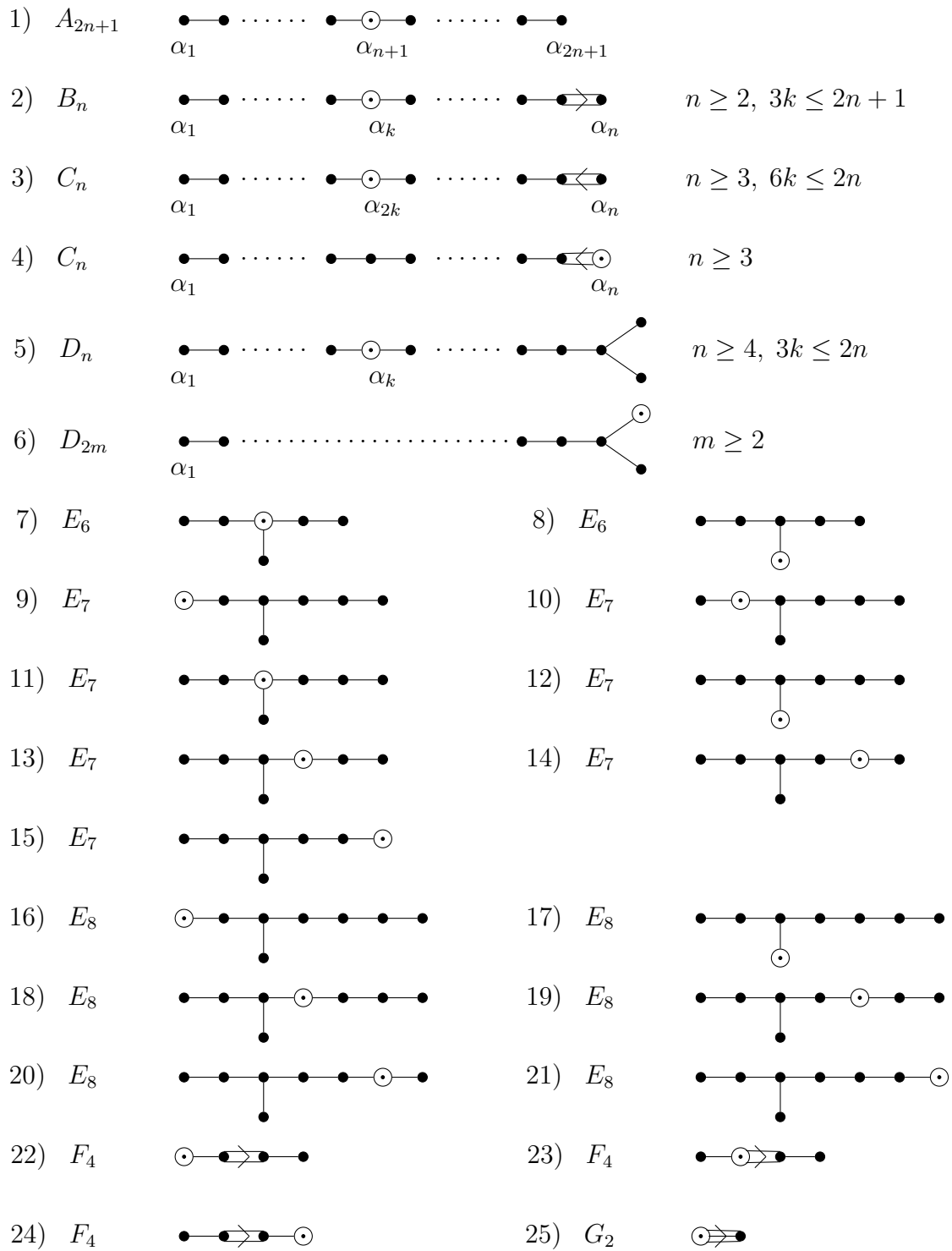
 
%
\caption{A complete list of irreducible admissible diagrams \cite[Table 1]{rubenthaler}.}
\label{fig:table1}
\end{figure}

For a subset $\theta$ of $\Psi$, set 
$$\mf{h}_{\theta} = \{ H \in \mf{h} \; \vert \; \alpha (H) = 0 \text{ for all } \alpha \in \theta \}.$$
Additionally, let $H_{\theta} \in \mf{h}_{\theta}$ be uniquely defined by
  \[
    \alpha (H_{\theta}) = \left\{\begin{array}{ll}
        0 & \text{if } \alpha \in \theta \\
        2 & \text{if } \alpha \in \Psi \setminus \theta .
        \end{array}\right.
  \]

The following theorem shows that for each irreducible component $D_i$ of the diagram $D$ associated with $(\Psi, \theta)$, the condition of appearing in Figure \ref{fig:table1} is precisely the same as the existence of a certain $\mf{sl}_2$-triplet.

\begin{thm}[{\cite[Theorem 2.3]{rubenthaler}}]
Suppose $\text{Card}(\Psi \setminus \theta) = 1$. Then the diagram $D$ associated with $(\Psi, \theta)$ appears in Figure \ref{fig:table1} if and only if $H_{\theta}$ is a defining vector (i.e.~if and only if there exist $X_{\theta}, Y_{\theta} \in \mf{g}$ such that $(X_{\theta},H_{\theta}, Y_{\theta})$ is an $\mf{sl}_2$-triplet).
\end{thm}

In this way, we see that if $(\Psi, \theta)$ is admissible, then we obtain a family of $\mf{sl}_2$-triplets, one for each irreducible component $D_i$ of the diagram $D$ associated with $(\Psi, \theta)$.

\begin{remark} \label{rmk:characteristic}
Note that for each irreducible component of a weighted diagram, the corresponding weighted diagram will have label 0 at every node except for the circled node, which will have label 2.
\end{remark}

\begin{thm}[{\cite[Theorem 3.3]{rubenthaler}}] \label{thm:g-tilde-theta}
If the pair $(\Psi, \theta)$ is admissible, then the family of $\mf{sl}_2$-triplets $(Y_i, H_i, X_i)$ generates a simple subalgebra $\tilde{\mf{g}_{\theta}}$ of $\mf{g}$ with Cartan subalgebra $\mf{h}_{\theta}$.
\end{thm}

\begin{example} \label{ex:A1E8-characteristic}
Let $\mf{g}$ be the complex Lie algebra of type $E_8$ with Cartan subalgebra $\mf{h} = \mathbb{C}^8$. By \cite[Plate VII]{bourbaki}, the root system corresponding to $(\mf{g},\mf{h})$ has the following simple roots:
\begin{align*}
\Psi= \Big \{ \alpha_1 &= \frac{1}{2} (\varepsilon_1 + \varepsilon_8) - \frac{1}{2} (\varepsilon_2 + \varepsilon_3 + \varepsilon_4 + \varepsilon_5 + \varepsilon_6 + \varepsilon_7 ), \\
\alpha_2 &= \varepsilon_1 + \varepsilon_2, \; \alpha_3 = \varepsilon_2 - \varepsilon_1, \; \alpha_4 = \varepsilon_3 - \varepsilon_2, \; \alpha_5 = \varepsilon_4 - \varepsilon_3, \\
\alpha_6 &= \varepsilon_5 - \varepsilon_4, \; \alpha_7 = \varepsilon_6 - \varepsilon_5, \; \alpha_8 = \varepsilon_7-\varepsilon_6 \Big \}.
\end{align*}
Let $\theta = \{ \alpha_2, \alpha_3, \alpha_4, \alpha_5, \alpha_6, \alpha_7, \alpha_8 \}$ so that $\Psi \setminus \theta = \{ \alpha_1 \}$. The diagram associated to $(\Psi , \theta)$ appears as \textit{16)} in Figure \ref{fig:table1}, and is consequently admissible. It is not hard to check that 
$$\mf{h}_{\theta} = \{ H \in \mf{h} \; \vert \; \alpha (H) = 0 \text{ for all } \alpha \in \theta \} = \left \{ (0,0,0,0,0,0,0,x) \; \vert \; x \in \mathbb{C} \right \}.$$
Additionally, $H_{\theta} \in \mf{h}_{\theta}$ is given by $H_{\theta} = (0,0,0,0,0,0,0,4)$. By Theorem \ref{thm:g-tilde-theta}, $\tilde{\mf{g}_{\theta}} = \langle X_{\theta}, H_{\theta}, Y_{\theta} \rangle \simeq \mf{sl}_2$ is a simple subalgebra of $\mf{g}$ with Cartan subalgebra $\mf{h}_{\theta}$ (and defining vector $H_{\theta}$). Moreover, as explained in Remark \ref{rmk:characteristic}, $\tilde{\mf{g}_{\theta}}$ has the following weighted diagram:
\vspace{.25em}
\begin{center}
\begin{tikzcd}[row sep = small, column sep = small]
2 \arrow[dash]{r} & 0 \arrow[dash]{r} & 0 \arrow[dash]{r} \arrow[dash]{d} & 0 \arrow[dash]{r} & 0 \arrow[dash]{r} & 0 \arrow[dash]{r} & 0 \\
 & & 0
\end{tikzcd}
\end{center}
\vspace{.25em}

Therefore, \cite[Table 20]{dynkin} gives that $\tilde{\mf{g}_{\theta}} = A_1^8$.
\end{example}

\section{\texorpdfstring{$S$}{S}-irreducible and admissible dual pairs} \label{sec:S-irred}

The goal of this section is to define the notion of an \textit{$S$-irreducible dual pair}, and to list all of the $S$-irreducible dual pairs of the exceptional Lie algebras. To start, we define the notion of an $S$-subalgebra, which was initially introduced in \cite[No.~23]{dynkin}. 

\begin{defn}[{\cite[Definition 5.9]{rubenthaler}}]
Let $\mf{g}$ be a semisimple Lie algebra. An \textbf{$S$-subalgebra} of $\mf{g}$ is a semisimple subalgebra that is \textit{not} contained in a proper regular subalgebra of $\mf{g}$. 
\end{defn}

With this notion of \textit{$S$-subalgebra}, we can define a notion of irreducibility for dual pairs:

\begin{defn}[{\cite[Definition 5.11]{rubenthaler}}] \label{def:5.11-S-irred}
Let $(\mf{a},\mf{b})$ be a semisimple dual pair of a semisimple algebra $\mf{g}$. We say that the pair $(\mf{a},\mf{b})$ is \textbf{$S$-irreducible} if the algebra $\mf{a} \oplus \mf{b}$ is an $S$-subalgebra.
\end{defn}

\begin{remark}
In \cite{HoweSeries}, Howe considers semisimple dual pairs in $\mf{sp}_{2n}$, and defines a notion of dual pair irreducibility in that setting. While the notion of dual pair irreducibility in Definition \ref{def:5.11-S-irred} a priori seems possibly different, it turns out that the two notions coincide in the setting of $\mf{sp}_{2n}$ \cite[Remark 5.13]{rubenthaler}.
\end{remark}

Figures \ref{fig:S-subalg-classical-1} and \ref{fig:S-subalg-classical-2} show all of the $S$-subalgebras of the classical simple Lie algebras up to rank 6. Similarly, Figure \ref{fig:S-subalgebras} shows all of the $S$-subalgebras of the exceptional Lie algebras, along with their inclusion relations. These figures, together with the following proposition, give us a method for easily identifying lots of $S$-irreducible dual pairs: 

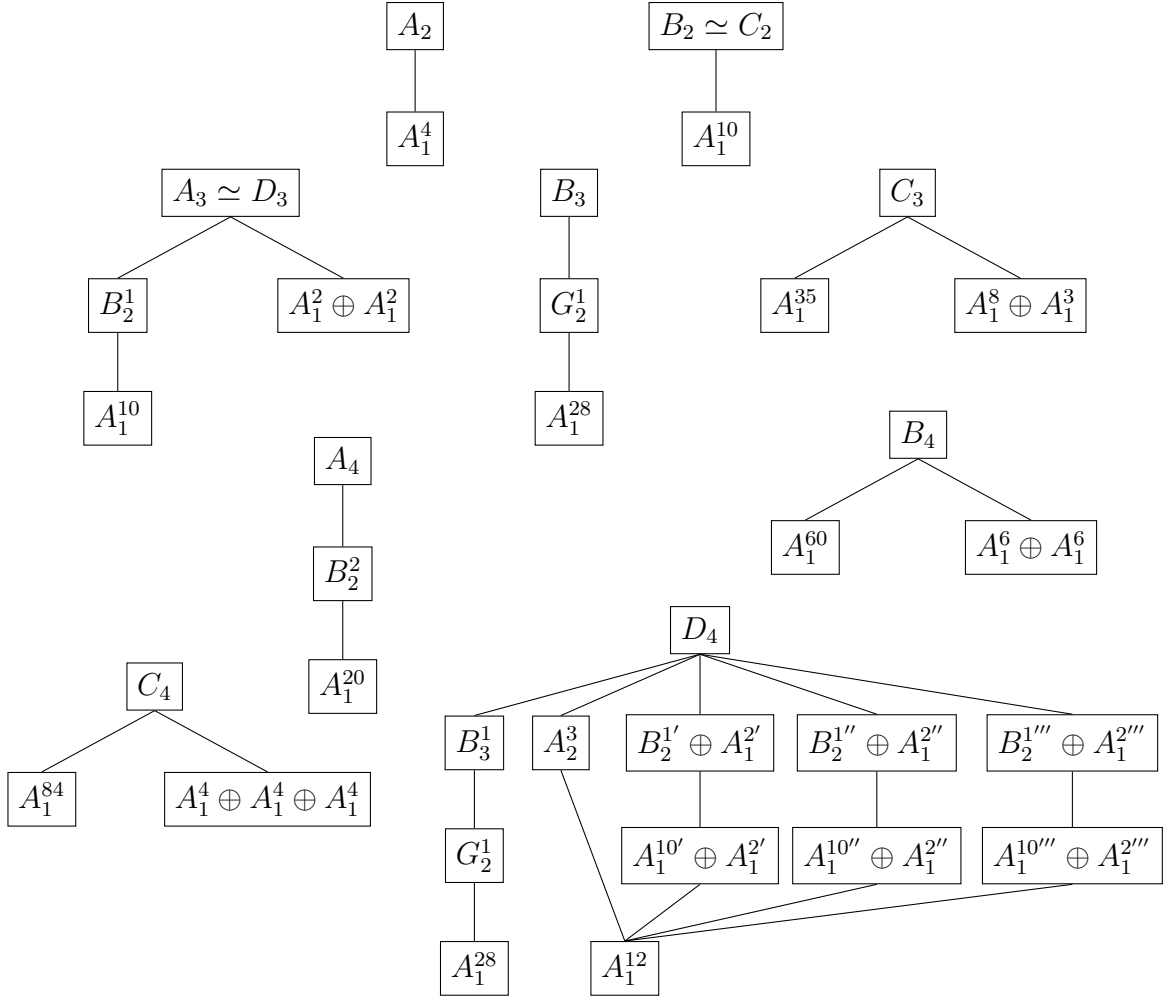
\begin{figure}
\begin{subfigure}{\textwidth}
        \centering\begin{tikzpicture}

\begin{scope}[xshift=-2cm]
    \node[draw, rectangle] (box1) at (0,0) {$A_2$};
    
    \node[draw, rectangle] (box2) at (0,-1.5) {$A_1^4$};
    
    \draw (box1.south) -- (box2.north);
\end{scope}

\begin{scope}[xshift=2cm]
    \node[draw, rectangle] (box1) at (0,0) {$B_2 \simeq C_2$};
    
    \node[draw, rectangle] (box2) at (0,-1.5) {$A_1^{10}$};
    
    \draw (box1.south) -- (box2.north);
\end{scope}

        \end{tikzpicture}
\end{subfigure}

\begin{subfigure}{\textwidth}
        \centering\begin{tikzpicture}

\begin{scope}[xshift=-4.5cm]
    \node (top) at (0,0) [draw, rectangle] {$A_3 \simeq D_3$};
    
    \node (left) at (-1.5,-1.5) [draw, rectangle] {$B_2^1$};
    
    \node (right) at (1.5,-1.5) [draw, rectangle] {$A_1^2 \oplus A_1^2$};
    
    \node (below2) at (-1.5,-3) [draw, rectangle] {$A_1^{10}$};
    
    \draw (top.south) -- (left.north);
    \draw (top.south) -- (right.north);
    \draw (left.south) -- (below2.north);
\end{scope}

\begin{scope}
    \node (box1) at (0,0) [draw, rectangle] {$B_3$};
    
    \node (box2) at (0,-1.5) [draw, rectangle] {$G_2^1$};
    
    \node (box3) at (0,-3) [draw, rectangle] {$A_1^{28}$};
    
    \draw (box1.south) -- (box2.north);
    \draw (box2.south) -- (box3.north);
\end{scope}

\begin{scope}[xshift=4.5cm]
    \node (top) at (0,0) [draw, rectangle] {$C_3$};
    
    \node (left) at (-1.5,-1.5) [draw, rectangle] {$A_1^{35}$};
    
    \node (right) at (1.5,-1.5) [draw, rectangle] {$A_1^8 \oplus A_1^3$};
    
    \draw (top.south) -- (left.north);
    \draw (top.south) -- (right.north);
\end{scope}

        \end{tikzpicture}
\end{subfigure}

\vspace{-.5cm}

\begin{subfigure}{\textwidth}
        \centering\begin{tikzpicture}

\begin{scope}[xshift=-3.5cm]
    \node (box1) at (0,0) [draw, rectangle] {$A_4$};
    \node (box2) at (0,-1.5) [draw, rectangle] {$B_2^2$};
    \node (box3) at (0,-3) [draw, rectangle] {$A_1^{20}$};
    \draw (box1.south) -- (box2.north);
    \draw (box2.south) -- (box3.north);
\end{scope}

\begin{scope}[xshift=4.15cm,yshift=.35cm]
    \node (top) at (0,0) [draw, rectangle] {$B_4$};
    
    \node (left) at (-1.5,-1.5) [draw, rectangle] {$A_1^{60}$};
    
    \node (right) at (1.5,-1.5) [draw, rectangle] {$A_1^6 \oplus A_1^6$};
    
    \draw (top.south) -- (left.north);
    \draw (top.south) -- (right.north);
\end{scope}

\begin{scope}[xshift=-6cm,yshift=-3cm]
    \node (top) at (0,0) [draw, rectangle] {$C_4$};
    
    \node (left) at (-1.5,-1.5) [draw, rectangle] {$A_1^{84}$};
    
    \node (right) at (1.5,-1.5) [draw, rectangle] {$A_1^4 \oplus A_1^4 \oplus A_1^4$};
    
    \draw (top.south) -- (left.north);
    \draw (top.south) -- (right.north);
\end{scope}

\begin{scope}[xshift=-.25cm,yshift=-2.25cm]
    \node (top) at (1.5,0) [draw, rectangle] {$D_4$};
    
    \node (left) at (-1.5,-1.5) [draw, rectangle] {$B_3^1$};

    \node (middle) at (-.35,-1.5) [draw, rectangle] {$A_2^3$};
    
    \node (right) at (1.5,-1.5) [draw, rectangle] {$B_2^{1'} \oplus A_1^{2'}$};
    
    \node (rightmost) at (3.85,-1.5) [draw, rectangle] {$B_2^{1''} \oplus A_1^{2''}$};
    
	\node (rrmost) at (6.45,-1.5) [draw, rectangle] {$B_2^{1'''} \oplus A_1^{2'''}$};    
    
	\node (rrmost-1) at (6.45,-3) [draw, rectangle] {$A_1^{10'''} \oplus A_1^{2'''}$};    
    
    \node (belowLeft) at (-1.5,-3) [draw, rectangle] {$G_2^1$};
    
    \node (belowBelowLeft) at (-1.5,-4.5) [draw, rectangle] {$A_1^{28}$};
    
    \node (belowRight) at (1.5,-3) [draw, rectangle] {$A_1^{10'} \oplus A_1^{2'}$};

    \node (belowmiddle) at (.5,-4.5) [draw, rectangle] {$A_1^{12}$};
	
	\node (belowRightmost) at (3.85,-3) [draw, rectangle] {$A_1^{10''} \oplus A_1^{2''}$};    
    
    \draw (top.south) -- (left.north);
    \draw (top.south) -- (right.north);
    \draw (top.south) -- (rightmost.north);

    \draw (top.south) -- (middle.north);
    \draw (middle.south) -- (belowmiddle.north);
    \draw (belowRight.south) -- (belowmiddle.north);
    \draw (rightmost.south) -- (belowRightmost.north);
    
    \draw (top.south) -- (rrmost.north);
    \draw (rrmost.south) -- (rrmost-1.north);
    \draw (rrmost-1.south) --(belowmiddle.north);
    
    \draw (left.south) -- (belowLeft.north);
    
    \draw (belowLeft.south) -- (belowBelowLeft.north);
    
    \draw (right.south) -- (belowRight.north);
    
    \draw (belowRightmost.south) -- (belowmiddle.north);
\end{scope}

        \end{tikzpicture}
\end{subfigure}

\caption{Inclusion relations among the $S$-subalgebras of the classical simple Lie algebras up to rank 4 (cf.~\cite[Table XII]{S-subalgebras}).}
\label{fig:S-subalg-classical-1}
\end{figure}

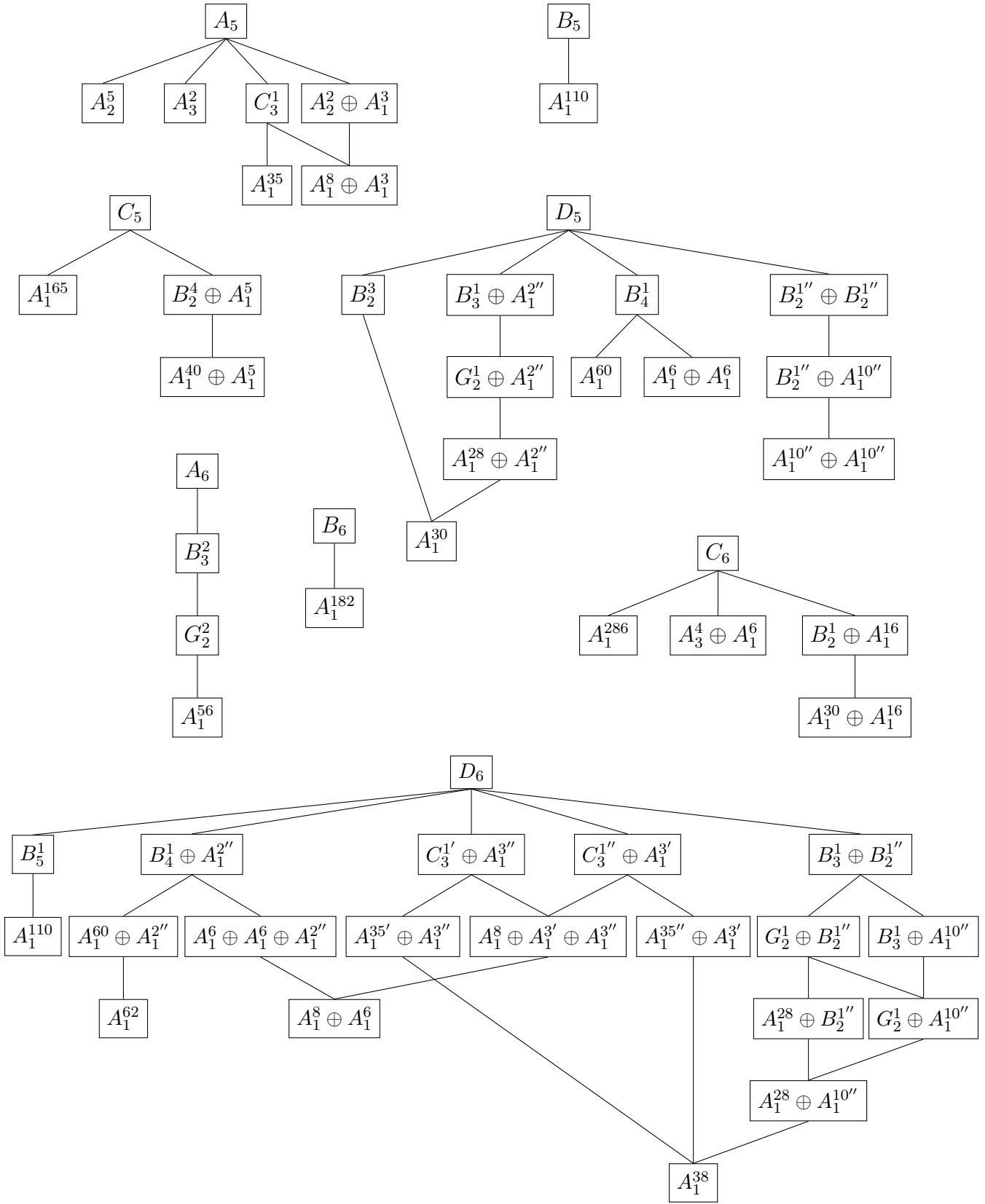
\begin{figure}
\begin{subfigure}{\textwidth}
\centering\begin{tikzpicture}

\begin{scope}[xshift=-2cm]
    \node[draw, rectangle] (top) at (-.25,0) {$A_5$};
    
    \foreach \i/\label in {1/$A_2^5$, 2/$A_3^2$, 3/$C_3^1$, 4/$A_2^2 \oplus A_1^3$} {
        \node[draw, rectangle] (box\i) at (1.5*\i-4,-1.5) {\label};
        \draw (top.south) -- (box\i.north);
    }

    \node[draw, rectangle] (box5) at (.5,-3) {$A_1^{35}$};
    \draw (box3.south) -- (box5.north);
    
    \node[draw, rectangle] (box6) at (2,-3) {$A_1^8 \oplus A_1^3$};
    \draw (box3.south) -- (box6.north);
    \draw (box4.south) -- (box6.north);
\end{scope}

\begin{scope}[xshift=4cm]
    \node[draw, rectangle] (box1) at (0,0) {$B_5$};
    
    \node[draw, rectangle] (box2) at (0,-1.5) {$A_1^{110}$};
    
    \draw (box1.south) -- (box2.north);
\end{scope}

\begin{scope}[xshift=-4cm,yshift=-3.5cm]
    \node (top) at (0,0) [draw, rectangle] {$C_5$};
    
    \node (left) at (-1.5,-1.5) [draw, rectangle] {$A_1^{165}$};
    
    \node (right) at (1.5,-1.5) [draw, rectangle] {$B_2^4 \oplus A_1^5$};
    
    \node (belowRight) at (1.5,-3) [draw, rectangle] {$A_1^{40} \oplus A_1^5$};
    
    \draw (top.south) -- (left.north);
    \draw (top.south) -- (right.north);
    
    \draw (right.south) -- (belowRight.north);
\end{scope}

\begin{scope}[xshift=4cm,yshift=-3.5cm]
    \node[draw, rectangle] (top) at (0,0) {$D_5$};
    
    \foreach \i/\label in {1/$B_2^3$, 2/$B_3^1 \oplus A_1^{2''}$, 3/$B_4^1$} {
        \node[draw, rectangle] (box\i) at (2.5*\i-6.25,-1.5) {\label};
        \draw (top.south) -- (box\i.north);
    }

 \node[draw, rectangle] (box4) at (4.75,-1.5) {$B_2^{1''} \oplus B_2^{1''}$};
    \draw (top.south) -- (box4.north);

 \node[draw, rectangle] (box9) at (4.75,-3) {$B_2^{1''} \oplus A_1^{10''}$};
    \draw (box4.south) -- (box9.north);

     \node[draw, rectangle] (box10) at (4.75,-4.5) {$A_1^{10''} \oplus A_1^{10''}$};
    \draw (box9.south) -- (box10.north);

    \node[draw, rectangle] (box5) at (-1.25,-3) {$G_2^{1} \oplus A_1^{2''}$};
    \draw (box2.south) -- (box5.north);

    \node[draw, rectangle] (box7) at (-1.25,-4.5) {$A_1^{28} \oplus A_1^{2''}$};
    \draw (box5.south) -- (box7.north);

    \node[draw, rectangle] (box6) at (.5,-3) {$A_1^{60}$};
    \draw (box3.south) -- (box6.north);

    \node[draw, rectangle] (box8) at (2.25,-3) {$A_1^{6} \oplus A_1^6$};
    \draw (box3.south) -- (box8.north);

    \node[draw, rectangle] (box11) at (-2.5,-6) {$A_1^{30} $};
    \draw (box1.south) -- (box11.north);
    \draw (box7.south) -- (box11.north);
\end{scope}
\end{tikzpicture}
\end{subfigure}

\vspace{-2cm}

\begin{subfigure}{\textwidth}
        \centering\begin{tikzpicture}
\begin{scope}[xshift=-5cm,yshift=1.5cm]
    \node (box1) at (0,0) [draw, rectangle] {$A_6$};
    \node (box2) at (0,-1.5) [draw, rectangle] {$B_3^2$};
    \node (box3) at (0,-3) [draw, rectangle] {$G_2^{2}$};

    \node (box4) at (0,-4.5) [draw, rectangle] {$A_1^{56}$};
    \draw (box1.south) -- (box2.north);
    \draw (box2.south) -- (box3.north);
    \draw (box3.south) -- (box4.north);
\end{scope}

\begin{scope}[xshift=-2.5cm,yshift=.5cm]
   
    \node (box1) at (0,0) [draw, rectangle] {$B_6$};
    
    \node (box2) at (0,-1.5) [draw, rectangle] {$A_1^{182}$};
    
    \draw (box1.south) -- (box2.north);
\end{scope}

\begin{scope}[xshift=4.5cm]
    \node (top) at (0,0) [draw, rectangle] {$C_6$};
    \node (left) at (-2,-1.5) [draw, rectangle] {$A_1^{286}$};
    \node (middle) at (0,-1.5) [draw, rectangle] {$A_3^4 \oplus A_1^6$};
    \node (right) at (2.5,-1.5) [draw, rectangle] {$B_2^1 \oplus A_1^{16}$};
    \node (belowRight) at (2.5,-3) [draw, rectangle] {$A_1^{30} \oplus A_1^{16}$};
    \draw (top.south) -- (left.north);
    \draw (top.south) -- (right.north);
    \draw (top.south) -- (middle.north);
    \draw (right.south) -- (belowRight.north);
\end{scope}

\begin{scope}[xshift=0cm,yshift=-4cm]
    \node (top) at (0,0) [draw, rectangle] {\small $D_6$};
    
    \node (box1) at (-8,-1.5) [draw, rectangle] {\small $B_5^{1}$};
    \node (box1-1) at (-8,-3) [draw, rectangle] {\small $A_1^{110}$};
    \draw (box1.south) -- (box1-1.north);

    \node (box2) at (-5.1,-1.5) [draw, rectangle] {\small $B_4^1 \oplus A_1^{2''}$};
    \node (box2-1) at (-6.35,-3) [draw, rectangle] {\small $A_1^{60} \oplus A_1^{2''}$};
    \node (box2-1-1) at (-6.35,-4.5) [draw, rectangle] {\small $A_1^{62}$};
    \node (box2-2) at (-3.85,-3) [draw, rectangle] {\small $A_1^{6} \oplus A_1^6 \oplus A_1^{2''}$};
    \node (box2-2-1) at (-2.5,-4.5) [draw, rectangle] {\small $A_1^8 \oplus A_1^6$};
    \draw (box2.south) -- (box2-1.north);
    \draw (box2-1.south) -- (box2-1-1.north);
    \draw (box2.south) -- (box2-2.north);
    \draw (box2-2.south) -- (box2-2-1.north);

    \node (newbox) at (0,-1.5) [draw, rectangle] {\small $C_3^{1'} \oplus A_1^{3''}$};
    \node (newbox-1) at (-1.25,-3) [draw, rectangle] {\small $A_1^{35'} \oplus A_1^{3''}$};

    \node (box3) at (2.85,-1.5) [draw, rectangle] {\small $C_3^{1''} \oplus A_1^{3'}$};
    \node (box3-1) at (1.4,-3) [draw, rectangle] {\small $A_1^{8} \oplus A_1^{3'} \oplus A_1^{3''}$};
    \node (box3-2) at (4.05,-3) [draw, rectangle] {\small $A_1^{35''} \oplus A_1^{3'} $};
    \draw (box3.south) -- (box3-1.north);
    \draw (box3-1.south) -- (box2-2-1.north);
    \draw (box3.south) -- (box3-2.north);  

    \node (box4) at (7.1,-1.5) [draw, rectangle] {\small $B_3^1 \oplus B_2^{1''}$};
    \node (box4-1) at (6.15,-3) [draw, rectangle] {\small $G_2^1 \oplus B_2^{1''}$}; 
    \node (box4-2) at (8.25,-3) [draw, rectangle] {\small $B_3^1 \oplus A_1^{10''}$}; 
    \node (box4-1-1) at (6.15,-4.5) [draw, rectangle] {\small $A_1^{28} \oplus B_2^{1''}$}; 
    \node (box4-1-2) at (8.25,-4.5) [draw, rectangle] {\small $G_2^{1} \oplus A_1^{10''}$}; 
    \node (box4-1-1-1) at (6.15,-6) [draw, rectangle] {\small $A_1^{28} \oplus A_1^{10''}$}; 
    \node (box4-1-1-1-1) at (4.05,-7.5) [draw, rectangle] {\small $A_1^{38}$};

    \draw (box4.south) -- (box4-1.north);
    \draw (box4.south) -- (box4-2.north);  
    \draw (box4-1.south) -- (box4-1-1.north);  
    \draw (box4-1.south) -- (box4-1-2.north); 
    \draw (box4-2.south) -- (box4-1-2.north); 
    \draw (box4-1-1.south) -- (box4-1-1-1.north); 
    \draw (box4-1-2.south) -- (box4-1-1-1.north); 
    \draw (box4-1-1-1.south) -- (box4-1-1-1-1.north); 
    \draw (box3-2.south) -- (box4-1-1-1-1.north); 
    
    \draw (top.south) -- (newbox.north);
    \draw (newbox.south) -- (newbox-1.north);
    \draw (newbox.south) -- (box3-1.north);
    \draw (newbox-1.south) -- (box4-1-1-1-1.north);

    \draw (top.south) -- (box1.north);
    \draw (top.south) -- (box2.north);
    \draw (top.south) -- (box3.north);
    \draw (top.south) -- (box4.north);
\end{scope}
\end{tikzpicture}
\end{subfigure}
\caption{Inclusion relations among the $S$-subalgebras of the classical simple Lie algebras of ranks 5 and 6 (cf.~\cite[Table XII]{S-subalgebras}).}
\label{fig:S-subalg-classical-2}
\end{figure}

\begin{prop}[{\cite[Proposition 5.19]{rubenthaler}}] \label{prop:max-S-is-dp}
Suppose $\mf{g}$ is simple. Let $\mf{a}$ and $\mf{b}$ be semisimple subalgebras of $\mf{g}$ such that $\mf{a} \oplus \mf{b}$ is a maximal $S$-subalgebra. Then $(\mf{a},\mf{b})$ is an $S$-irreducible dual pair in $\mf{g}$.
\end{prop}

It is important to note that the converse of Proposition \ref{prop:max-S-is-dp} does not hold (i.e.~not all $S$-irreducible dual pairs come from maximal $S$-subalgebras). 

\begin{example}
As we saw in Example \ref{ex:A1E8-characteristic}, $A_1^8$ is an admissible subalgebra of $E_8$ with $\Psi \setminus \theta = \{ \alpha_1 \}$. By \cite[Example 25, p.~146]{S-K}, we can deduce that $\mf{z}_{E_8} (A_1^8) = G_2^1 \oplus G_2^1$ (cf.~\cite[Section 6.10]{rubenthaler}). From Figure \ref{fig:S-subalgebras}, we see that $A_1^8 \oplus G_2^1 \oplus G_2^1$ is a (non-maxial) $S$-subalgebra in $E_8$. Additionally, Figure \ref{fig:S-subalgebras} and Proposition \ref{prop:max-S-is-dp} imply that $(A_1^8, G_2^1)$ is a dual pair in $F_4$ and that $(G_2^1, F_4^1)$ is a dual pair in $E_8$. Therefore, 
$$\mf{z}_{E_8}(G_2^1 \oplus G_2^1) = \mf{z}_{F_4}(G_2^1) = A_1^8.$$
In this way, we see that $(A_1^8, G_2^1 \oplus G_2^1)$ is an $S$-irreducible dual pair in $E_8$.
\end{example}

As we will see in Theorem \ref{thm:8.1}, this is where the notion of admissibility becomes useful.

\newpage

\begin{thm}[{\cite[Theorem 8.1]{rubenthaler}}] \label{thm:8.1}
Let $\mf{g}$ be a simple Lie algebra and let $(\mf{a}, \mf{b})$ be an $S$-irreducible dual pair in $\mf{g}$. Then with the exception of the following cases, at least one of $\mf{a}$ and $\mf{b}$ is admissible:
\begin{enumerate}[label = (\roman*)]
\item $\mf{g}$ is of type $D_n$ and $(\mf{a},\mf{b})$ is of type $(D_k, D_{\ell})$, where $n=2k\ell$ and $k,\ell \geq 2$. 
\item $\mf{g}$ is of type $D_n$ and $(\mf{a}, \mf{b})$ is of type $(B_{p_1} \oplus B_{q_1}, B_{p_2} \oplus B_{q_2})$, where $(B_{p_1}, B_{p_2})$ is an $S$-irreducible dual pair in $B_p$ (with $2p+1 = (2p_1 + 1)(2p_2+1)$), where $(B_{q_1}, B_{q_2})$ is an $S$-irreducible dual pair in $B_q$ (with $2q+1 = (2q_1 + 1)(2q_2+1)$), and where $n = p + q + 1$. 
\end{enumerate}
\end{thm}

Therefore, to generate a full list of $S$-irreducible dual pairs in a given simple Lie algebra $\mf{g}$ of any type other than $D_n$, it suffices to consider all admissible diagrams associated with an appropriate $(\Psi, \theta)$, and for each such diagram compute the corresponding subalgebra $\tilde{\mf{g}_{\theta}}$. Finally, compute $\mf{z}_{\mf{g}} ( \tilde{\mf{g}_{\theta}} )$ and check whether $\tilde{\mf{g}_{\theta}} \oplus \mf{z}_{\mf{g}} ( \tilde{\mf{g}_{\theta}} )$ appears in Figure \ref{fig:S-subalgebras}, \ref{fig:S-subalg-classical-1}, or \ref{fig:S-subalg-classical-2} as an $S$-subalgebra of $\mf{g}$. For $\mf{g}$ of type $D_n$, the $S$-irreducibles are those obtained in this way in addition to those of the forms (i) and (ii) in Theorem \ref{thm:8.1} (see \cite[Theorem 7.3.2]{rubenthaler} and \cite[Proposition 7.3.3]{rubenthaler}). (Note that some admissible subalgebras are members of dual pairs that are \textit{not} $S$-irreducible. These dual pairs -- along with the $S$-irreducible dual pairs with an admissible factor -- will be called \textbf{admissible dual pairs}). 

Note that the ``compute $\mf{z}_{\mf{g}}(\tilde{\mf{g}_{\theta}})$" step in this process is very nontrivial. Fortunately, the type of many of these centralizers can be deduced from \cite{S-K}. Using this source, Rubenthaler carries out this process to classify the admissible dual pairs in the classical simple Lie algebras (see \cite[Sections 6.3--6.7]{rubenthaler}) and in the exceptional Lie algebras (see \cite[Sections 6.8--6.12]{rubenthaler}). These dual pairs will be included later on in the paper (see Section \ref{sec:lists-classical} for the classical simple Lie algebras and Section \ref{sec:lists} for the exceptional Lie algebras).

\section{Some of the dual pairs in the classical simple Lie algebras} \label{sec:lists-classical}

Using what we have done so far, we can classify all of the dual pairs in simple Lie algebras coming from maximal regular subalgebras of maximal rank or having an admissible factor (which includes \textit{all} $S$-irreducible dual pairs for algebras of types other than $D_n$). 

Recall that the dual pairs coming from maximal regular subalgebras of maximal rank are summarized in Table \ref{table:max-reg-dps}. In this section, we follow Rubenthaler in classifying the admissible dual pairs in the classical simple Lie algebras. Additionally, we explicitly write out these dual pairs (including index and ``primes") for the classical simple Lie algebras up to rank 6. These explicit dual pairs will be useful later on for our classification of dual pairs in the exceptional Lie algebras. Note that in the below tables and in what follows, whenever we encounter subalgebras of types $B_1$, $C_1$, $D_1$, $C_2$, or $D_3$, we will instead write $A_1$, $B_2$, or $A_3$ so as to make these table entries more easily findable in Appendices A and B.

While we follow Rubenthaler's classification approach in this section for notational consistency, it is worth noting that the dual pairs in the classical simple Lie algebras can also be understood more comprehensively using elementary methods. See \cite{myarxiv} for more details.

\subsection{Type \texorpdfstring{$A_n$}{A}} \label{subsec:A_n}

Recall that the complex Lie algebra of type $A_n$ has no maximal-rank regular subalgebras. As we will soon see from Theorem \ref{thm:non-S-irred}, this means that all semisimple dual pairs in $A_n$ are $S$-irreducible. By Theorem \ref{thm:8.1}, any such dual pair has at least one factor that is an admissible subalgebra of $A_n$. Moreover, it turns out that all admissible subalgebras of $A_n$ are $S$-irreducible. Consequently, we obtain exactly the dual pairs outlined below (with the first factor having the following admissible diagram), where $n=(k+1)p-1$, $k \geq 1$, $n \geq 1$, and $p \geq 2$:

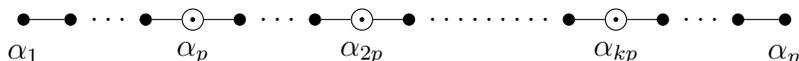
\vspace{1em}

\begin{center}
\begin{tikzpicture}[scale=1.25]
    \node (dot1) at (0,0) [draw, circle, inner sep=1.5pt, fill=black] {};
    \node (dot2) at (.5,0) [draw, circle, inner sep=1.5pt, fill=black] {};
    \draw (dot1) -- (dot2);
    \node[below] at (0,-.15) {$\alpha_1$};
    
    \node (ellipsis) at (.9,0) {$\ldots$};
    
    \node (dot3) at (1.3,0) [draw, circle, inner sep=1.5pt, fill=black] {};
    \node (dot4) at (1.8,0) [draw, circle, inner sep=.5pt, fill=black] {};
    \node[draw, circle, inner sep=0pt, minimum size=8pt] (circle4) at (1.8,0) {};
    \node (dot5) at (2.3,0) [draw, circle, inner sep=1.5pt, fill=black] {};
    \draw (dot3) -- (circle4.west);
    \draw (circle4.east) -- (dot5);
    \node[below] at (circle4.south) {$\alpha_p$};

    \node (ellipsis2) at (2.7,0) {$\ldots$};

    \node (dot6) at (3.1,0) [draw, circle, inner sep=1.5pt, fill=black] {};
    \node (dot7) at (3.6,0) [draw, circle, inner sep=.5pt, fill=black] {};
    \node[draw, circle, inner sep=0pt, minimum size=8pt] (circle7) at (3.6,0) {};
    \node (dot8) at (4.1,0) [draw, circle, inner sep=1.5pt, fill=black] {};
    \draw (dot6) -- (circle7.west);
    \draw (circle7.east) -- (dot8);
    \node[below] at (circle7.south) {$\alpha_{2p}$};

    \node (ellipsis3) at (4.5,0) {$\ldots$};
    \node (ellipsis4) at (4.95,0) {$\ldots$};
    \node (ellipsis5) at (5.4,0) {$\ldots$};

    \node (dot9) at (5.8,0) [draw, circle, inner sep=1.5pt, fill=black] {};
    \node (dot10) at (6.3,0) [draw, circle, inner sep=.5pt, fill=black] {};
    \node[draw, circle, inner sep=0pt, minimum size=8pt] (circle10) at (6.3,0) {};
    \node (dot11) at (6.8,0) [draw, circle, inner sep=1.5pt, fill=black] {};
    \draw (dot9) -- (circle10.west);
    \draw (circle10.east) -- (dot11);
    \node[below] at (circle10.south) {$\alpha_{kp}$};

    \node (ellipsis3) at (7.2,0) {$\ldots$};

    \node (dot12) at (7.6,0) [draw, circle, inner sep=1.5pt, fill=black] {};
    \node (dot13) at (8.1,0) [draw, circle, inner sep=1.5pt, fill=black] {};
    \draw (dot12) -- (dot13);
    \node[below] at (8.1,-.15) {$\alpha_n$};
\end{tikzpicture}
\end{center}

\begin{itemize}
\item \underline{$S$-irreducible:} 
\begin{itemize}
\item $(A_k, A_{p-1})$ 
\end{itemize}
\end{itemize}

The index of these subalgebras can be computed by hand or can be found in \cite[Table XII]{S-subalgebras} for $n \leq 6$. Putting this all together, we get that the following is a complete list of dual pairs in $A_n$ (for $3 \leq n \leq 6$).

\begin{center}
\begin{table}[H]
\begin{tabular}{| c | c || c | c || c | c || c | c |} \hline
Rank & Dual Pair & Rank & Dual Pair & Rank & Dual Pair & Rank & Dual Pair \\ \hline
3 & $(A_1^2, A_1^2)$ & 4 & N/A & 5 & $(A_1^3, A_2^2)$ & 6 & N/A  \\ \hline
\end{tabular}
\caption{A complete list of semisimple dual pairs in $A_n$ (for $3 \leq n \leq 6$).}
\label{table:dps-in-An}
\end{table}
\end{center}

\subsection{Type \texorpdfstring{$B_n$}{B}}

Recall that the complex Lie algebra of type $B_n$ has $(A_1, A_1 \oplus B_{n-2})$ and $(B_{n-k}, D_k)$ ($k=2,3,\ldots , n$) as dual pairs arising from maximal regular subalgebras of maximal rank (see Table \ref{table:max-reg-dps}). By Theorem \ref{thm:8.1}, any $S$-irreducible dual pair in $B_n$ has at least one factor that is an admissible subalgebra of $B_n$. Additionally, it turns out that \textit{some} of the non-$S$-irreducible dual pairs in $B_n$ have a factor that is an admissible subalgebra of $B_n$. In particular, as we will see below, the $B_{n-k}$ factor of the dual pairs $(B_{n-k}, D_k)$ ($k=2,3,\ldots , n$) is an admissible subalgebra of $B_n$. In total, for type $B_n$ (with $n \geq 4$), the admissible dual pairs are outlined below (with the first factor having the following admissible diagram), where $n \geq kp$, $(2k+1)p \leq 2n+1$, $p \geq 1$, and $p \neq 2$:

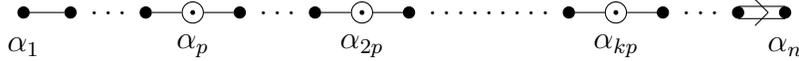
\vspace{1em}

\begin{center}
\begin{tikzpicture}[scale=1.25]
    \node (dot1) at (0,0) [draw, circle, inner sep=1.5pt, fill=black] {};
    \node (dot2) at (.5,0) [draw, circle, inner sep=1.5pt, fill=black] {};
    \draw (dot1) -- (dot2);
    \node[below] at (0,-.15) {$\alpha_1$};
    
    \node (ellipsis) at (.9,0) {$\ldots$};
    
    \node (dot3) at (1.3,0) [draw, circle, inner sep=1.5pt, fill=black] {};
    \node (dot4) at (1.8,0) [draw, circle, inner sep=.5pt, fill=black] {};
    \node[draw, circle, inner sep=0pt, minimum size=8pt] (circle4) at (1.8,0) {};
    \node (dot5) at (2.3,0) [draw, circle, inner sep=1.5pt, fill=black] {};
    \draw (dot3) -- (circle4.west);
    \draw (circle4.east) -- (dot5);
    \node[below] at (circle4.south) {$\alpha_p$};

    \node (ellipsis2) at (2.7,0) {$\ldots$};

    \node (dot6) at (3.1,0) [draw, circle, inner sep=1.5pt, fill=black] {};
    \node (dot7) at (3.6,0) [draw, circle, inner sep=.5pt, fill=black] {};
    \node[draw, circle, inner sep=0pt, minimum size=8pt] (circle7) at (3.6,0) {};
    \node (dot8) at (4.1,0) [draw, circle, inner sep=1.5pt, fill=black] {};
    \draw (dot6) -- (circle7.west);
    \draw (circle7.east) -- (dot8);
    \node[below] at (circle7.south) {$\alpha_{2p}$};

    \node (ellipsis3) at (4.5,0) {$\ldots$};
    \node (ellipsis4) at (4.95,0) {$\ldots$};
    \node (ellipsis5) at (5.4,0) {$\ldots$};

    \node (dot9) at (5.8,0) [draw, circle, inner sep=1.5pt, fill=black] {};
    \node (dot10) at (6.3,0) [draw, circle, inner sep=.5pt, fill=black] {};
    \node[draw, circle, inner sep=0pt, minimum size=8pt] (circle10) at (6.3,0) {};
    \node (dot11) at (6.8,0) [draw, circle, inner sep=1.5pt, fill=black] {};
    \draw (dot9) -- (circle10.west);
    \draw (circle10.east) -- (dot11);
    \node[below] at (circle10.south) {$\alpha_{kp}$};

    \node (ellipsis3) at (7.2,0) {$\ldots$};

    \node (dot12) at (7.6,0) [draw, circle, inner sep=1.5pt, fill=black] {};
    \node (dot13) at (8.1,0) [draw, circle, inner sep=1.5pt, fill=black] {};
    \draw (dot12.north) -- (dot13.north);
    \draw (dot12.south) -- (dot13.south);
    \draw (7.92,0) -- (7.78,.14);
    \draw (7.92,0) -- (7.78,-.14);
    \node[below] at (8.1,-.15) {$\alpha_n$};
\end{tikzpicture}
\end{center}

\begin{itemize}
\item \underline{$S$-irreducible:} 
	\begin{enumerate}[label = (\alph*)]
		\item $(B_k, B_{\ell})$: $p=2\ell + 1$; $k,\ell \geq 1$; $2k\ell + k + \ell = n$.  
	\end{enumerate}
\item \underline{Non-$S$-irreducible:} 
	\begin{enumerate}[label = (\roman*)]
		\item $(B_k, D_{\ell} \oplus B_{n-2k\ell - \ell})$: $p = 2\ell$; $k \geq 1$; $\ell \geq 2$; $2k\ell + \ell \leq n$.
		\item $(B_k, B_{\ell} \oplus D_{n-2k\ell - k - \ell})$: $p = 2\ell + 1$; $k \geq 1$; $\ell \geq 0$; $2k\ell + k + \ell + 1 < n$.
	\end{enumerate}
\end{itemize}

\vspace{1em}

The index of these subalgebras can be computed by hand or can be found in \cite[Table XII]{S-subalgebras} for $n \leq 6$. Putting this all together, we get the following partial list of dual pairs in $B_n$ (for $2 \leq n \leq 6$), including \textit{all} $S$-irreducible dual pairs and \textit{all} dual pairs coming from maximal regular subalgebras of maximal rank. 

\begin{center}
\begin{table}[H]
\begin{tabular}{| c | c || c | c || c | c || c | c |} \hline
Rank & Dual Pair & Rank & Dual Pair & Rank & Dual Pair & Rank & Dual Pair  \\ \hline
2 & $(A_1, A_1)$ & 4 & $(A_1, A_1 \oplus B_2^{1''})$  & 5 & $(A_1, A_1 \oplus B_3)$ & 6 & $(A_1, A_1 \oplus B_4)$ \\ \cline{1-2}
3 & $(\tilde{A_1}, 2A_1)$ & 4(ii) & $(A_1^{2''}, A_3^{1''})$ & 5(ii) & $(A_1^{2''}, D_4)$ & 6(i) & $(A_1^8, 2A_1^{3'})$ \\
3 & $(A_1, A_1 \oplus \tilde{A_1})$ & 4(ii) & $(B_2^{1''}, 2A_1)$ & 5(ii) & $(B_2^{1''}, A_3^{1''})$ & 6(ii) & $(A_1^{2''}, D_5)$ \\
 & & 4(a) & $(A_1^6, A_1^6)$ & 5(ii) & $(B_3, 2A_1)$ & 6(ii) & $(B_2^{1''}, D_4)$ \\
 & & & & & & 6(ii) & $(B_3, A_3^{1''})$ \\
 & & & & & & 6(ii) & $(B_4, 2A_1)$ \\ 
 & & & & & & 6(ii) & $(A_1^6, A_1^6 \oplus 2A_1)$ \\ \hline
\end{tabular}
\caption{A \textit{partial} list of dual pairs in $B_n$ (for $2 \leq n \leq 6$).}
\label{table:dps-in-Bn}
\end{table}
\end{center}

To find the remaining non-$S$-irreducible dual pairs, one can carry out the process described in Section \ref{sec:non-S-irred}.

\subsection{Type \texorpdfstring{$C_n$}{C}}

Recall that the complex Lie algebra of type $C_n$ has $(C_k, C_{n-k})$ ($k=1,2,\ldots , n-1$) as dual pairs arising from maximal regular subalgebras (see Table \ref{table:max-reg-dps}). By Theorem \ref{thm:8.1}, any $S$-irreducible dual pair in $C_n$ has at least one factor that is an admissible subalgebra of $C_n$. Additionally, it turns out that \textit{some} of the non-$S$-irreducible dual pairs in $C_n$ have a factor that is an admissible subalgebra of $C_n$. For type $C_n$ (with $n \geq 3$), the admissible dual pairs arise from two different admissible diagrams, as outlined below. \\

\noindent \textit{\underline{Type 1:}} We require that $p = 2 \ell$ and that $(2k+1) \ell \leq n$.

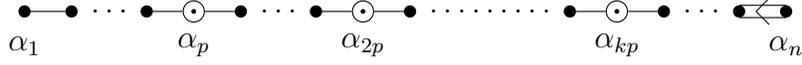
\vspace{1em}

\begin{center}
\begin{tikzpicture}[scale=1.25]
    \node (dot1) at (0,0) [draw, circle, inner sep=1.5pt, fill=black] {};
    \node (dot2) at (.5,0) [draw, circle, inner sep=1.5pt, fill=black] {};
    \draw (dot1) -- (dot2);
    \node[below] at (0,-.15) {$\alpha_1$};
    
    \node (ellipsis) at (.9,0) {$\ldots$};
    
    \node (dot3) at (1.3,0) [draw, circle, inner sep=1.5pt, fill=black] {};
    \node (dot4) at (1.8,0) [draw, circle, inner sep=.5pt, fill=black] {};
    \node[draw, circle, inner sep=0pt, minimum size=8pt] (circle4) at (1.8,0) {};
    \node (dot5) at (2.3,0) [draw, circle, inner sep=1.5pt, fill=black] {};
    \draw (dot3) -- (circle4.west);
    \draw (circle4.east) -- (dot5);
    \node[below] at (circle4.south) {$\alpha_p$};

    \node (ellipsis2) at (2.7,0) {$\ldots$};

    \node (dot6) at (3.1,0) [draw, circle, inner sep=1.5pt, fill=black] {};
    \node (dot7) at (3.6,0) [draw, circle, inner sep=.5pt, fill=black] {};
    \node[draw, circle, inner sep=0pt, minimum size=8pt] (circle7) at (3.6,0) {};
    \node (dot8) at (4.1,0) [draw, circle, inner sep=1.5pt, fill=black] {};
    \draw (dot6) -- (circle7.west);
    \draw (circle7.east) -- (dot8);
    \node[below] at (circle7.south) {$\alpha_{2p}$};

    \node (ellipsis3) at (4.5,0) {$\ldots$};
    \node (ellipsis4) at (4.95,0) {$\ldots$};
    \node (ellipsis5) at (5.4,0) {$\ldots$};

    \node (dot9) at (5.8,0) [draw, circle, inner sep=1.5pt, fill=black] {};
    \node (dot10) at (6.3,0) [draw, circle, inner sep=.5pt, fill=black] {};
    \node[draw, circle, inner sep=0pt, minimum size=8pt] (circle10) at (6.3,0) {};
    \node (dot11) at (6.8,0) [draw, circle, inner sep=1.5pt, fill=black] {};
    \draw (dot9) -- (circle10.west);
    \draw (circle10.east) -- (dot11);
    \node[below] at (circle10.south) {$\alpha_{kp}$};

    \node (ellipsis3) at (7.2,0) {$\ldots$};

    \node (dot12) at (7.6,0) [draw, circle, inner sep=1.5pt, fill=black] {};
    \node (dot13) at (8.1,0) [draw, circle, inner sep=1.5pt, fill=black] {};
    \draw (dot12.north) -- (dot13.north);
    \draw (dot12.south) -- (dot13.south);
    \draw (7.78,0) -- (7.92,.14);
    \draw (7.78,0) -- (7.92,-.14);
    \node[below] at (8.1,-.15) {$\alpha_n$};
\end{tikzpicture}
\end{center}

\begin{itemize}
\item \underline{$S$-irreducible:} 
	\begin{enumerate}[label = (\alph*)]
	\item $(B_k, C_{\ell})$: $k,\ell \geq 1$; $(2k+1)\ell = n$.
	\end{enumerate}
\item \underline{Non-$S$-irreducible:}
	\begin{enumerate}[label = (\roman*)]
	\item $(B_k, C_{\ell} \oplus C_{n-2k\ell - \ell})$: $k \geq 1$; $\ell \geq 1$; $(2k+1)\ell < n$.
	\end{enumerate}
\end{itemize}

\vspace{1em}

\noindent \textit{\underline{Type 2:}} We require that $n = (k+1)p$, $k \geq 0$, and $p \geq 1$.

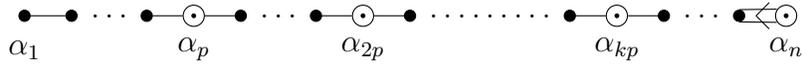
\vspace{1em}

\begin{center}
\begin{tikzpicture}[scale=1.25]
    \node (dot1) at (0,0) [draw, circle, inner sep=1.5pt, fill=black] {};
    \node (dot2) at (.5,0) [draw, circle, inner sep=1.5pt, fill=black] {};
    \draw (dot1) -- (dot2);
    \node[below] at (0,-.15) {$\alpha_1$};
    
    \node (ellipsis) at (.9,0) {$\ldots$};
    
    \node (dot3) at (1.3,0) [draw, circle, inner sep=1.5pt, fill=black] {};
    \node (dot4) at (1.8,0) [draw, circle, inner sep=.5pt, fill=black] {};
    \node[draw, circle, inner sep=0pt, minimum size=8pt] (circle4) at (1.8,0) {};
    \node (dot5) at (2.3,0) [draw, circle, inner sep=1.5pt, fill=black] {};
    \draw (dot3) -- (circle4.west);
    \draw (circle4.east) -- (dot5);
    \node[below] at (circle4.south) {$\alpha_p$};

    \node (ellipsis2) at (2.7,0) {$\ldots$};

    \node (dot6) at (3.1,0) [draw, circle, inner sep=1.5pt, fill=black] {};
    \node (dot7) at (3.6,0) [draw, circle, inner sep=.5pt, fill=black] {};
    \node[draw, circle, inner sep=0pt, minimum size=8pt] (circle7) at (3.6,0) {};
    \node (dot8) at (4.1,0) [draw, circle, inner sep=1.5pt, fill=black] {};
    \draw (dot6) -- (circle7.west);
    \draw (circle7.east) -- (dot8);
    \node[below] at (circle7.south) {$\alpha_{2p}$};

    \node (ellipsis3) at (4.5,0) {$\ldots$};
    \node (ellipsis4) at (4.95,0) {$\ldots$};
    \node (ellipsis5) at (5.4,0) {$\ldots$};

    \node (dot9) at (5.8,0) [draw, circle, inner sep=1.5pt, fill=black] {};
    \node (dot10) at (6.3,0) [draw, circle, inner sep=.5pt, fill=black] {};
    \node[draw, circle, inner sep=0pt, minimum size=8pt] (circle10) at (6.3,0) {};
    \node (dot11) at (6.8,0) [draw, circle, inner sep=1.5pt, fill=black] {};
    \draw (dot9) -- (circle10.west);
    \draw (circle10.east) -- (dot11);
    \node[below] at (circle10.south) {$\alpha_{kp}$};

    \node (ellipsis3) at (7.2,0) {$\ldots$};

    \node (dot12) at (7.6,0) [draw, circle, inner sep=1.5pt, fill=black] {};
    \node (dot13) at (8.1,0) [draw, circle, inner sep=.5pt, fill=black] {};
    \node[draw, circle, inner sep=0pt, minimum size=8pt] (circle13) at (8.1,0) {};
    \node[below] at (circle13.south) {$\alpha_{n}$};
    \draw (dot12.north) -- (8,.07);
    \draw (dot12.south) -- (8,-.07);
    \draw (7.78,0) -- (7.92,.14);
    \draw (7.78,0) -- (7.92,-.14);
\end{tikzpicture}
\end{center}

\begin{itemize}
\item \underline{$S$-irreducible:}
	\begin{enumerate}[label=(\alph*)]
	\item $(C_{k+1}, B_{\ell})$: $p = 2\ell+1$; $k \geq 0$; $\ell \geq 1$.
	\item $(C_{k+1}, D_{\ell})$: $p = 2\ell$; $k \geq 0$; $\ell \geq 2$.
	\end{enumerate}
\end{itemize}

\vspace{1em}

The index of these subalgebras can be computed by hand or can be found in \cite[Table XII]{S-subalgebras} for $n \leq 6$. \\

Putting this all together, we get the following partial list of dual pairs in $C_n$ (for $3 \leq n \leq 6$), including \textit{all} $S$-irreducible dual pairs and \textit{all} dual pairs coming from maximal regular subalgebras of maximal rank.

\begin{center}
\begin{table}[H]
\begin{tabular}{| c | c || c | c || c | c || c | c |} \hline
Rank & Dual Pair & Rank & Dual Pair & Rank & Dual Pair & Rank & Dual Pair \\ \hline
3 & $(A_1, B_2)$ & 4 & $(A_1, C_3)$ & 5 & $(A_1, C_4)$ & 6 & $(A_1, C_5)$ \\
3(1a) & $(A_1^8, A_1^3)$ & 4 & $(B_2, B_2)$ & 5 & $(B_2, C_3)$ & 6 & $(B_2, C_4)$ \\ 
 &  & 4(1i) & $(A_1^8, A_1^3 \oplus A_1)$ & 5(1i) & $(A_1^8, A_1^3 \oplus B_2)$ & 6 & $(C_3, C_3)$ \\
 &  & 4(2b) & $(A_1^4, 2A_1^4)$ & 5(1a) & $(B_2^4, A_1^5)$ & 6(1i) & $(A_1^8, A_1^3 \oplus C_3)$ \\
 &  &  &  & & & 6(1i) & $(B_2^4, A_1^5 \oplus A_1)$ \\
 &  &  &  & & & 6(1a) & $(A_1^{16}, B_2^3)$ \\ 
 &  & &  & & & 6(2b) & $(A_1^6, A_3^4)$ \\ \hline
\end{tabular}
\caption{A \textit{partial} list of dual pairs in $C_n$ (for $3 \leq n \leq 6$).}
\label{table:dps-in-Cn}
\end{table}
\end{center}

To find the remaining non-$S$-irreducible dual pairs, one can carry out the process described in Section \ref{sec:non-S-irred}.

\subsection{Type \texorpdfstring{$D_n$}{D}} \label{subsec:D_n}

Recall that the complex Lie algebra of type $D_n$ has $(A_1, A_1 \oplus D_{n-2})$ and $(D_k, D_{n-k})$ ($k=2,3,\ldots , n-2$) as dual pairs arising from maximal regular subalgebras (see Table \ref{table:max-reg-dps}). For type $D_n$, Theorem \ref{thm:8.1} suggests that there are certain $S$-irreducible dual pairs in $D_n$ that do \textit{not} have admissible factors. In particular, $D_n$ has $S$-irreducible dual pairs of type $(D_k, D_{\ell})$, where $n = 2k\ell$ and $k,\ell \geq 2$ \cite[Theorem 7.3.2]{rubenthaler}. Additionally, $D_n$ has $S$-irreducible dual pairs of type $(B_{p_1} \oplus B_{q_1}, B_{p_2} \oplus B_{q_2})$, where $(B_{p_1}, B_{p_2})$ is an $S$-irreducible dual pair in $B_p$ (with $2p+1 = (2p_1 + 1)(2p_2+1)$), where $(B_{q_1}, B_{q_2})$ is an $S$-irreducible dual pair in $B_q$ (with $2q+1 = (2q_1 + 1)(2q_2+1)$), and where $n = p + q + 1$ \cite[Proposition 7.3.3]{rubenthaler}. However, note that neither of these types of dual pairs appear for $n \leq 6$.

For all other $S$-irreducible dual pairs in $D_n$, Theorem \ref{thm:8.1} implies that at least one factor of the pair is an admissible subalgebra of $D_n$. Additionally, it turns out that \textit{some} of the non-$S$-irreducible dual pairs in $D_n$ have a factor that is an admissible subalgebra of $D_n$. For type $D_n$ (with $n \geq 4$), the admissible dual pairs arise from two different admissible diagrams, as outlined below. \\

\noindent \textit{\underline{Type 1:}} We require that $n \geq kp+2$ and that $(2k+1)p \leq 2n$.

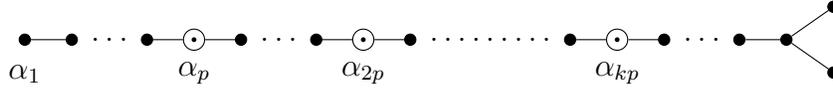
\vspace{1em}

\begin{center}
\begin{tikzpicture}[scale=1.25]
    \node (dot1) at (0,0) [draw, circle, inner sep=1.5pt, fill=black] {};
    \node (dot2) at (.5,0) [draw, circle, inner sep=1.5pt, fill=black] {};
    \draw (dot1) -- (dot2);
    \node[below] at (0,-.15) {$\alpha_1$};
    
    \node (ellipsis) at (.9,0) {$\ldots$};
    
    \node (dot3) at (1.3,0) [draw, circle, inner sep=1.5pt, fill=black] {};
    \node (dot4) at (1.8,0) [draw, circle, inner sep=.5pt, fill=black] {};
    \node[draw, circle, inner sep=0pt, minimum size=8pt] (circle4) at (1.8,0) {};
    \node (dot5) at (2.3,0) [draw, circle, inner sep=1.5pt, fill=black] {};
    \draw (dot3) -- (circle4.west);
    \draw (circle4.east) -- (dot5);
    \node[below] at (circle4.south) {$\alpha_p$};

    \node (ellipsis2) at (2.7,0) {$\ldots$};

    \node (dot6) at (3.1,0) [draw, circle, inner sep=1.5pt, fill=black] {};
    \node (dot7) at (3.6,0) [draw, circle, inner sep=.5pt, fill=black] {};
    \node[draw, circle, inner sep=0pt, minimum size=8pt] (circle7) at (3.6,0) {};
    \node (dot8) at (4.1,0) [draw, circle, inner sep=1.5pt, fill=black] {};
    \draw (dot6) -- (circle7.west);
    \draw (circle7.east) -- (dot8);
    \node[below] at (circle7.south) {$\alpha_{2p}$};

    \node (ellipsis3) at (4.5,0) {$\ldots$};
    \node (ellipsis4) at (4.95,0) {$\ldots$};
    \node (ellipsis5) at (5.4,0) {$\ldots$};

    \node (dot9) at (5.8,0) [draw, circle, inner sep=1.5pt, fill=black] {};
    \node (dot10) at (6.3,0) [draw, circle, inner sep=.5pt, fill=black] {};
    \node[draw, circle, inner sep=0pt, minimum size=8pt] (circle10) at (6.3,0) {};
    \node (dot11) at (6.8,0) [draw, circle, inner sep=1.5pt, fill=black] {};
    \draw (dot9) -- (circle10.west);
    \draw (circle10.east) -- (dot11);
    \node[below] at (circle10.south) {$\alpha_{kp}$};

    \node (ellipsis3) at (7.2,0) {$\ldots$};

    \node (dot12) at (7.6,0) [draw, circle, inner sep=1.5pt, fill=black] {};
    \node (dot13) at (8.1,0) [draw, circle, inner sep=1.5pt, fill=black] {};
    \node (dot14) at (8.6,.35) [draw, circle, inner sep=1.5pt, fill=black] {};
    \node (dot15) at (8.6,-.35) [draw, circle, inner sep=1.5pt, fill=black] {};  
    \draw (dot12) -- (dot13);
    \draw (dot13) -- (dot14);
    \draw (dot13) -- (dot15);
\end{tikzpicture}
\end{center}

\begin{itemize}
\item \underline{$S$-irreducible:}
	\begin{enumerate}[label=(\alph*)]
	\item $(B_k, D_{\ell})$: $p=2\ell$; $k \geq 1$; $\ell \geq 2$; $(2k+1)\ell = n$.
	\item $(B_k, B_{\ell} \oplus B_{n-2k\ell-k-\ell -1})$: $p=2\ell + 1$; $k \geq 1$; $\ell \geq 0$; $2k\ell + k + \ell < n$. 
	\end{enumerate}
\item \underline{Non-$S$-irreducible:} 
	\begin{enumerate}[label = (\roman*)]
	\item $(B_k, D_{\ell} \oplus D_{n-2k\ell - \ell})$: $p=2\ell$; $k \geq 1$; $\ell \geq 3$; $(2k+1)\ell + 1 < n$.
	\end{enumerate}
\end{itemize}

\vspace{1em}

\noindent \textit{\underline{Type 2:}} We require that $n = (k+1)p$ and that $p = 2\ell$. 

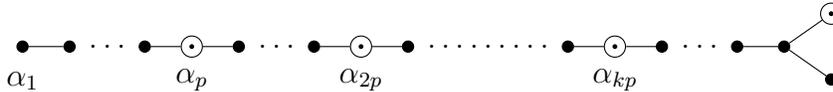
\vspace{1em}

\begin{center}
\begin{tikzpicture}[scale=1.25]
    \node (dot1) at (0,0) [draw, circle, inner sep=1.5pt, fill=black] {};
    \node (dot2) at (.5,0) [draw, circle, inner sep=1.5pt, fill=black] {};
    \draw (dot1) -- (dot2);
    \node[below] at (0,-.15) {$\alpha_1$};
    
    \node (ellipsis) at (.9,0) {$\ldots$};
    
    \node (dot3) at (1.3,0) [draw, circle, inner sep=1.5pt, fill=black] {};
    \node (dot4) at (1.8,0) [draw, circle, inner sep=.5pt, fill=black] {};
    \node[draw, circle, inner sep=0pt, minimum size=8pt] (circle4) at (1.8,0) {};
    \node (dot5) at (2.3,0) [draw, circle, inner sep=1.5pt, fill=black] {};
    \draw (dot3) -- (circle4.west);
    \draw (circle4.east) -- (dot5);
    \node[below] at (circle4.south) {$\alpha_p$};

    \node (ellipsis2) at (2.7,0) {$\ldots$};

    \node (dot6) at (3.1,0) [draw, circle, inner sep=1.5pt, fill=black] {};
    \node (dot7) at (3.6,0) [draw, circle, inner sep=.5pt, fill=black] {};
    \node[draw, circle, inner sep=0pt, minimum size=8pt] (circle7) at (3.6,0) {};
    \node (dot8) at (4.1,0) [draw, circle, inner sep=1.5pt, fill=black] {};
    \draw (dot6) -- (circle7.west);
    \draw (circle7.east) -- (dot8);
    \node[below] at (circle7.south) {$\alpha_{2p}$};

    \node (ellipsis3) at (4.5,0) {$\ldots$};
    \node (ellipsis4) at (4.95,0) {$\ldots$};
    \node (ellipsis5) at (5.4,0) {$\ldots$};

    \node (dot9) at (5.8,0) [draw, circle, inner sep=1.5pt, fill=black] {};
    \node (dot10) at (6.3,0) [draw, circle, inner sep=.5pt, fill=black] {};
    \node[draw, circle, inner sep=0pt, minimum size=8pt] (circle10) at (6.3,0) {};
    \node (dot11) at (6.8,0) [draw, circle, inner sep=1.5pt, fill=black] {};
    \draw (dot9) -- (circle10.west);
    \draw (circle10.east) -- (dot11);
    \node[below] at (circle10.south) {$\alpha_{kp}$};

    \node (ellipsis3) at (7.2,0) {$\ldots$};

    \node (dot12) at (7.6,0) [draw, circle, inner sep=1.5pt, fill=black] {};
    \node (dot13) at (8.1,0) [draw, circle, inner sep=1.5pt, fill=black] {};
    \node[draw, circle, inner sep=0pt, minimum size=8pt, fill=white] (circle14) at (8.6,.35) {};
    \node (dot15) at (8.6,-.35) [draw, circle, inner sep=1.5pt, fill=black] {};  
    \draw (dot12) -- (dot13);
    \draw (dot13) -- (circle14);
    \node (dot14) at (8.6,.35) [draw, circle, inner sep=.5pt, fill=black] {};
    \draw (dot13) -- (dot15);
\end{tikzpicture}
\end{center}

\begin{itemize}
\item \underline{$S$-irreducible:}
	\begin{enumerate}[label = (\alph*)]
	\item $(C_{k+1}, C_{\ell})$: $k \geq 0$; $\ell \geq 1$.
	\end{enumerate}
\end{itemize}

\vspace{1em}

The index of these subalgebras can be computed by hand or can be found in \cite[Table XII]{S-subalgebras} for $n \leq 6$. \\

Putting this all together, we get the following partial list of dual pairs in $D_n$ (for $4 \leq n \leq 6$), including \textit{all} $S$-irreducible dual pairs and \textit{all} dual pairs coming from maximal regular subalgebras of maximal rank.

\begin{table}[H]
\begin{tabular}{| c | c || c | c || c | c |} \hline
Rank & Dual Pair & Rank & Dual Pair & Rank & Dual Pair \\ \hline
4 & $(A_1, 3A_1)$ & 5 & $(A_1, A_1 \oplus A_3^{''})$ & 6 & $(A_1, A_1 \oplus D_4)$ \\
4 & $(2A_1, 2A_1)$ & 5 & $(2A_1, A_3^{''})$ & 6 & $(2A_1, D_4)$ \\
4(1b) & $(A_1^{2'}, B_2^{1'})$ & 5(1b) & $(A_1^{2''}, B_3^1)$ & 6 & $(A_3^{''}, A_3^{''})$ \\ 
4(1b) & $(A_1^{2''}, B_2^{1''})$ & 5(1b) & $(B_2^{1''}, B_2^{1''})$ & 6(1a) & $(A_1^8, A_1^{3'} \oplus A_1^{3''})$ \\
4(1b) & $(A_1^{2'''}, B_2^{1'''})$ & 5(1b) & $(A_1^6, A_1^6)$ & 6(1b) & $(A_1^{2''}, B_4^1)$ \\ 
& & & & 6(1b) & $(B_2^{1''}, B_3^1)$ \\
& & & & 6(1b) & $(A_1^6, A_1^6 \oplus A_1^{2''})$ \\
& & & & 6(2a) & $(A_1^{3'}, C_3^{1''})$ \\ 
& & & & 6(2a) & $(A_1^{3''}, C_3^{1'})$ \\ \hline
\end{tabular}
\caption{A \textit{partial} list of dual pairs in $D_n$ (for $4 \leq n \leq 6$).}
\label{table:dps-in-Dn}
\end{table}

To find the remaining non-$S$-irreducible dual pairs, one can carry out the process described in Section \ref{sec:non-S-irred}. 

\section{Non-\texorpdfstring{$S$}{S}-irreducible dual pairs} \label{sec:non-S-irred}

So far, we have established how to classify dual pairs coming from maximal regular subalgebras of maximal rank (see Table \ref{table:max-reg-dps}) and how to classify all dual pairs with an admissible factor (see Section \ref{sec:S-irred}). In this section, we discuss how to classify the remaining non-$S$-irreducible dual pairs.

\begin{thm}[{\cite[Theorem 5.12]{rubenthaler}}] \label{thm:non-S-irred}
Let $(\mf{a},\mf{b})$ be a semisimple non-$S$-irreducible dual pair in a semisimple $\mf{g}$. Then there exists a maximal-rank regular semisimple subalgebra $\mf{u}$ in $\mf{g}$ such that $\mf{a} \oplus \mf{b} \subset \mf{u} \subset \mf{g}$ and such that $\mf{a} \oplus \mf{b}$ is an $S$-subalgebra of $\mf{u}$. 
\end{thm}

Note that the subalgebra $\mf{u}$ in this theorem is maximal-rank and regular but need not be maximal regular. This theorem shows that to find all of the non-$S$-irreducible dual pairs in a simple Lie algebra $\mf{g}$, one can carry out the following process:
\begin{enumerate}
\item Write down all of the maximal-rank regular subalgebras of $\mf{g}$ (found in Table \ref{table:reg-max-rank}).
\item For each maximal-rank regular subalgebra $\mf{u}$ of $\mf{g}$, find all $S$-subalgebras $\mf{a} \oplus \mf{b}$ of $\mf{u}$.
\item For each pair $(\mf{a},\mf{b})$ from step 2, determine whether $(\mf{a},\mf{b})$ is a dual pair in $\mf{g}$.
\end{enumerate}

Using Table \ref{table:max-reg-dps}, step 1 is straightforward. For step 2, the $S$-subalgebras of $\mf{u}$ can be found in Figure \ref{fig:S-subalg-classical-1} or \ref{fig:S-subalg-classical-2} if $\mf{u}$ is a classical simple Lie algebra of rank at most 6 and can be found in Figure \ref{fig:S-subalgebras} if $\mf{u}$ is an exceptional Lie algebra. This information (together with the following lemma) can be used to find the $S$-subalgebras of the semisimple Lie algebras $\mf{u}$ we will consider.

\begin{lemma} \label{lem:S-subalgebras}
Let $\mf{u}_1$ and $\mf{u}_2$ be complex semisimple Lie algebras.
\begin{enumerate}[label=(\alph*)]
\item Let $s_{1}$ (resp.~$s_{2}$) be an $S$-subalgebra of $\mf{u}_1$ (resp.~$\mf{u}_2$). Then $s_{1} \oplus s_{2}$ is an $S$-subalgebra of $\mf{u}_1 \oplus \mf{u}_2$.
\item Let $s_{1} \oplus s_{2}$ be an $S$-subalgebra of $\mf{u}_1 \oplus \mf{u}_2$. Then $s_{1}$ is an $S$-subalgebra of $\mf{u}_1$ and $s_{2}$ is an $S$-subalgebra of $\mf{u}_2$. 
\end{enumerate}
\end{lemma}

\begin{proof}[Proof of (a)]
Suppose, for the sake of contradiction, that $s_{1} \oplus s_{2}$ is contained in some proper regular subalgebra $r_{1} \oplus r_{2}$ of $\mf{u}_1 \oplus \mf{u}_2$. Then there exists a Cartan subalgebra $\mf{h}_{1} \oplus \mf{h}_{2}$ of $\mf{u}_1 \oplus \mf{u}_2$ such that 
$$[\mf{h}_{1} \oplus \mf{h}_{2}, r_{1} \oplus r_{2}] = [\mf{h}_{1}, r_{1}] \oplus [\mf{h}_{2}, r_{2}] \subset r_{1} \oplus r_{2}.$$ 
It is not hard to see that $\mf{h}_{1}$ and $\mf{h}_{2}$ are Cartan subalgebras of $\mf{u}_1$ and $\mf{u}_2$, respectively, and hence that $r_{1}$ and $r_{2}$ are regular subalgebras of $\mf{u}_1$ and $\mf{u}_2$ containing $s_{1}$ and $s_{2}$, respectively. Moreover, since $r_{1} \oplus r_{2}$ is proper in $\mf{u}_1 \oplus \mf{u}_2$, at least one of $r_{1}$ and $r_{2}$ has to be proper. This then gives a contradiction, proving that $s_{1} \oplus s_{2}$ is an $S$-subalgebra of $\mf{u}_1 \oplus \mf{u}_2$. 
\end{proof}

\begin{proof}[Proof of (b)]
Suppose, for the sake of contradiction, that $s_{1}$ is contained in a proper regular subalgebra $r_{1}$ of $\mf{u}_1$. Then $r_{1} \oplus \mf{u}_2$ is clearly a proper regular subalgebra of $\mf{u}_1 \oplus \mf{u}_2$ containing $s_{1} \oplus s_{2}$, a contradiction. Repeating this argument with the roles of $\mf{u}_1$ and $\mf{u}_2$ reversed, we again obtain a contradiction. 
\end{proof}

For step 3, we first need to view the $S$-subalgebras of $\mf{u}$ as subalgebras of $\mf{g}$. For a simple $S$-subalgebra $\mf{a}$ of a simple factor of $\mf{u}$, we can apply the formula $i_{\mf{a} \hookrightarrow \mf{g}} = i_{\mf{a}\hookrightarrow \mf{u}} \cdot i_{\mf{u} \hookrightarrow \mf{g}}$ to find the index of $\mf{a}$ in $\mf{g}$. Note that since the simple factors of $\mf{u}$ will almost always have index 1 in $\mf{g}$ (see Remark \ref{rmk:index-1}), we will usually have that $i_{\mf{a} \hookrightarrow \mf{u}} = i_{\mf{a} \hookrightarrow \mf{g}}$. While the value of $i_{\mf{a} \hookrightarrow \mf{g}}$ is easy to find, it can often be more challenging to figure out the conjugacy class of $\mf{a}$ in situations where the conjugacy class is not uniquely determined by the type and index of $\mf{a}$. This challenge will be addressed in several of the examples in the following section. Finally, the remainder of step 3 (i.e.~determining whether a candidate dual pair is actually a dual pair) is even more invovled. The next section is primarily dedicated to discussing different methods for confirming and eliminating these candidate dual pairs for the exceptional Lie algebras.

\section{Eliminating and confirming candidate dual pairs} \label{sec:elim-and-confirm}

\subsection{Straightforward eliminations}

Perhaps the most straightforward approach for ruling out a candidate dual pair $(\mf{a}, \mf{b})$ in $\mf{g}$ is to show that $\mf{z}_{\mf{g}} (\mf{a}) \supsetneq \mf{b}$ using information about subalgebras and dual pairs that have already been confirmed. For example, whenever $\mf{a} \oplus \mf{a}'$ is a subalgebra of $\mf{g}$ with $\mf{a}' \supsetneq \mf{b}$, this gives that 
$$\mf{z}_{\mf{g}}(\mf{a}) \supseteq \mf{a}' \supsetneq \mf{b},$$
so we can conclude that $(\mf{a}, \mf{b})$ is \textit{not} a dual pair in $\mf{g}$.

Similarly, suppose that $(\mf{a}, \mf{b})$ is a dual pair in $\mf{g}$ and that $\mf{b}_1 \oplus \mf{b}_2$ is a subalgebra of $\mf{b}$. This situation often leads to a candidate dual pair in $\mf{g}$ of the form $(\mf{a} \oplus \mf{b}_1, \mf{b}_2)$ (e.g.~when $\mf{a} \oplus \mf{b}$ is regular of maximal rank in $\mf{g}$ and $\mf{b}_1 \oplus \mf{b}_2$ is an $S$-subalgebra of $\mf{b}$). Since $(\mf{a}, \mf{b})$ is a dual pair in $\mf{g}$, we have that 
$$\mf{z}_{\mf{g}} (\mf{a} \oplus \mf{b}_1) = \mf{z}_{\mf{b}} (\mf{b}_1)$$
and that 
$$\mf{z}_{\mf{g}} (\mf{b}_2) \supseteq \mf{z}_{\mf{a} \oplus \mf{b}} (\mf{b}_2) = \mf{a} \oplus \mf{z}_{\mf{b}}(\mf{b}_2).$$
In this way, we see that for $(\mf{a} \oplus \mf{b}_1, \mf{b}_2)$ to be a dual pair in $\mf{g}$, it is necessary that $(\mf{b}_1, \mf{b}_2)$ is a dual pair in $\mf{b}$. (Note, however, that this condition is not enough to guarantee that $(\mf{a} \oplus \mf{b}_1, \mf{b}_2)$ is a dual pair in $\mf{g}$.) This idea is demonstrated in the following example:

\begin{example}[$(A_1^2 \oplus 3A_1, B_2^1)$ is \textit{not} a dual pair in $E_7$]
Since $(B_2^{1''}, B_3^1)$ is an $S$-irreducible dual pair in $D_6$, we have that $B_2^1 \oplus B_3^1 \oplus A_1$ is an $S$-subalgebra of $E_7$, and hence that 
$$\mf{z}_{E_7} (B_2^1) \supseteq B_3^1 \oplus A_1 \supsetneq A_1^2 \oplus 3A_1.$$
Therefore, $(A_1^2 \oplus 3A_1, B_2^1)$ is \textit{not} a dual pair in $E_7$. 
\end{example}

\subsection{Eliminating and confirming using Dynkin's and Carter's tables}

To augment our strategy of using information about known dual pairs to make straightforward eliminations of candidate dual pairs, we can reference \cite[Table 25]{dynkin} and the tables in \cite[pp.~401--405]{Carter}. The tables in \cite[pp.~401--405]{Carter} indicate the type of the centralizer of each 3-dimensional subalgebra of the exceptional Lie algebras, and \cite[Table 25]{dynkin} lists all of the conjugacy classes of subalgebras of the exceptional Lie algebras.

\begin{example}[$(A_3, \tilde{A_1})$ is a dual pair in $F_4$] \label{ex:A3-A1tilde-dp-F4}
Recall that $A_3 \oplus \tilde{A_1}$ is a maximal-rank regular subalgebra in $F_4$, but \textit{not} a maximal regular subalgebra. Therefore, $(A_3, \tilde{A_1})$ is a candidate dual pair in $F_4$, but Proposition \ref{prop:maximal-reg-is-dp} does not apply. Fortunately, by \cite[p.~401]{Carter}, we have that the centralizer of $\tilde{A_1}$ in $F_4$ has type $A_3$. Moreover, by \cite[Table 25]{dynkin}, there is a unique (up to conjugation) subalgebra of type $A_3$ in $F_4$, and it has index 1. On the other hand, the centralizer of $A_3$ in $F_4$ clearly contains $\tilde{A_1}$ and has rank at most 1, so we can conclude that $\mf{z}_{F_4}(A_3) = \tilde{A_1}$. It follows that $(A_3, \tilde{A_1})$ is a dual pair in $F_4$.
\end{example}

\begin{example}[$(\tilde{A_1}, B_2^1)$ and $(\tilde{A_1}, 2\tilde{A_1})$ are \textit{not} dual pairs in $F_4$]
Recall that $(A_1^{2''}, B_2^{1''})$ is an $S$-irreducible dual pair in $D_4$ (see Table \ref{table:dps-in-Dn}). Since $D_4$ is a maximal-rank regular subalgebra of index 1 in $F_4$, $(A_1^2, B_2^1)$ is a candidate dual pair in $F_4$. However, by \cite[Table 20]{dynkin}, there is a unique (up to conjugation) subalgebra of type $A_1$ and index 2 in $F_4$ (namely, $\tilde{A_1}$). Since we already know that the centralizer of $\tilde{A_1}$ in $F_4$ is $A_3$ (by \cite[p.~401]{Carter} or Example \ref{ex:A3-A1tilde-dp-F4}), it follows that $(A_1^2, B_2^1) = (\tilde{A_1}, B_2^1)$ is \textit{not} a dual pair in $F_4$. Similarly, $(A_1^2, 2 A_1^2) = (\tilde{A_1}, 2 \tilde{A_1})$ is \textit{not} a dual pair in $F_4$. 
\end{example}

\subsection{Eliminating and confirming using dimension of the centralizer}

Let $\mf{g}$ be a complex simple Lie algebra, and let $\tilde{\mf{g}}$ be a semisimple subalgebra of $\mf{g}$. The representation which is induced on $\tilde{\mf{g}}$ by the adjoint representation of $\mf{g}$ splits into two components: one component acts on $\tilde{\mf{g}}$ and can be thought of as the adjoint representation of $\tilde{\mf{g}}$; the complementary component, called the \textbf{characteristic representation} of $\tilde{\mf{g}}$, is denoted by $\chi_{\tilde{\mf{g}}}$ \cite[No.~6]{dynkin}. 

\begin{lemma} \label{lem:dimension-of-cent}
Let $M$ denote the number of copies of the trivial representation in $\chi_{\tilde{\mf{g}}}$. Then 
$$M = \dim \mf{z}_{\mf{g}}(\tilde{\mf{g}}).$$
\end{lemma}

\begin{proof}
Let $G$ be the simply connected Lie group associated to $\mf{g}$, and let $\tilde{G} \subset G$ be the simply connected Lie group associated to $\tilde{\mf{g}}$. We have that
$$\Res^{G}_{\tilde{G}} (\text{Ad}_{G}) = \text{Ad}_{\tilde{G}} \oplus \chi_{\tilde{G}},$$
where $\text{Ad}_G : G \rightarrow \text{Aut} (\mf{g})$ is the adjoint representation of $G$, where $\text{Ad}_{\tilde{G}} : \tilde{G} \rightarrow \text{Aut}(\tilde{\mf{g}})$ is the adjoint representation of $\tilde{G}$, and where $\chi_{\tilde{G}} : \tilde{G} \rightarrow \text{Aut} (\mf{g}/\tilde{\mf{g}})$ is the resulting complementary component. Noting that $\chi_{\tilde{\mf{g}}} = d \chi_{\tilde{G}}$, we see that the number of copies of the trivial representation in $\chi_{\tilde{\mf{g}}}$ equals the number of copies of the trivial representation in $\chi_{\tilde{G}}$. 

The set of copies of the trivial representation in $\chi_{\tilde{G}}$ correspond to a set of linearly independent vectors, each of which is fixed by all of $\chi_{\tilde{G}}(\tilde{G})$. Indeed, any vector that is fixed by all of $\chi_{\tilde{G}}(\tilde{G})$ spans a one-dimensional subrepresentation, which is necessarily the trivial representation. Moreover, we note that each such fixed vector corresponds to an element of $\mf{z}_{\mf{g}}(\tilde{\mf{g}})$. Indeed, for $Y \in \mf{g}/\tilde{\mf{g}}$, we have that
$$e^X Y e^{-X} = Y \; \text{ for all } X \in \tilde{\mf{g}} \hspace{.5cm} \iff \hspace{.5cm} [X,Y] = 0 \; \text{ for all } X \in \tilde{\mf{g}}.$$
Finally, since $\tilde{\mf{g}}$ is semisimple, we have that $\tilde{\mf{g}} \cap \mf{z}_{\mf{g}} (\tilde{\mf{g}}) = \varnothing$, so these elements of $\mf{g}/\tilde{\mf{g}}$ span $\mf{z}_{\mf{g}}(\tilde{\mf{g}})$. It follows that $M = \dim \mf{z}_{\mf{g}}(\tilde{\mf{g}})$, as desired.
\end{proof}

The decomposition of $\chi_{\tilde{\mf{g}}}$ into irreducibles is included in \cite[Table 25]{dynkin} for subalgebras of the exceptional Lie algebras. The number $M$ in the previous lemma can be read off of this table as the coefficient of the trivial representation, which Dynkin denotes by $N$. This gives us an easy way to look up the dimension of $\mf{z}_{\mf{g}}(\tilde{\mf{g}})$ for such subalgebras, which in some cases allows us to confirm or eliminate candidate dual pairs.

\begin{example}[$(A_2^3, A_2^3)$ is \textit{not} a dual pair in $E_8$]
By Subsection \ref{subsec:A_n}, there is an $S$-irreducible dual pair in $A_8$ where both factors have type $A_2$ and index 3. By \cite[Table 25]{dynkin}, $E_8$ has two non-conjugate subalgebras of type $A_2$ and index 3 (i.e.~$A_2^{3'}$ and $A_2^{3''}$). However, by the same table, $\mf{z}_{E_8}(A_2^{3'})$ has dimension 28 and $\mf{z}_{E_8}(A_2^{3''})$ has dimension 14. Since $\dim A_2 = 8$, we therefore see that neither of these possibilities leads to a dual pair in $E_8$.
\end{example}

\begin{example}[$(B_3^1, B_4^1)$ is a dual pair in $E_8$] \label{ex:E8-B3-B4}
By Subsection \ref{subsec:D_n}, there is an $S$-irreducible dual pair in $D_8$ where one factor has type $B_3$ and index 1 and the other factor has type $B_4$ and index 1. By \cite[Table 25]{dynkin}, these types and index values uniquely specify subalgebras of $E_8$ (up to conjugation). By the same table, $\dim \mf{z}_{E_8}(B_3^1) = 36 = \dim B_4$ and $\dim \mf{z}_{E_8}(B_4^1) = 21 = \dim B_3$. From this, it's clear that $(B_3^1, B_4^1)$ is a dual pair in $E_8$.
\end{example}

\begin{example}[$(A_1^3, A_2^{2''} \oplus A_1)$ is a dual pair in $E_6$] \label{ex:E6A1^3-dp}
By Table \ref{table:dps-in-An}, $(A_1^3, A_2^2)$ is an $S$-irreducible dual pair in $A_5$. Since $A_5 \oplus A_1$ is a maximal regular subalgebra of maximal rank in $E_6$, $(A_5, A_1)$ is a dual pair in $E_6$ and there is a candidate dual pair in $E_6$ with the following types and index values: $(A_1^3, A_2^{2} \oplus A_1)$. By \cite[Table 25]{dynkin}, there is a unique (up to conjugation) subalgebra of type $A_1$ and index 3 in $E_6$, as well as a unique (up to conjugation) subalgebra of type $A_1$ and index 1 in $E_6$; however, there are two conjugacy classes of subalgebras of type $A_2$ and index 2. Fortunately, with Lemma \ref{lem:dimension-of-cent}, \cite[Table 25]{dynkin}, and some further investigation, we can determine which conjugacy class we are dealing with here.  

To this end, consider the extended Dynkin diagram for $E_6$:

\setlength{\unitlength}{2cm}
\begin{picture}(5,2.0)(-5.2,-0.5)
  \multiput(-1.59,-.05)(.5,0){1}{\large $\triangledown$}
  \multiput(-1.5,.5)(.5,0){1}{\circle*{.10}}
  \multiput(-2.5,1.0)(.5,0){2}{\circle*{.10}}
  \multiput(-1.0,1.0)(.5,0){2}{\circle*{.10}}
  \multiput(-1.5,1.0)(.5,0){1}{\circle*{.10}}
  \multiput(-2.5, 1.0)(0,.5){1}{\line(1,0){1}}
  \multiput(-1.5,1.0)(0,.5){1}{\line(1,0){1}}
  \multiput(-1.5,0.1)(.5,0){1}{\line(0,1){.9}}
  \put(-1.4,.47){\tiny $\alpha_2$}
  \put(-1.55,1.15){\tiny $\alpha_4$}
  \put(-2.55,1.15){\tiny $\alpha_1$}
  \put(-2.07,1.15){\tiny $\alpha_3$}
  \put(-1.07,1.15){\tiny $\alpha_5$}
  \put(-0.57,1.15){\tiny $\alpha_6$}
\end{picture}

Crossing out the node corresponding to $\alpha_2$, we obtain the maximal regular subalgebra of type $A_5 \oplus A_1$ in $E_6$. To understand $A_1^3$ and this $A_2^2$ as subalgebras of $E_6$, we can consider the following diagram (where the admissible diagram corresponding to $A_1^3$ in $A_5$ is indicated in blue):

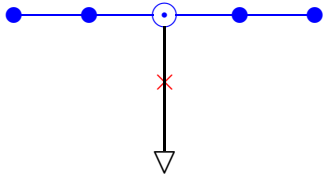
\setlength{\unitlength}{2cm}
\begin{picture}(5,2.0)(-5.2,-0.5)
  \multiput(-1.59,-.05)(.5,0){1}{\large $\triangledown$}
  \multiput(-1.58,.5)(.5,0){1}{\textcolor{red}{$\times$}}
  \multiput(-2.5,1.0)(.5,0){2}{\color{blue}\circle*{.10}}
  \multiput(-1.0,1.0)(.5,0){2}{\color{blue}\circle*{.10}}
  \multiput(-1.5,1.0)(.5,0){1}{\color{blue}\circle{.15}}
  \multiput(-1.5,1.0)(.5,0){1}{\color{blue}\circle*{.04}}
  \multiput(-2.5, 1.0)(0,.5){1}{\color{blue}\line(1,0){.925}}
  \multiput(-1.425,1.0)(0,.5){1}{\color{blue}\line(1,0){.925}}
  \multiput(-1.5,0.1)(.5,0){1}{\line(0,1){.825}}
\end{picture}

Realizing the $A_5$ root system as in \cite[Plate I]{bourbaki}, it is not hard to show that the defining vector for $A_1^3$ is $(1,1,1,-1,-1,-1)$. With this, we see that the centralizer of $A_1^3$ in $A_5$ has type $A_2$ and is diagonally embedded in the subalgebra of type $A_2 \oplus A_2$ coming from $\alpha_1$, $\alpha_2$, $\alpha_5$, and $\alpha_6$ in $E_6$. In this way, we see that the centralizer of this $A_2^2$ in $E_6$ should contain $A_1^3$ in addition to a subalgebra of type $A_2$ coming from $\alpha_2$ and $\tilde{\alpha}$ (since the nodes corresponding to $\alpha_2$ and $\tilde{\alpha}$ are not adjacent to $\alpha_1$, $\alpha_2$, $\alpha_5$, or $\alpha_6$). Therefore, the dimension of the centralizer of this $A_2^2$ is at least 11. Consulting \cite[Table 25]{dynkin}, we see that this subalgebra must therefore be $A_2^{2''}$.

Finally, \cite[p.~402]{Carter} gives that $\mf{z}_{E_6}(A_1^3)$ has type $A_2 \oplus A_1$, and 
$$\mf{z}_{E_6} (A_2^{2''} \oplus A_1) = \mf{z}_{A_5} (A_2^2) = A_1^3,$$
so we can conclude that $(A_1^3, A_2^{2''} \oplus A_1)$ is a dual pair in $E_6$.
\end{example}

\subsection{Eliminating and confirming based on fixed vectors \texorpdfstring{of $\chi_{\tilde{\mf{g}}}$}{}}

While looking up the dimension of centralizers in \cite[Table 25]{dynkin} (using Lemma \ref{lem:dimension-of-cent}) helps us confirm and eliminate many candidate dual pairs, there are a few cases in which we need to compute $\chi_{\tilde{\mf{g}}}$ more explicitly. For example, consider the candidate dual pair in $E_7$ with types and index values $(A_2^2, A_1^3 \oplus A_2)$ (to be explored in more detail in Example \ref{ex:E7-A2^2-A1^3-A2}). By \cite[Table 25]{dynkin}, there are two conjugacy classes of subalgebras of type $A_2$ and index 2 in $E_7$. Moreover, while we know that the centralizer of the $A_2^2$ in question contains $A_1^3 \oplus A_2$, Lemma \ref{lem:dimension-of-cent} and \cite[Table 25]{dynkin} show that the centralizers of both $A_2^{2'}$ and $A_2^{2''}$ have dimension at least 11 (so we cannot conclude whether we're working with $A_2^{2'}$ or $A_2^{2''}$ based on this information).

\begin{remark} \label{rmk:rep-theory-elims}
In situations like the one just described, we can calculate $\chi_{\tilde{\mf{g}}}$ by hand to check for copies of the trivial representation in $\chi_{\tilde{\mf{g}}}$ and to determine the conjugacy class of $\tilde{\mf{g}}$.
\end{remark}

\begin{example}[$(A_2^{2'}, A_1^{3''} \oplus A_2)$ is a dual pair in $E_7$] \label{ex:E7-A2^2-A1^3-A2}
Consider once again the candidate dual pair in $E_7$ with types and index values $(A_2^2, A_1^3 \oplus A_2)$. To determine the conjugacy class of $A_1^3$, we can look at the embeddings $A_1^3 \subset A_5^{''} \subset E_7$. To this end, consider the extended Dynkin diagram for $E_7$:

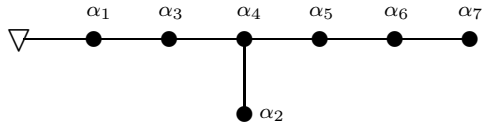
\setlength{\unitlength}{2cm}
\begin{picture}(5,1.5)(-5.2,0)
  \multiput(-1.5,.5)(.5,0){1}{\circle*{.10}}
  \put(-3.09,.92){\large $\triangledown$}
  \multiput(-2.5,1.0)(.5,0){2}{\circle*{.10}}
  \multiput(-1.0,1.0)(.5,0){3}{\circle*{.10}}
  \multiput(-1.5,1.0)(.5,0){1}{\circle*{.10}}
  \multiput(-2.5, 1.0)(0,.5){1}{\line(1,0){1}}
  \multiput(-1.5,1.0)(0,.5){1}{\line(1,0){1.5}}
  \multiput(-1.5,0.5)(.5,0){1}{\line(0,1){.5}}
  \put(-2.97,1){\line(1,0){.5}}
  \put(-1.4,.47){\tiny $\alpha_2$}
  \put(-1.55,1.15){\tiny $\alpha_4$}
  \put(-2.55,1.15){\tiny $\alpha_1$}
  \put(-2.07,1.15){\tiny $\alpha_3$}
  \put(-1.07,1.15){\tiny $\alpha_5$}
  \put(-0.57,1.15){\tiny $\alpha_6$}
  \put(-.08,1.15){\tiny $\alpha_7$}
\end{picture}

Crossing out the node corresponding to $\alpha_3$, we obtain the maximal regular subalgebra of type $A_2 \oplus A_5^{''}$ in $E_7$. To understand this index-3 subalgebra of type $A_1$, we can consider the following diagram (where the admissible diagram corresponding to $A_1^3$ in $A_5^{''}$ is indicated in blue):

\setlength{\unitlength}{2cm}
\begin{picture}(5,1.5)(-5.2,0)
  \multiput(-1.5,.5)(.5,0){1}{\color{blue}\circle*{.10}}
  \put(-3.09,.92){\large $\triangledown$}
  \multiput(-2.5,1.0)(.5,0){1}{\circle*{.10}}
  \put(-2.1,.95){\textcolor{red}{$\times$}}
  \put(-1.0,1.0){\color{blue}\circle{.15}}
  \put(-1.0,1.0){\color{blue}\circle*{.04}}
  \multiput(-0.5,1.0)(.5,0){2}{\color{blue}\circle*{.10}}
  \multiput(-1.5,1.0)(.5,0){1}{\color{blue}\circle*{.10}}
  \multiput(-2.5, 1.0)(0,.5){1}{\line(1,0){.95}}
  \put(-1.5,1.0){\color{blue}\line(1,0){.43}}
  \put(-.93,1.0){\color{blue}\line(1,0){.93}}
  \multiput(-1.5,0.5)(.5,0){1}{\color{blue}\line(0,1){.5}}
  \put(-2.97,1){\line(1,0){.5}}
\end{picture}

We can realize the $A_5$ root system as in \cite[Plate I]{bourbaki}. With this, it is not hard to show that this $A_1^3$ has defining vector $(1,1,1,-1,-1,-1)$. Putting this vector in dominant form, we see that this subalgebra has weighted diagram

\begin{center}
\begin{tikzcd}[row sep = small, column sep = small]
0 \arrow[dash]{r} & 0 \arrow[dash]{r} & 0 \arrow[dash]{r} \arrow[dash]{d} & 0 \arrow[dash]{r} & 0 \arrow[dash]{r} & 2 \\
 & & 0
\end{tikzcd}
\end{center}
in $E_7$, meaning the subalgebra in question is $A_1^{3''}$ \cite[Table 19]{dynkin}.

For $A_2^2$, it is a bit more difficult to figure out which conjugacy class we're working with. However, the information about how this $A_2^2$ sits inside of $E_7$ allows us to determine the number of copies of the trivial representation in $\chi_{A_2^2}$ without knowing beforehand which conjugacy class this $A_2^2$ belongs to. To this end, let $\mf{s} = A_2^2 \oplus A_1^{3''} \oplus A_2$, and let 
$$S = SU(3,\mathbb{C}) \times SU(2,\mathbb{C}) \times SU(3,\mathbb{C}).$$
Note that by the same reasoning as in Lemma \ref{lem:dimension-of-cent}, $\Res^{\mf{s}}_{A_2^2} ( \chi_{\mf{s}})$ has no copies of the trivial representation if and only if $\Res^{S}_{SU(3,\mathbb{C})}(\chi_{S})$ has no fixed vectors if and only if $\mf{z}_{E_7}(A_2^2) = A_1^{3''} \oplus A_2$. We have that $\chi_{S}$ acts on the space $E
_7/(A_2^2 \oplus A_1^{3''} \oplus A_2)$, and that
$$E_7/ (A_2^2 \oplus A_1^{3''} \oplus A_2) \simeq [E_7/(A_5^{''} \oplus A_2)] \oplus [A_5^{''}/(A_2^2 \oplus A_1^{3})].$$
Now, since $(A_2^2, A_1^3)$ is a dual pair in $A_5$, $\Res^S_{SU(3,\mathbb{C})}(\chi_{S})$ will not have any fixed vectors coming from $A_5^{''}/(A_2^2 \oplus A_1^3)$. The $E_7/(A_5^{''} \oplus A_2)$-factor of $\chi_{S}$ decomposes in terms of fundamental representations of $SU(3,\mathbb{C})$ and $SU(2,\mathbb{C})$ as follows:
$$[\wedge^2 (\mathbb{C}^3 \otimes \mathbb{C}^2) \otimes \mathbb{C}^3] \oplus [ ( \wedge^2 (\mathbb{C}^3 \otimes \mathbb{C}^2) )^* \otimes (\mathbb{C}^3)^* ] .$$
Letting $\{ (1,0,0), (0,1,0), (0,0,1) \}$ and $\{ (1,0), (0,1) \}$ be the standard weight bases for $\mathbb{C}^3$ and $\mathbb{C}^2$, respectively, we get the following weight basis for $\mathbb{C}^3 \otimes \mathbb{C}^2$:
\begin{align*}
\{ & (1,0,0,1,0), \, (0,1,0,1,0), \, (0,0,1,1,0), \\
 & (1,0,0,0,1), \, (0,1,0,0,1), \, (0,0,1,0,1) \}.
\end{align*}
A weight basis for $\wedge^2 (\mathbb{C}^3 \otimes \mathbb{C}^2)$ is then obtained by taking sums of any two distinct elements from the $\mathbb{C}^3 \otimes \mathbb{C}^2$ weight basis. We get $(2,0,0,1,1)$ and $(1,1,0,2,0)$ as highest weights, which have weight spaces of dimensions 6 and 9, respectively (by the Weyl dimension formula). In particular, when decomposing the $E_7/(A_5^{''} \oplus A_2)$-factor of $\chi_S$ into irreducibles, we do not get any copies of the trivial representation of the $SU(3,\mathbb{C})$ corresponding to $A_2^2$. It follows that $\mf{z}_{E_7}(A_2^2) = A_1^{3''} \oplus A_2$. Finally, by \cite[Table 25]{dynkin} and Lemma \ref{lem:dimension-of-cent}, we see that this index-2 subalgebra of type $A_2$ is $A_2^{2'}$.
\end{example}

\section{All of the semisimple dual pairs in the exceptional Lie algebras} \label{sec:lists}

\subsection{\texorpdfstring{$S$}{S}-irreducible dual pairs in the exceptional Lie algebras}

As described in Section \ref{sec:S-irred}, one can generate a complete list of the $S$-irreducible dual pairs in an exceptional Lie algebra $\mf{g}$ by considering all admissible diagrams associated with an appropriate $(\Psi, \theta)$, computing the subalgebra $\tilde{\mf{g}_{\theta}}$ corresponding to $(\Psi, \theta)$, computing $\mf{z}_{\mf{g}}(\tilde{\mf{g}_{\theta}})$ and checking whether $\tilde{\mf{g}_{\theta}} \oplus \mf{z}_{\mf{g}} ( \tilde{\mf{g}_{\theta}} )$ appears in Figure \ref{fig:S-subalgebras} as an $S$-subalgebra of $\mf{g}$. Rubenthaler carries out this process in \cite[Sections 6.8--6.12]{rubenthaler}, and the results are summarized in Table \ref{table:S-irreducibles}.

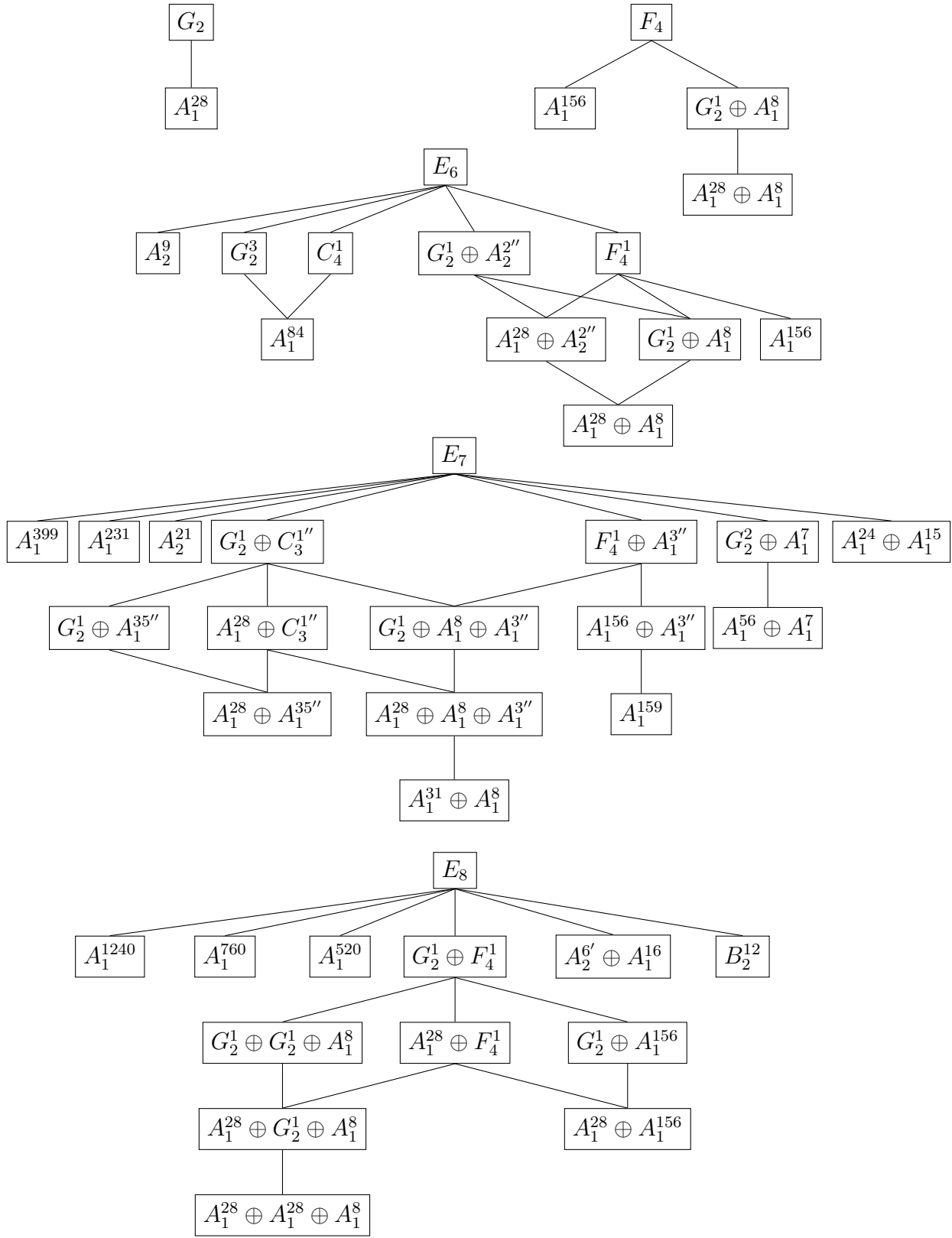
\begin{figure} 
\begin{subfigure}{\textwidth}
        \centering\begin{tikzpicture}
\begin{scope}[xshift=-1cm]
    \node (box1) at (0,0) [draw, rectangle] {$G_2$};
    
    \node (box2) at (0,-1.5) [draw, rectangle] {$A_1^{28}$};
    
    \draw (box1.south) -- (box2.north);
\end{scope}
\begin{scope}[xshift=7cm]
    \node (top) at (0,0) [draw, rectangle] {$F_4$};
    
    \node (left) at (-1.5,-1.5) [draw, rectangle] {$A_1^{156}$};
    
    \node (right) at (1.5,-1.5) [draw, rectangle] {$G_2^1 \oplus A_1^8$};
    
    \node (belowRight) at (1.5,-3) [draw, rectangle] {$A_1^{28} \oplus A_1^8$};
    
    \draw (top.south) -- (left.north);
    \draw (top.south) -- (right.north);
    
    \draw (right.south) -- (belowRight.north);
\end{scope}
\end{tikzpicture}
\end{subfigure}

\vspace{-1.2cm}

\begin{subfigure}{\textwidth}
   \centering \begin{tikzpicture}
    \node (top) at (0,0) [draw, rectangle] {$E_6$};
    
    \node (box1) at (-5,-1.5) [draw, rectangle] {$A_2^9$};
    
    \node (box2) at (-3.5,-1.5) [draw, rectangle] {$G_2^3$};

    \node (box2-1) at (-2.75,-3) [draw, rectangle] {$A_1^{84}$};
    \draw (box2.south) -- (box2-1.north);
    
    \node (box3) at (-2,-1.5) [draw, rectangle] {$C_4^1$};
    \draw (box3.south) -- (box2-1.north);
    
    \node (box4) at (.5,-1.5) [draw, rectangle] {$G_2^1 \oplus A_2^{2''}$};
    
    \node (box5) at (3,-1.5) [draw, rectangle] {$F_4^1$};

    \node (box4-1) at (1.75,-3) [draw, rectangle] {$A_1^{28} \oplus A_2^{2''}$};
    \draw (box4.south) -- (box4-1.north);
    \draw (box5.south) -- (box4-1.north);

    \node (box5-2) at (4.25,-3) [draw, rectangle] {$G_2^{1} \oplus A_1^{8}$};
    \draw (box5.south) -- (box5-2.north);
    \draw (box4.south) -- (box5-2.north);

    \node (box4-1-1) at (3,-4.5) [draw, rectangle] {$A_1^{28} \oplus A_1^{8}$};
    \draw (box4-1.south) -- (box4-1-1.north);
    \draw (box5-2.south) -- (box4-1-1.north);

    \node (box5-3) at (6,-3) [draw, rectangle] {$A_1^{156}$};
    \draw (box5.south) -- (box5-3.north);
    
    \draw (top.south) -- (box1.north);
    \draw (top.south) -- (box2.north);
    \draw (top.south) -- (box3.north);
    \draw (top.south) -- (box4.north);
    \draw (top.south) -- (box5.north);
\end{tikzpicture}
\end{subfigure}

\vspace{-.2cm}

\begin{subfigure}{\textwidth}
   \centering \begin{tikzpicture}
    \node (top) at (0,0) [draw, rectangle] {$E_7$};
    
    \node (box1) at (-7.25,-1.5) [draw, rectangle] {$A_1^{399}$};
    
    \node (box2) at (-6,-1.5) [draw, rectangle] {$A_1^{231}$};

    \node (box3) at (-4.85,-1.5) [draw, rectangle] {$A_2^{21}$};

    \node (box4) at (-3.25,-1.5) [draw, rectangle] {$G_2^1 \oplus C_3^{1''}$};

    \node (box5) at (3.25,-1.5) [draw, rectangle] {$F_4^1 \oplus A_1^{3''}$};

    \node (box6) at (5.45,-1.5) [draw, rectangle] {$G_2^2 \oplus A_1^{7}$};

    \node (box7) at (7.6,-1.5) [draw, rectangle] {$A_1^{24} \oplus A_1^{15}$};

    \node (box4-2) at (-3.25,-3) [draw, rectangle] {$A_1^{28} \oplus C_3^{1''}$};
    \node (box4-2-1) at (-3.25,-4.5) [draw, rectangle] {$A_1^{28} \oplus A_1^{35''}$};

    \node (box4-3) at (0,-3) [draw, rectangle] {$G_2^{1} \oplus A_1^8 \oplus A_1^{3''}$};
    \node (box4-3-1) at (0,-4.5) [draw, rectangle] {$A_1^{28} \oplus A_1^8 \oplus A_1^{3''}$};
    \node (box4-3-1-1) at (0,-6) [draw, rectangle] {$A_1^{31} \oplus A_1^8$};

    \node (box5-1) at (3.25,-3) [draw, rectangle] {$A_1^{156} \oplus A_1^{3''}$};
    \node (box5-1-1) at (3.25,-4.5) [draw, rectangle] {$A_1^{159}$};

    \node (box6-1) at (5.45,-3) [draw, rectangle] {$A_1^{56} \oplus A_1^{7}$};

    \node (box4-1) at (-6,-3) [draw, rectangle] {$G_2^{1} \oplus A_1^{35''}$};

    \draw (box4.south) -- (box4-1.north);
    \draw (box4.south) -- (box4-2.north);
    \draw (box4.south) -- (box4-3.north);
    \draw (box4-1.south) -- (box4-2-1.north);
    \draw (box4-2.south) -- (box4-2-1.north); 
    \draw (box4-2.south) -- (box4-3-1.north);
    \draw (box4-3-1.south) -- (box4-3-1-1.north);  
    \draw (box4-3.south) -- (box4-3-1.north); \draw (box5.south) -- (box4-3.north);
    \draw (box5.south) -- (box5-1.north);
    \draw (box5-1.south) -- (box5-1-1.north);
    \draw (box6.south) -- (box6-1.north);

    \draw (top.south) -- (box1.north);
    \draw (top.south) -- (box2.north);
    \draw (top.south) -- (box3.north);
    \draw (top.south) -- (box4.north);
    \draw (top.south) -- (box5.north);
    \draw (top.south) -- (box6.north);
    \draw (top.south) -- (box7.north);
\end{tikzpicture}
\end{subfigure}

\vspace{.5cm}

\begin{subfigure}{\textwidth}
\hspace{1.1cm}
   \begin{tikzpicture}
   \begin{scope}
    \node (top) at (0,0) [draw, rectangle] {$E_8$};
    
    \node (box1) at (-6,-1.5) [draw, rectangle] {$A_1^{1240}$};
    
    \node (box2) at (-4,-1.5) [draw, rectangle] {$A_1^{760}$};

    \node (box3) at (-2,-1.5) [draw, rectangle] {$A_1^{520}$};

    \node (box4) at (0,-1.5) [draw, rectangle] {$G_2^1 \oplus F_4^{1}$};

    \node (box5) at (2.75,-1.5) [draw, rectangle] {$A_2^{6'} \oplus A_1^{16}$};

    \node (box6) at (5,-1.5) [draw, rectangle] {$B_2^{12}$};

    \node (box4-1) at (0,-3) [draw, rectangle] {$A_1^{28} \oplus F_4^{1}$};

    \node (box4-2) at (-3,-3) [draw, rectangle] {$G_2^{1} \oplus G_2^{1} \oplus A_1^{8}$};
    \node (box4-2-1) at (-3,-4.5) [draw, rectangle] {$A_1^{28} \oplus G_2^{1} \oplus A_1^{8}$};
    \node (box4-2-1-1) at (-3,-6) [draw, rectangle] {$A_1^{28} \oplus A_1^{28} \oplus A_1^{8}$};

    \node (box4-3) at (3,-3) [draw, rectangle] {$G_2^{1} \oplus A_1^{156}$};
    \node (box4-3-1) at (3,-4.5) [draw, rectangle] {$A_1^{28} \oplus A_1^{156}$};

    \draw (box4.south) -- (box4-1.north); 
    \draw (box4.south) -- (box4-2.north); 
    \draw (box4.south) -- (box4-3.north); 
    \draw (box4-2.south) -- (box4-2-1.north); 
    \draw (box4-2-1.south) -- (box4-2-1-1.north); 
    \draw (box4-1.south) -- (box4-2-1.north); 
    \draw (box4-1.south) -- (box4-3-1.north); 
    \draw (box4-3.south) -- (box4-3-1.north); 

    \draw (top.south) -- (box1.north);
    \draw (top.south) -- (box2.north);
    \draw (top.south) -- (box3.north);
    \draw (top.south) -- (box4.north);
    \draw (top.south) -- (box5.north);
    \draw (top.south) -- (box6.north); 
\end{scope}
\end{tikzpicture}
\end{subfigure}

\caption{Inclusion relations among the $S$-subalgebras of the exceptional Lie algebras \cite[Table 39]{dynkin}.}
\label{fig:S-subalgebras}
\end{figure}

\begin{table}[H]
\begin{tabular}{ | c | c | c | c | } \hline
$\mf{g}$ & Dual Pair $(\tilde{\mf{g}_{\theta}}, \mf{z}_{\mf{g}}(\tilde{\mf{g}_{\theta}}))$ & $\Psi \setminus \theta$ & Maximal $S$-Subalgebra? \\ \hline 
$F_4$ & $(A_1^8, G_2^1)$ & $\{ \alpha_4 \}$ & Yes \\ \hline 
$E_6$ & $(G_2^1, A_2^{2''})$ & $\{ \alpha_2, \alpha_4 \}$ & Yes \\ \hline
$E_7$ & $(G_2^1, C_3^{1''})$ & $\{ \alpha_1, \alpha_3 \}$ & Yes \\
  & $(A_1^7, G_2^2)$ & $\{ \alpha_2 \}$ & Yes \\
  & $(A_1^{24}, A_1^{15})$ & $\{ \alpha_4 \}$ & Yes \\
  & $(A_1^{3''}, F_4^1)$ & $\{ \alpha_7 \}$ & Yes \\
  & $(A_1^8, G_2^1 \oplus A_1^{3''})$ & $\{ \alpha_6 \}$ & No $(\subset F_4^1 \oplus A_1^{3''})$ \\ \hline 
$E_8$ & $(A_1^{16}, A_2^{6'})$ & $\{ \alpha_2 \}$ & Yes \\ 
 & $(G_2^1, F_4^1)$ & $\{ \alpha_7, \alpha_8 \}$ & Yes \\
  & $(A_1^8 , G_2^1 \oplus G_2^1)$ & $\{ \alpha_1 \}$ & No $(\subset G_2^1 \oplus F_4^1)$ \\
\hline 
\end{tabular}
\caption{$S$-irreducible dual pairs in the exceptional Lie algebras.}
\label{table:S-irreducibles}
\end{table}

\subsection{Non-\texorpdfstring{$S$}{S}-irreducible dual pairs in the exceptional Lie algebras}

The $S$-irreducible dual pairs in the exceptional Lie algebras are shown in Table \ref{table:S-irreducibles}. To find the non-$S$-irreducible dual pairs in the exceptional Lie algebras, we will carry out the process described in Section \ref{sec:non-S-irred}.

\subsubsection{Dual pairs in \texorpdfstring{$G_2$}{G2}}

Recall that $G_2$ has no $S$-irreducible dual pairs (see Table \ref{table:S-irreducibles}). Additionally, recall that by Table \ref{table:maximal-regular} and Proposition \ref{prop:maximal-reg-is-dp}, $(A_1, \tilde{A_1})$ is a dual pair in $G_2$. By Theorem \ref{thm:non-S-irred}, to find any remaining non-$S$-irreducible dual pairs, it suffices to consider the proper $S$-subalgebras of $A_1 \oplus \tilde{A_1}$, of which there are none. Therefore, $(A_1, \tilde{A_1})$ is the only dual pair in $G_2$.

\begin{table}[H]
\begin{tabular}{| c | c | c |} \hline
Max-Rank Reg.~Subalgebra & Dual Pair & How to Verify \\ \hline
$A_1 \oplus \tilde{A_1}$ & $(A_1, \tilde{A_1})$ & Proposition \ref{prop:maximal-reg-is-dp} \\ \hline
\end{tabular}
\caption{A complete list of dual pairs in $G_2$.}
\end{table}

\newpage

\subsubsection{Dual pairs in \texorpdfstring{$F_4$}{F4}} \label{subsec:F4}

By Table \ref{table:S-irreducibles}, $(A_1^8, G_2^1)$ is an $S$-irreducible dual pair in $F_4$. Additionally, as indicated in Table \ref{table:max-reg-dps}, we have that $(A_1, C_3)$ and $(A_2, \tilde{A_2})$ are dual pairs in $F_4$. By Theorem \ref{thm:non-S-irred}, to find the remaining non-$S$-irreducible dual pairs of $F_4$, it suffices to consider the $S$-subalgebras of the maximal-rank regular subalgebras of $F_4$ and to eliminate or confirm the resulting candidate dual pairs. The results of this analysis are summarized in Table \ref{table:F4}.  

\begin{table}[H]
\begin{center}
\begin{tabular}{| c | c | c |} \hline 
Max-Rank Reg.~Subalgebra & Dual Pair & How to Verify \\ \hline
$B_4$ & $(A_1^6, A_1^6)$ & \cite[p.~401]{Carter} \\ \hline 
$A_2 \oplus \tilde{A_2}$ & $(A_2, \tilde{A_2})$ & Proposition \ref{prop:maximal-reg-is-dp} \\ \hline
$C_3 \oplus A_1$ & $(C_3, A_1)$ & Proposition \ref{prop:maximal-reg-is-dp} \\ 
 & $(A_1^3, A_1^8 \oplus A_1)$ & Table \ref{table:dps-in-Cn}, \cite[p.~401]{Carter} \\ \hline 
 $A_3 \oplus \tilde{A_1}$ & $(A_3, \tilde{A_1})$ & Example \ref{ex:A3-A1tilde-dp-F4} \\ \hline
$B_2 \oplus 2A_1$ & $(B_2, 2A_1)$ & Table \ref{table:dps-in-Cn}, Lemma \ref{lem:dimension-of-cent}/\cite[Table 25]{dynkin} \\ \hline
\end{tabular}
\end{center}
\caption{A complete list of the non-$S$-irreducible dual pairs in $F_4$.}
\label{table:F4}
\end{table}

\subsubsection{Dual pairs in \texorpdfstring{$E_6$}{E6}}

Recall from Table \ref{table:S-irreducibles} that $(G_2^1, A_2^{2''})$ is an $S$-irreducible dual pair in $E_6$. Additionally, as indicated in Table \ref{table:max-reg-dps}, we have that $(A_1, A_5)$ and $(A_2, 2 A_2)$ are dual pairs in $E_6$. By Theorem \ref{thm:non-S-irred}, to find the remaining non-$S$-irreducible dual pairs of $E_6$, it suffices to consider the $S$-subalgebras of the maximal-rank regular subalgebras of $E_6$ and to eliminate or confirm the resulting candidate dual pairs. Carrying out this process, we get a complete list of non-$S$-irreducible dual pairs in $E_6$ as shown in Table \ref{table:E6}.

\begin{table}[H]
\begin{center}
\begin{tabular}{| c | c | c |} \hline 
Max-Rank Reg.~Subalgebra & Dual Pair & How to Verify \\ \hline  
$A_5 \oplus A_1$ & $(A_5, A_1)$ & Proposition \ref{prop:maximal-reg-is-dp} \\
 & $(A_1^3, A_2^{2''} \oplus A_1)$ & Example \ref{ex:E6A1^3-dp} \\ \hline 
$3A_2$ & $(A_2, 2 A_2)$ & Proposition \ref{prop:maximal-reg-is-dp} \\ \hline
\end{tabular}
\end{center}
\caption{A complete list of the non-$S$-irreducible dual pairs in $E_6$.}
\label{table:E6}
\end{table}

\subsubsection{Dual pairs in \texorpdfstring{$E_7$}{E7}}

By Table \ref{table:S-irreducibles}, $(G_2^1, C_3^{1''})$, $(A_1^7, G_2^2)$, $(A_1^{24}, A_1^{15})$, $(A_1^{3''}, F_4^1)$, and $(A_1^8, G_2^1 \oplus A_1^{3''})$ are $S$-irreducible dual pairs in $E_7$. Additionally, as indicated in Table \ref{table:max-reg-dps}, we have that $(A_1, D_6)$ and $(A_2, A_5^{''})$ are dual pairs in $E_7$. By Theorem \ref{thm:non-S-irred}, to find the remaining non-$S$-irreducible dual pairs of $E_7$, it suffices to consider the $S$-subalgebras of the maximal-rank regular subalgebras of $E_7$ and to eliminate or confirm the resulting candidate dual pairs. The results of this analysis are summarized in Table \ref{table:E7}. 

\begin{table}[H]
\begin{center} 
\begin{tabular}{| c | c | c |} \hline
Max-Rank Reg.~Subalgebra & Dual Pair & How to Verify \\ \hline
$D_6 \oplus A_1$ & $(D_6, A_1)$ & Proposition \ref{prop:maximal-reg-is-dp} \\ 
 & $(B_2^1, B_3^1 \oplus A_1)$ & Table \ref{table:dps-in-Dn}, Lemma \ref{lem:dimension-of-cent}/\cite[Table 25]{dynkin} \\
 & $(B_3^1, (B_2^1 \oplus A_1)^{''})$ & Table \ref{table:dps-in-Dn}, Lemma \ref{lem:dimension-of-cent}/\cite[Table 25]{dynkin} \\
 & $(C_3^{1'}, A_1^{3'} \oplus A_1)$ & Table \ref{table:dps-in-Dn}, Remark \ref{rmk:rep-theory-elims} \\
 & $(A_1^{3'}, C_3^{1'} \oplus A_1)$ & Table \ref{table:dps-in-Dn}, Lemma \ref{lem:dimension-of-cent}/\cite[Table 25]{dynkin} \\
 & $(A_1^{3'} \oplus A_1^{3''}, A_1^8 \oplus A_1)$ & Previous entry, Tables \ref{table:dps-in-Cn} \& \ref{table:dps-in-Dn} \\
 & $(B_4^1, (A_1^2 \oplus A_1)^{''})$ &  Table \ref{table:dps-in-Dn}, Lemma \ref{lem:dimension-of-cent}/\cite[Table 25]{dynkin} \\
 & $(A_1^2, B_4^1 \oplus A_1)$ & Table \ref{table:dps-in-Dn}, \cite[p.~403]{Carter} \\
 & $(A_1^6, A_1^6 \oplus A_1^2 \oplus A_1)$ & Table \ref{table:dps-in-Dn}, \cite[p.~403]{Carter} \\
  & $(A_1^6 \oplus A_1^2, A_1^6 \oplus A_1)$ & Previous entry, \cite[p.~403]{Carter} \\ \hline  
$A_5'' \oplus A_2$ & $(A_5'', A_2)$ & Proposition \ref{prop:maximal-reg-is-dp} \\
 & $(A_2^{2'}, A_1^{3''} \oplus A_2)$ & Example \ref{ex:E7-A2^2-A1^3-A2} \\ \hline
$D_4 \oplus 3A_1$ & $(D_4, 3A_1)$ & Table \ref{table:dps-in-Dn}, Lemma \ref{lem:dimension-of-cent}/\cite[Table 25]{dynkin} \\
 & $(D_4 \oplus A_1, 2A_1)$ & Proposition \ref{prop:maximal-reg-is-dp} for $D_6$ \\ 
 & $((B_2^1 \oplus A_1)^{'}, A_1^2 \oplus 2A_1 )$ & Previous entry, Tables \ref{table:dps-in-Dn} \& \ref{table:B2-subalgebras} \\ 
 & $((A_1^2 \oplus A_1)^{'}, B_2^1 \oplus 2A_1)$ & Table \ref{table:dps-in-Dn}, $(D_4 \oplus A_1, 2A_1)$ \\ \hline
$2A_3 \oplus A_1$ & $(A_3, A_3 \oplus A_1)$ & Table \ref{table:dps-in-Dn}, Lemma \ref{lem:dimension-of-cent}/\cite[Table 25]{dynkin} \\ \hline
\end{tabular}
\end{center}
\caption{A complete list of the non-$S$-irreducible dual pairs in $E_7$. Here, $(B_2^1 \oplus A_1)^{'}$ (resp.~$(B_2^1 \oplus A_1)^{''}$) refers to the subalgebra of $E_7$ that restricts to $B_2^{1'} \oplus A_1$ (resp.~$B_2^{1''} \oplus A_1$) in $D_6 \oplus A_1$. Similarly, $(A_1^2 \oplus A_1)^{'}$ (resp.~$(A_1^2 \oplus A_1)^{''}$) restricts to $A_1^{2'} \oplus A_1$ (resp.~$A_1^{2''} \oplus A_1$) in $D_6 \oplus A_1$.}
\label{table:E7}
\end{table}

\subsubsection{Dual pairs in \texorpdfstring{$E_8$}{E8}}

Finally, recall from Table \ref{table:S-irreducibles} that $(A_1^{16}, A_2^{6'})$, $(G_2^1, F_4^1)$, and $(A_1^8, G_2^1 \oplus G_2^1)$ are $S$-irreducible dual pairs in $E_8$. Additionally, as indicated in Table \ref{table:max-reg-dps}, we have that $(A_4, A_4)$, $(A_2, E_6)$, and $(A_1, E_7)$ are dual pairs in $E_8$. By considering the list of admissible subalgebras of $E_8$, we also get that $(A_1^{24}, A_1^{15}\oplus A_1)$ and $(G_2^2, A_1^7 \oplus A_1)$ are non-$S$-irreducible dual pairs \cite[Section 6.10]{rubenthaler}. By Theorem \ref{thm:non-S-irred}, to find the remaining non-$S$-irreducible dual pairs of $E_8$, it suffices to consider the $S$-subalgebras of the maximal-rank regular subalgebras of $E_8$ and to eliminate or confirm the resulting candidate dual pairs. Carrying out this process, we get the following complete list of non-$S$-irreducible dual pairs in $E_8$. 

\newpage

\begin{center}
\begin{longtable}[H]{| c | c | c |} \hline
Max-Rank Reg.~Subalgebra & Dual Pair & How to Verify \\ \hline
$D_8$ & $(A_1^2, B_6^1)$ & \cite[p.~405]{Carter}, Lemma \ref{lem:dimension-of-cent}/\cite[Table 25]{dynkin} \\ 
 & $(A_1^6, A_1^6 \oplus B_3^1)$ & \cite[p.~405]{Carter}, $(B_3^1, B_4^1)$, Table \ref{table:dps-in-Bn} \\
 & $(A_1^{10''}, B_2^{3''})$ & \cite[p.~405]{Carter}, Remark \ref{rmk:rep-theory-elims} \\
 & $(B_2^1, B_5^1)$ & Lemma \ref{lem:dimension-of-cent}/\cite[Table 25]{dynkin} \\
 & $(B_3^1, B_4^1)$ & Example \ref{ex:E8-B3-B4} \\
 & $(B_2^{2''}, B_2^{2''})$ & Remark \ref{rmk:rep-theory-elims} \\ 
 & $(C_4^{1''}, A_1^{4'})$ & \cite[p.~405]{Carter}, Remark \ref{rmk:rep-theory-elims} \\ \hline 
$E_7 \oplus A_1$ & $(E_7, A_1)$ & Proposition \ref{prop:maximal-reg-is-dp} \\
 & $(C_3^{1}, G_2^1 \oplus A_1)$ & Table \ref{table:S-irreducibles}, Lemma \ref{lem:dimension-of-cent}/\cite[Table 25]{dynkin} \\ 
 & $(A_1^{3}, F_4^1 \oplus A_1)$ & Table \ref{table:S-irreducibles}, \cite[p.~405]{Carter} \\
 & $(G_2^2, A_1^7 \oplus A_1)$ & \cite[Section 6.10]{rubenthaler} \\
 & $(A_1^7, G_2^2 \oplus A_1)$ & Table \ref{table:S-irreducibles}, \cite[p.~405]{Carter} \\
 & $(A_1^{24}, A_1^{15} \oplus A_1)$ & \cite[Section 6.10]{rubenthaler} \\
 & $(A_1^{15}, A_1^{24} \oplus A_1)$ & Table \ref{table:S-irreducibles}, \cite[p.~405]{Carter} \\
 & $(G_2^1 \oplus A_1^{3}, A_1^8 \oplus A_1)$ & $(A_1^3, F_4^1 \oplus A_1)$, Table \ref{table:S-irreducibles}, \cite[p.~403]{Carter}  \\ \hline
$E_6 \oplus A_2$ & $(E_6, A_2)$ & Proposition \ref{prop:maximal-reg-is-dp} \\
 & $(A_2^{2}, G_2^1 \oplus A_2)$ & Table \ref{table:S-irreducibles}, Lemma \ref{lem:dimension-of-cent}/\cite[Table 25]{dynkin} \\ \hline 
$2A_4$ & $(A_4, A_4)$ & Proposition \ref{prop:maximal-reg-is-dp} \\ \hline
$D_5 \oplus A_3$ & $(D_5, A_3)$ & Lemma \ref{lem:dimension-of-cent}/\cite[Table 25]{dynkin} \\ \hline
$A_5 \oplus A_2 \oplus A_1$ & $(A_5, A_2 \oplus A_1)$ & Table \ref{table:E7}, Lemma \ref{lem:dimension-of-cent}/\cite[Table 25]{dynkin} \\ 
 & $(A_1^3 \oplus A_2, A_2^2 \oplus A_1)$ & Tables \ref{table:E6} \& \ref{table:E7} \\ \hline
$D_6 \oplus 2A_1$ & $(D_6, 2A_1)$ & Table \ref{table:E7}, Lemma \ref{lem:dimension-of-cent}/\cite[Table 25]{dynkin} \\ 
 & $(B_4^1 \oplus A_1, A_1^2 \oplus A_1)$ & Table \ref{table:E7} \\
 & $(B_3^1 \oplus A_1, B_2^1 \oplus A_1)$ & Table \ref{table:E7} \\
 & $(A_1^6 \oplus A_1^2, A_1^6 \oplus 2A_1)$ & $(A_1^2, B_6^1)$, Tables \ref{table:dps-in-Bn} \& \ref{table:dps-in-Dn} \\
 & $(A_1^6 \oplus A_1^2 \oplus A_1, A_1^6 \oplus A_1)$ & Table \ref{table:E7} \\ 
 & $(A_1^8 \oplus 2A_1, 2A_1^3)$ & $(A_1^3, F_4^1 \oplus A_1)$, Tables \ref{table:F4} \& \ref{table:dps-in-Dn} \\
 & $((A_1^3 \oplus A_1)^{'}, (C_3^1 \oplus A_1)^{'})$ & \cite[p.~403]{Carter}, \cite[Table 25]{dynkin} \\ \hline
$4A_2$ & $(2 A_2, 2A_2)$ & Table \ref{table:E6} \\ \hline 
 $2D_4$ & $(D_4, D_4)$ & Lemma \ref{lem:dimension-of-cent}/\cite[Table 25]{dynkin} \\ 
 & $( (B_2^1 \oplus A_1^2)^{'}, (B_2^1 \oplus A_1^2)^{'} )$ & $(A_1^2, B_6^1)$, $(B_2^1, B_5^1)$, Table \ref{table:B2-subalgebras} \\ \hline
$D_4 \oplus 4A_1$ & $(D_4 \oplus A_1, 3A_1)$ & Tables \ref{table:E7} \& \ref{table:dps-in-Dn} \\ \hline \caption{A complete list of the non-$S$-irreducible dual pairs in $E_8$. Here, $(A_1^3 \oplus A_1)^{'}$ and $(C_3^1 \oplus A_1)^{'}$ refer to the subalgebras of $E_8$ that restrict to $A_1^{3'}$ and $C_3^{1'}$, respectively, in $E_7 \oplus A_1$. Similarly, $(B_2^1 \oplus A_1^2)^{'}$ restricts to $B_2^{1'}$ in $B_6$ and to $A_1^{2'}$ in $B_5$.} \label{table:E8}
\end{longtable}
\end{center}

\section*{Acknowledgements}

The author would like to thank David Vogan for suggesting this topic of study and for his helpful comments on the manuscript.

\bibliography{biblio} 
\bibliographystyle{plain}

\newpage

\section*{Appendix A: Subalgebras of type $A_1$}

In this appendix, we describe all of the conjugacy classes of subalgebras of type $A_1$ in the complex simple Lie algebras of rank up to 6.\footnote{Much of this information is also available in \cite{S-subalgebras}, which was used as a reference when compiling these tables. However, while \cite{S-subalgebras} is concerned with $O(n,\C)$ conjugacy classes, we are concerned with $SO(n,\C)$ conjugacy classes. For example, the subalgebras $A_1^{2'}$ and $A_1^{2''}$ of $D_4$ appear in two non-conjugate dual pairs of $SO(8,\C)$ (see Table \ref{table:dps-in-Dn}) but they are not listed as distinct conjugacy classes in \cite{S-subalgebras}.}

\begin{multicols}{2}
\subsection*{Subalgebras of type $A_1$ in $A_2$}

\begin{center}
    \begin{tabular}{c|c|c}
        Index & Defining Vector & Diagram  \\ \hline
        $1$ & $[1,0]$ & $[1,1]$ \\
        4 & $[2,0]$ & $[2,2]$ 
    \end{tabular}
\end{center}

\subsection*{Subalgebras of type $A_1$ in $A_3$}

\begin{center}
    \begin{tabular}{c|l|c}
        Index & Defining Vector & Diagram  \\ \hline
        $1$ & \;\;\;\;\;\;$[1,0,0]$ & $[1,0,1]$ \\
        2 & \;\;\;\;\;\;$[1,1,-1]$ & $[0,2,0]$ \\
        4 & \;\;\;\;\;\;$[2,0,0]$ & $[2,0,2]$ \\
        10 & \;\;\;\;\;\;$[3,1,-1]$ & $[2,2,2]$
    \end{tabular}
\end{center}

\subsection*{Subalgebras of type $A_1$ in $A_4$}

\begin{center}
    \begin{tabular}{c|l|c}
        Index & Defining Vector & Diagram  \\ \hline
        $1$ & \;\;\;\;$[1,0,0,0]$ & $[1,0,0,1]$ \\
        2 & \;\;\;\;$[1,1,0,-1]$ & $[0,1,1,0]$ \\
        4 & \;\;\;\;$[2,0,0,0]$ & $[2,0,0,2]$ \\
        5 & \;\;\;\;$[2,1,0,-1]$ & $[1,1,1,1]$ \\
        10 & \;\;\;\;$[3,1,0,-1]$ & $[2,1,1,2]$ \\
        20 & \;\;\;\;$[4,2,0,-2]$ & $[2,2,2,2]$
    \end{tabular}
\end{center}

\vspace{1cm}

\;

\columnbreak

\subsection*{Subalgebras of type $A_1$ in $A_5$}

\begin{center}
    \begin{tabular}{c|l|c}
        Index & Defining Vector & Diagram  \\ \hline
        $1$ & \;$[1,0,0,0,0]$ & $[1,0,0,0,1]$ \\
        2 & \;$[1,1,0,0,-1]$ & $[0,1,0,1,0]$ \\
        3 & \;$[1,1,1,-1,-1]$ & $[0,0,2,0,0]$ \\
        4 & \;$[2,0,0,0,0]$ & $[2,0,0,0,2]$ \\
        5 & \;$[2,1,0,0,-1]$ & $[1,1,0,1,1]$ \\
        8 & \;$[2,2,0,0,-2]$ & $[0,2,0,2,0]$ \\
        10 & \;$[3,1,0,0,-1]$ & $[2,1,0,1,2]$ \\
        11 & \;$[3,1,1,-1,-1]$ & $[2,0,2,0,2]$ \\
        20 & \;$[4,2,0,0,-2]$ & $[2,2,0,2,2]$ \\
        35 & \;$[5,3,1,-1,-3]$ & $[2,2,2,2,2]$
    \end{tabular}
\end{center}

\subsection*{Subalgebras of type $A_1$ in $A_6$}

\begin{center}
    \begin{tabular}{c|l|c}
        Index & \;Defining Vector & Diagram  \\ \hline
        $1$ & $[1,0,0,0,0,0]$ & $[1,0,0,0,0,1]$ \\
         2 & $[1,1,0,0,0,-1]$ & $[0,1,0,0,1,0]$ \\
         3 & $[1,1,1,0,-1,-1]$ & $[0,0,1,1,0,0]$ \\
         4 & $[2,0,0,0,0,0]$ & $[2,0,0,0,0,2]$ \\
         5 & $[2,1,0,0,0,-1]$ & $[1,1,0,0,1,1]$ \\
         6 & $[2,1,1,0,-1,-1]$ & $[1,0,1,1,0,1]$ \\
         8 & $[2,2,0,0,0,-2]$ & $[0,2,0,0,2,0]$ \\
         10 & $[3,1,0,0,0,-1]$ & $[2,1,0,0,1,2]$ \\
         11 & $[3,1,1,0,-1,-1]$ & $[2,0,1,1,0,2]$ \\
         14 & $[3,2,1,0,-1,-2]$ & $[1,1,1,1,1,1]$ \\
         20 & $[4,2,0,0,0,-2]$ & $[2,2,0,0,2,2]$ \\
         21 & $[4,2,1,0,-1,-2]$ & $[2,1,1,1,1,2]$ \\
         35 & $[5,3,1,0,-1,-3]$ & $[2,2,1,1,2,2]$ \\
         56 & $[6,4,2,0,-2,-4]$ & $[2,2,2,2,2,2]$
    \end{tabular}
\end{center}
\end{multicols}

\subsection*{Subalgebras of type $A_1$ in $B_2$}

\begin{multicols}{2}
\begin{center}
    \begin{tabular}{c|c|c}
        Index & Defining Vector & Diagram  \\ \hline
        $1$ & $[1,1]$ & $[0,1]$ \\
         2 & $[2,0]$ & $[2,0]$ \\
         10 & $[4,2]$ & $[2,2]$ 
    \end{tabular}
\end{center}

\vspace{1cm}

\subsection*{Subalgebras of type $A_1$ in $B_3$}

\begin{center}
    \begin{tabular}{c|c|c}
        Index & Defining Vector & Diagram  \\ \hline
        $1$ & $[1,1,0]$ & $[0,1,0]$ \\
         2 & $[2,0,0]$ & $[2,0,0]$ \\
         3 & $[2,1,1]$ & $[1,0,1]$ \\
         4 & $[2,2,0]$ & $[0,2,0]$ \\
         10 & $[4,2,0]$ & $[2,2,0]$ \\
         28 & $[6,4,2]$ & $[2,2,2]$
    \end{tabular}
\end{center}

\vspace{1cm}

\subsection*{Subalgebras of type $A_1$ in $B_4$}

\begin{center}
    \begin{tabular}{l|c|c}
        Index & Defining Vector & Diagram  \\ \hline
        \;\;\;$1$ & $[1,1,0,0]$ & $[0,1,0,0]$ \\
        \;\;\;$2'$ & $[1,1,1,1]$ & $[0,0,0,1]$ \\
        \;\;\;$2''$ & $[2,0,0,0]$ & $[2,0,0,0]$ \\
         \;\;\;3 & $[2,1,1,0]$ & $[1,0,1,0]$ \\
         \;\;\;4 & $[2,2,0,0]$ & $[0,2,0,0]$ \\
         \;\;\;6 & $[2,2,2,0]$ & $[0,0,2,0]$ \\
         \;\;$10'$ & $[3,3,1,1]$ & $[0,2,0,1]$ \\
         \;\;$10''$ & $[4,2,0,0]$ & $[2,2,0,0]$ \\
         \;\;11 & $[4,2,1,1]$ & $[2,1,0,1]$ \\
         \;\;12 & $[4,2,2,0]$ & $[2,0,2,0]$ \\
         \;\;28 & $[6,4,2,0]$ & $[2,2,2,0]$ \\
         \;\;60 & $[8,6,4,2]$ & $[2,2,2,2]$
    \end{tabular}
\end{center}

\columnbreak

\subsection*{Subalgebras of type $A_1$ in $B_5$}

\begin{center}
    \begin{tabular}{l|c|c}
        Index & Defining Vector & Diagram  \\ \hline
        \;\;\;$1$ & $[1,1,0,0,0]$ & $[0,1,0,0,0]$ \\
        \;\;\;$2'$ & $[1,1,1,1,0]$ & $[0,0,0,1,0]$ \\
        \;\;\;$2''$ & $[2,0,0,0,0]$ & $[2,0,0,0,0]$ \\
        \;\;\;3 & $[2,1,1,0,0]$ & $[1,0,1,0,0]$ \\
        \;\;\;$4'$ & $[2,1,1,1,1]$ & $[1,0,0,0,1]$ \\
        \;\;\;$4''$ & $[2,2,0,0,0]$ & $[0,2,0,0,0]$ \\
        \;\;\;5 & $[2,2,1,1,0]$ & $[0,1,0,1,0]$ \\
        \;\;\;6 & $[2,2,2,0,0]$ & $[0,0,2,0,0]$ \\
        \;\;$10'$ & $[3,3,1,1,0]$ & $[0,2,0,1,0]$ \\
        \;\;$10''$ & $[4,2,0,0,0]$ & $[2,2,0,0,0]$ \\
        \;\;11 & $[4,2,1,1,0]$ & $[2,1,0,1,0]$ \\
        \;\;$12'$ & $[3,3,2,1,1]$ & $[0,1,1,0,1]$ \\ 
        \;\;$12''$ & $[4,2,2,0,0]$ & $[2,0,2,0,0]$ \\  
        \;\;14 & $[4,2,2,2,0]$ & $[2,0,0,2,0]$ \\
        \;\;20 & $[4,4,2,2,0]$ & $[0,2,0,2,0]$ \\
        \;\;28 & $[6,4,2,0,0]$ & $[2,2,2,0,0]$ \\
        \;\;29 & $[6,4,2,1,1]$ & $[2,2,1,0,1]$ \\
        \;\;30 & $[6,4,2,2,0]$ & $[2,2,0,2,0]$ \\
        \;\;60 & $[8,6,4,2,0]$ & $[2,2,2,2,0]$ \\
        \;110 & $[10,8,6,4,2]$ & $[2,2,2,2,2]$ 
    \end{tabular}
\end{center}
\end{multicols}

\newpage

\subsection*{Subalgebras of type $A_1$ in $B_6$}

\begin{center}
    \begin{tabular}{l|c|c||l|c|c}
Index & Defining Vector & Diagram & Index & Defining Vector & Diagram \\ \hline
\;\;\;1 & $[ 1, 1, 0, 0, 0, 0 ]$ & $ [0,1,0,0,0,0]$ & \;\;$12'''$ & $[ 4, 2, 2, 0, 0, 0 ]$ & $ [2,0,2,0,0,0]$ \\
\;\;\;$2'$ & $[ 1, 1, 1, 1, 0, 0 ]$ & $ [0,0,0,1,0,0]$ & \;\;13 & $[ 4, 2, 2, 1, 1, 0 ]$ & $ [2,0,1,0,1,0]$ \\
\;\;\;$2''$ & $[ 2, 0, 0, 0, 0, 0 ]$ & $ [2,0,0,0,0,0]$ & \;\;14 & $[ 4, 2, 2, 2, 0, 0 ]$ & $ [2,0,0,2,0,0]$ \\
\;\;\;$3'$ & $[ 1, 1, 1, 1, 1, 1 ]$ & $ [0,0,0,0,0,1]$ & \;\;$20'$ & $[ 4, 3, 3, 2, 1, 1 ]$ & $ [1,0,1,1,0,1]$ \\
\;\;\;$3''$ & $[ 2, 1, 1, 0, 0, 0 ]$ & $ [1,0,1,0,0,0]$ & \;\;$20''$ & $[ 4, 4, 2, 2, 0, 0 ]$ & $ [0,2,0,2,0,0]$ \\
\;\;\;$4'$ & $[ 2, 1, 1, 1, 1, 0 ]$ & $ [1,0,0,0,1,0]$ & \;\;22 & $[ 4, 4, 2, 2, 2, 0 ]$ & $ [0,2,0,0,2,0]$ \\
\;\;\;$4''$ & $[ 2, 2, 0, 0, 0, 0 ]$ & $ [0,2,0,0,0,0]$ & \;\;28 & $[ 6, 4, 2, 0, 0, 0 ]$ & $ [2,2,2,0,0,0]$ \\
\;\;\;5 & $[ 2, 2, 1, 1, 0, 0 ]$ & $ [0,1,0,1,0,0]$ & \;\;29 & $[ 6, 4, 2, 1, 1, 0 ]$ & $ [2,2,1,0,1,0]$ \\
\;\;\;6 & $[ 2, 2, 2, 0, 0, 0 ]$ & $ [0,0,2,0,0,0]$ & \;\;30 & $[ 6, 4, 2, 2, 0, 0 ]$ & $ [2,2,0,2,0,0]$ \\
\;\;\;7 & $[ 2, 2, 2, 1, 1, 0 ]$ & $ [0,0,1,0,1,0]$ & \;\;32 & $[ 6, 4, 2, 2, 2, 0 ]$ & $ [2,2,0,0,2,0]$ \\
\;\;\;8 & $[ 2, 2, 2, 2, 0, 0 ]$ & $ [0,0,0,2,0,0]$ & \;\;35 & $[ 5, 5, 3, 3, 1, 1 ]$ & $ [0,2,0,2,0,1]$ \\
\;\;$10'$ & $[ 3, 3, 1, 1, 0, 0 ]$ & $ [0,2,0,1,0,0]$ & \;\;38 & $[ 6, 4, 4, 2, 2, 0 ]$ & $ [2,0,2,0,2,0]$ \\
\;\;$10''$ & $[ 4, 2, 0, 0, 0, 0 ]$ & $ [2,2,0,0,0,0]$ & \;\;60 & $[ 8, 6, 4, 2, 0, 0 ]$ & $ [2,2,2,2,0,0]$ \\
\;\;$11'$ & $[ 3, 3, 1, 1, 1, 1 ]$ & $ [0,2,0,0,0,1]$ & \;\;61 & $[ 8, 6, 4, 2, 1, 1 ]$ & $ [2,2,2,1,0,1]$ \\
\;\;$11''$ & $[ 4, 2, 1, 1, 0, 0 ]$ & $ [2,1,0,1,0,0]$ & \;\;62 & $[ 8, 6, 4, 2, 2, 0 ]$ & $ [2,2,2,0,2,0]$ \\
\;\;$12'$ & $[ 3, 3, 2, 1, 1, 0 ]$ & $ [0,1,1,0,1,0]$ & \;110 & $[ 10,  8,  6,  4,  2,  0 ]$ & $ [2,2,2,2,2,0]$ \\
\;\;$12''$ & $[ 4, 2, 1, 1, 1, 1 ]$ & $ [2,1,0,0,0,1]$ & \;182 & $[ 12, 10,  8,  6,  4,  2 ]$ & $ [2,2,2,2,2,2]$ 
    \end{tabular}
\end{center}

\newpage

\begin{multicols}{2}
    
\subsection*{Subalgebras of type $A_1$ in $C_2$}

\begin{center}
    \begin{tabular}{c|c|c}
Index & Defining Vector & Diagram \\ \hline
1 & $[ 1, 0 ]$ & $ [1,0]$ \\
2 & $[ 1, 1 ]$ & $ [0,2]$ \\
10 & $[ 3, 1 ]$ & $ [2,2]$ 
    \end{tabular}
\end{center}

\vspace{.75cm}

\subsection*{Subalgebras of type $A_1$ in $C_3$}

\begin{center}
    \begin{tabular}{c|c|c}
Index & Defining Vector & Diagram \\ \hline
1 & $[ 1, 0, 0 ]$ & $ [1,0,0]$ \\
2 & $[ 1, 1, 0 ]$ & $ [0,1,0]$ \\
3 & $[ 1, 1, 1 ]$ & $ [0,0,2]$ \\
8 & $[ 2, 2, 0 ]$ & $ [0,2,0]$ \\
10 & $[ 3, 1, 0 ]$ & $ [2,1,0]$ \\
11 & $[ 3, 1, 1 ]$ & $ [2,0,2]$ \\
35 & $[ 5, 3, 1 ]$ & $ [2,2,2]$ 
    \end{tabular}
\end{center}

\vspace{.75cm}

\subsection*{Subalgebras of type $A_1$ in $C_4$}

\begin{center}
    \begin{tabular}{c|c|c}
Index & Defining Vector & Diagram \\ \hline
1 & $[ 1, 0, 0, 0 ]$ & $ [1,0,0,0]$ \\
2 & $[ 1, 1, 0, 0 ]$ & $ [0,1,0,0]$ \\
3 & $[ 1, 1, 1, 0 ]$ & $ [0,0,1,0]$ \\
4 & $[ 1, 1, 1, 1 ]$ & $ [0,0,0,2]$ \\
8 & $[ 2, 2, 0, 0 ]$ & $ [0,2,0,0]$ \\
9 & $[ 2, 2, 1, 0 ]$ & $ [0,1,1,0]$ \\
10 & $[ 3, 1, 0, 0 ]$ & $ [2,1,0,0]$ \\
11 & $[ 3, 1, 1, 0 ]$ & $ [2,0,1,0]$ \\
12 & $[ 3, 1, 1, 1 ]$ & $ [2,0,0,2]$ \\
20 & $[ 3, 3, 1, 1 ]$ & $ [0,2,0,2]$ \\
35 & $[ 5, 3, 1, 0 ]$ & $ [2,2,1,0]$ \\
36 & $[ 5, 3, 1, 1 ]$ & $ [2,2,0,2]$ \\
84 & $[ 7, 5, 3, 1 ]$ & $ [2,2,2,2]$ \\
    \end{tabular}
\end{center}

\columnbreak

\subsection*{Subalgebras of type $A_1$ in $C_5$}

\begin{center}
    \begin{tabular}{l|c|c}
Index & Defining Vector & Diagram \\ \hline
\;\;\;1 & $[ 1, 0, 0, 0, 0 ]$ & $ [1,0,0,0,0]$ \\
\;\;\;2 & $[ 1, 1, 0, 0, 0 ]$ & $ [0,1,0,0,0]$ \\
\;\;\;3 & $[ 1, 1, 1, 0, 0 ]$ & $ [0,0,1,0,0]$ \\
\;\;\;4 & $[ 1, 1, 1, 1, 0 ]$ & $ [0,0,0,1,0]$ \\
\;\;\;5 & $[ 1, 1, 1, 1, 1 ]$ & $ [0,0,0,0,2]$ \\
\;\;\;8 & $[ 2, 2, 0, 0, 0 ]$ & $ [0,2,0,0,0]$ \\
\;\;\;9 & $[ 2, 2, 1, 0, 0 ]$ & $ [0,1,1,0,0]$ \\
\;\;\;$10'$ & $[ 2, 2, 1, 1, 0 ]$ & $ [0,1,0,1,0]$ \\
\;\;\;$10''$ & $[ 3, 1, 0, 0, 0 ]$ & $ [2,1,0,0,0]$ \\
\;\;\;11 & $[ 3, 1, 1, 0, 0 ]$ & $ [2,0,1,0,0]$ \\
\;\;\;12 & $[ 3, 1, 1, 1, 0 ]$ & $ [2,0,0,1,0]$ \\
\;\;\;13 & $[ 3, 1, 1, 1, 1 ]$ & $ [2,0,0,0,2]$ \\
\;\;\;18 & $[ 3, 2, 2, 1, 0 ]$ & $ [1,0,1,1,0]$ \\
\;\;\;20 & $[ 3, 3, 1, 1, 0 ]$ & $ [0,2,0,1,0]$ \\
\;\;\;21 & $[ 3, 3, 1, 1, 1 ]$ & $ [0,2,0,0,2]$ \\
\;\;\;35 & $[ 5, 3, 1, 0, 0 ]$ & $ [2,2,1,0,0]$ \\
\;\;\;36 & $[ 5, 3, 1, 1, 0 ]$ & $ [2,2,0,1,0]$ \\
\;\;\;37 & $[ 5, 3, 1, 1, 1 ]$ & $ [2,2,0,0,2]$ \\
\;\;\;40 & $[ 4, 4, 2, 2, 0 ]$ & $ [0,2,0,2,0]$ \\
\;\;\;45 & $[ 5, 3, 3, 1, 1 ]$ & $ [2,0,2,0,2]$ \\
\;\;\;84 & $[ 7, 5, 3, 1, 0 ]$ & $ [2,2,2,1,0]$ \\
\;\;\;85 & $[ 7, 5, 3, 1, 1 ]$ & $ [2,2,2,0,2]$ \\
\;\;165 & $[ 9, 7, 5, 3, 1 ]$ & $ [2,2,2,2,2]$ \\
    \end{tabular}
\end{center}

\end{multicols}

\newpage

\subsection*{Subalgebras of type $A_1$ in $C_6$}

\begin{center}
    \begin{tabular}{l|c|c||l|c|c}
        Index & Defining Vector & Diagram & Index & Defining Vector & Diagram \\ \hline
\;\;\;1 & $[ 1, 0, 0, 0, 0, 0 ]$ & $ [1,0,0,0,0,0]$ & \;\;22 & $[ 3, 3, 1, 1, 1, 1 ]$ & $ [0,2,0,0,0,2]$ \\
\;\;\;2 & $[ 1, 1, 0, 0, 0, 0 ]$ & $ [0,1,0,0,0,0]$ & \;\;30 & $[ 3, 3, 3, 1, 1, 1 ]$ & $ [0,0,2,0,0,2]$ \\
\;\;\;3 & $[ 1, 1, 1, 0, 0, 0 ]$ & $ [0,0,1,0,0,0]$ & \;\;35 & $[ 5, 3, 1, 0, 0, 0 ]$ & $ [2,2,1,0,0,0]$ \\
\;\;\;4 & $[ 1, 1, 1, 1, 0, 0 ]$ & $ [0,0,0,1,0,0]$ & \;\;36 & $[ 5, 3, 1, 1, 0, 0 ]$ & $ [2,2,0,1,0,0]$ \\
\;\;\;5 & $[ 1, 1, 1, 1, 1, 0 ]$ & $ [0,0,0,0,1,0]$ & \;\;37 & $[ 5, 3, 1, 1, 1, 0 ]$ & $ [2,2,0,0,1,0]$ \\
\;\;\;6 & $[ 1, 1, 1, 1, 1, 1 ]$ & $ [0,0,0,0,0,2]$ & \;\;38 & $[ 5, 3, 1, 1, 1, 1 ]$ & $ [2,2,0,0,0,2]$ \\
\;\;\;8 & $[ 2, 2, 0, 0, 0, 0 ]$ & $ [0,2,0,0,0,0]$ & \;\;40 & $[ 4, 4, 2, 2, 0, 0 ]$ & $ [0,2,0,2,0,0]$ \\
\;\;\;9 & $[ 2, 2, 1, 0, 0, 0 ]$ & $ [0,1,1,0,0,0]$ & \;\;41 & $[ 4, 4, 2, 2, 1, 0 ]$ & $ [0,2,0,1,1,0]$ \\
\;\;$10'$ & $[ 2, 2, 1, 1, 0, 0 ]$ & $ [0,1,0,1,0,0]$ & \;\;43 & $[ 5, 3, 2, 2, 1, 0 ]$ & $ [2,1,0,1,1,0]$ \\
\;\;$10''$ & $[ 3, 1, 0, 0, 0, 0 ]$ & $ [2,1,0,0,0,0]$ & \;\;45 & $[ 5, 3, 3, 1, 1, 0 ]$ & $ [2,0,2,0,1,0]$ \\
\;\;$11'$ & $[ 2, 2, 1, 1, 1, 0 ]$ & $ [0,1,0,0,1,0]$ & \;\;46 & $[ 5, 3, 3, 1, 1, 1 ]$ & $ [2,0,2,0,0,2]$ \\
\;\;$11''$ & $[ 3, 1, 1, 0, 0, 0 ]$ & $ [2,0,1,0,0,0]$ & \;\;70 & $[ 5, 5, 3, 3, 1, 1 ]$ & $ [0,2,0,2,0,2]$ \\
\;\;12 & $[ 3, 1, 1, 1, 0, 0 ]$ & $ [2,0,0,1,0,0]$ & \;\;84 & $[ 7, 5, 3, 1, 0, 0 ]$ & $ [2,2,2,1,0,0]$ \\
\;\;13 & $[ 3, 1, 1, 1, 1, 0 ]$ & $ [2,0,0,0,1,0]$ & \;\;85 & $[ 7, 5, 3, 1, 1, 0 ]$ & $ [2,2,2,0,1,0]$ \\
\;\;14 & $[ 3, 1, 1, 1, 1, 1 ]$ & $ [2,0,0,0,0,2]$ & \;\;86 & $[ 7, 5, 3, 1, 1, 1 ]$ & $ [2,2,2,0,0,2]$ \\
\;\;16 & $[ 2, 2, 2, 2, 0, 0 ]$ & $ [0,0,0,2,0,0]$ & \;\;94 & $[ 7, 5, 3, 3, 1, 1 ]$ & $ [2,2,0,2,0,2]$ \\
\;\;18 & $[ 3, 2, 2, 1, 0, 0 ]$ & $ [1,0,1,1,0,0]$ & \;165 & $[ 9, 7, 5, 3, 1, 0 ]$ & $ [2,2,2,2,1,0]$ \\
\;\;19 & $[ 3, 2, 2, 1, 1, 0 ]$ & $ [1,0,1,0,1,0]$ & \;166 & $[ 9, 7, 5, 3, 1, 1 ]$ & $ [2,2,2,2,0,2]$ \\
\;\;20 & $[ 3, 3, 1, 1, 0, 0 ]$ & $ [0,2,0,1,0,0]$ & \;286 & $[ 11,  9,  7,  5,  3,  1 ]$ & $ [2,2,2,2,2,2]$ \\
\;\;21 & $[ 3, 3, 1, 1, 1, 0 ]$ & $ [0,2,0,0,1,0]$ & & & \\
    \end{tabular}
\end{center}

\newpage

\begin{multicols}{2}

\subsection*{Subalgebras of type $A_1$ in $D_2$}

\begin{center}
    \begin{tabular}{l|c|c}
        Index & Defining Vector & Diagram  \\ \hline
        \;\;\;$1'$ & $[1,-1]$ & $[2,0]$ \\
        \;\;\;$1''$ & $[1,1]$ & $[0,2]$ \\
        \;\;\;2 & $[2,0]$ & $[2,2]$
    \end{tabular}
\end{center}

\vspace{1cm}

\subsection*{Subalgebras of type $A_1$ in $D_3$}

\begin{center}
    \begin{tabular}{c|c|c}
Index & Defining Vector & Diagram  \\ \hline
1 & $[ 1, 1, 0 ]$ & $ [0,1,1]$ \\
2 & $[ 2, 0, 0 ]$ & $ [2,0,0]$ \\
4 & $[ 2, 2, 0 ]$ & $ [0,2,2]$ \\
10 & $[ 4, 2, 0 ]$ & $ [2,2,2]$ \\
    \end{tabular}
\end{center}

\vspace{1cm}

\subsection*{Subalgebras of type $A_1$ in $D_4$}

\begin{center}
    \begin{tabular}{l|l|c}
Index & Defining Vector & Diagram  \\ \hline
\;\;\;1 & \;\;\;\;\,$[ 1, 1, 0, 0 ]$ & $ [0,1,0,0]$ \\
\;\;\;$2'$ & \;\;\;\;\,$[  1,  1,  1, -1 ]$ & $ [0,0,2,0]$ \\
\;\;\;$2''$ & \;\;\;\;\,$[ 1, 1, 1, 1 ]$ & $ [0,0,0,2]$ \\
\;\;\;$2'''$ & \;\;\;\;\,$[ 2, 0, 0, 0 ]$ & $ [2,0,0,0]$ \\
\;\;\;3 & \;\;\;\;\,$[ 2, 1, 1, 0 ]$ & $ [1,0,1,1]$ \\
\;\;\;4 & \;\;\;\;\,$[ 2, 2, 0, 0 ]$ & $ [0,2,0,0]$ \\
\;\;$10'$ & \;\;\;\;\,$[  3,  3,  1, -1 ]$ & $ [0,2,2,0]$ \\
\;\;$10''$ & \;\;\;\;\,$[ 3, 3, 1, 1 ]$ & $ [0,2,0,2]$ \\
\;\;$10'''$ & \;\;\;\;\,$[ 4, 2, 0, 0 ]$ & $ [2,2,0,0]$ \\
\;\;12 & \;\;\;\;\,$[ 4, 2, 2, 0 ]$ & $ [2,0,2,2]$ \\
\;\;28 & \;\;\;\;\,$[ 6, 4, 2, 0 ]$ & $ [2,2,2,2]$ \\
    \end{tabular}
\end{center}

\columnbreak

\subsection*{Subalgebras of type $A_1$ in $D_5$}

\begin{center}
    \begin{tabular}{l|c|c}
Index & Defining Vector & Diagram  \\ \hline
\;\;\;1 & $[ 1, 1, 0, 0, 0 ]$ & $ [0,1,0,0,0]$ \\
\;\;\;$2'$ & $[ 1, 1, 1, 1, 0 ]$ & $ [0,0,0,1,1]$ \\
\;\;\;$2''$ & $[ 2, 0, 0, 0, 0 ]$ & $ [2,0,0,0,0]$ \\
\;\;\;3 & $[ 2, 1, 1, 0, 0 ]$ & $ [1,0,1,0,0]$ \\
\;\;\;4 & $[ 2, 2, 0, 0, 0 ]$ & $ [0,2,0,0,0]$ \\
\;\;\;5 & $[ 2, 2, 1, 1, 0 ]$ & $ [0,1,0,1,1]$ \\
\;\;\;6 & $[ 2, 2, 2, 0, 0 ]$ & $ [0,0,2,0,0]$ \\
\;\;$10'$ & $[ 3, 3, 1, 1, 0 ]$ & $ [0,2,0,1,1]$ \\
\;\;$10''$ & $[ 4, 2, 0, 0, 0 ]$ & $ [2,2,0,0,0]$ \\
\;\;11 & $[ 4, 2, 1, 1, 0 ]$ & $ [2,1,0,1,1]$ \\
\;\;12 & $[ 4, 2, 2, 0, 0 ]$ & $ [2,0,2,0,0]$ \\
\;\;20 & $[ 4, 4, 2, 2, 0 ]$ & $ [0,2,0,2,2]$ \\
\;\;28 & $[ 6, 4, 2, 0, 0 ]$ & $ [2,2,2,0,0]$ \\
\;\;30 & $[ 6, 4, 2, 2, 0 ]$ & $ [2,2,0,2,2]$ \\
\;\;60 & $[ 8, 6, 4, 2, 0 ]$ & $ [2,2,2,2,2]$ \\
    \end{tabular}
\end{center}

\end{multicols}

\newpage

\subsection*{Subalgebras of type $A_1$ in $D_6$}

\begin{center}
    \begin{tabular}{l|l|c||l|l|c}
        Index & Defining Vector & Diagram & Index & Defining Vector & Diagram \\ \hline
         \;\;\;1& \;$[1,1,0,0,0,0]$ & $[0,1,0,0,0,0]$ & \;\;\,$11'''$& \;$[4,2,1,1,0,0]$ & $[2,1,0,1,0,0]$ \\
         \;\;\;$2'$& \;$[1,1,1,1,0,0]$ & $[0,0,0,1,0,0]$ & \;\;\,$12'$& \;$[3,3,2,1,1,0]$ & $[0,1,1,0,1,1]$ \\
         \;\;\;$2''$& \;$[2,0,0,0,0,0]$ & $[2,0,0,0,0,0]$ &  \;\;\,$12''$& \;$[4,2,2,0,0,0]$ & $[2,0,2,0,0,0]$       \\
         \;\;\;$3'$& \;$[1,1,1,1,1,-1]$ & $[0,0,0,0,2,0]$ &  \;\;\,13& \;$[4,2,2,1,1,0]$ & $[2,0,1,0,1,1]$ \\
         \;\;\;$3''$& \;$[1,1,1,1,1,1]$ & $[0,0,0,0,0,2]$ &  \;\;\,14& \;$[4,2,2,2,0,0]$ & $[2,0,0,2,0,0]$ \\
         \;\;\;$3'''$& \;$[2,1,1,0,0,0]$ & $[1,0,1,0,0,0]$ &  \;\;\,20& \;$[4,4,2,2,0,0]$ & $[0,2,0,2,0,0]$ \\
          \;\;\;$4'$& \;$[2,1,1,1,1,0]$ & $[1,0,0,0,1,1]$ &  \;\;\,28& \;$[6,4,2,0,0,0]$ & $[2,2,2,0,0,0]$ \\
         \;\;\;$4''$& \;$[2,2,0,0,0,0]$ & $[0,2,0,0,0,0]$ &  \;\;\,29& \;$[6,4,2,1,1,0]$ & $[2,2,1,0,1,1]$ \\
         \;\;\;5& \;$[2,2,1,1,0,0]$ & $[0,1,0,1,0,0]$ & \;\;\,30& \;$[6,4,2,2,0,0]$ & $[2,2,0,2,0,0]$ \\
         \;\;\;6& \;$[2,2,2,0,0,0]$ & $[0,0,2,0,0,0]$ &  \;\;\,$35'$& \;$[5,5,3,3,1,-1]$ & $[0,2,0,2,2,0]$ \\
         \;\;\;8& \;$[2,2,2,2,0,0]$ & $[0,0,0,2,0,0]$ &   \;\;\,$35''$& \;$[5,5,3,3,1,1]$ & $[0,2,0,2,0,2]$ \\
         \;\;$10'$& \;$[3,3,1,1,0,0]$ & $[0,2,0,1,0,0]$ & \;\;\,38& \;$[6,4,4,2,2,0]$ & $[2,0,2,0,2,2]$ \\
         \;\;$10''$& \;$[4,2,0,0,0,0]$ & $[2,2,0,0,0,0]$ & \;\;\,60& \;$[8,6,4,2,0,0]$ & $[2,2,2,2,0,0]$ \\
          \;\;$11'$& \;$[3,3,1,1,1,-1]$ & $[0,2,0,0,2,0]$ &   \;\;\,62& \;$[8,6,4,2,2,0]$ & $[2,2,2,0,2,2]$ \\
          \;\;$11''$& \;$[3,3,1,1,1,1]$ & $[0,2,0,0,0,2]$ &  \;\,110& \;$[10,8,6,4,2,0]$ & $[2,2,2,2,2,2]$ \\
    \end{tabular}
\end{center}

\newpage

\section*{Appendix B: Subalgebras of ranks 2 and 3}

In this appendix, we describe all of the conjugacy classes of subalgebras of ranks 2 and 3 in the complex simple Lie algebras of rank up to 6.\footnote{Much of this information is also available in \cite{S-subalgebras}, which was used as a reference when compiling these tables. However, while \cite{S-subalgebras} is concerned with $O(n,\C)$ conjugacy classes, we are concerned with $SO(n,\C)$ conjugacy classes. For example, the subalgebras $B_2^{1'}$ and $B_2^{1''}$ of $D_4$ appear in two non-conjugate dual pairs of $SO(8,\C)$ (see Table \ref{table:dps-in-Dn}) but they are not listed as distinct conjugacy classes in \cite{S-subalgebras}.} 

Note that there are some examples where the same conjugacy class can be specified by multiple non-conjugate maps. For example, $A_2^1$ in $A_4$ can be obtained either as the image of $V_3 \oplus V_1^{\oplus 2}$ with Cartan embedding 
$$[x,y,-x-y] \mapsto [x,y,-x-y,0,0]$$ 
or as the image of $V_3^* \oplus V_1^{\oplus 2}$ with Cartan embedding 
$$[x,y,-x-y] \mapsto [-x,-y,x+y,0,0].$$ 
In these situations, we will include the conjugacy class twice and will include both maps in the table.

For the highest weights and Cartan subalgebras in this appendix, we will use the realizations from \cite[Plates I-IX]{bourbaki}.

\subsection*{Subalgebras of type $A_2$}

We denote the irreducible representations of $A_2$ of dimension at most 13 as follows:

\begin{center}
\begin{table}[H]
\begin{tabular}{c|l|c|c}
        Label & Highest Weight & Dimension & Symplectic or Orthogonal? \\ \hline
        $V_1$ & \;\;\;\,$[0,0,0]$ & 1 & Orthogonal \\
        $V_3$ & \;\;\;\,$\left[ \frac{2}{3} , -\frac{1}{3}, - \frac{1}{3} \right ]$ & 3 & Neither \\
        $V_3^*$ & \;\;\;\,$\left [ \frac{1}{3}, \frac{1}{3}, - \frac{2}{3} \right ]$ & 3 & Neither \\
        $V_{\text{Ad}}$ & \;\;\;\,$[1,0,-1]$ & 8 & Orthogonal \\
        $\Sym^2 V_3$ & \;\;\;\,$\left[ \frac{4}{3} , -\frac{2}{3}, - \frac{2}{3} \right ]$ & 6 & Neither \\ 
        $\Sym^2 V_3^*$ & \;\;\;\,$\left [ \frac{2}{3}, \frac{2}{3}, - \frac{4}{3} \right ]$ & 6 & Neither \\
        $\Sym^3 V_3$ & \;\;\;\,$\left[ 2 , -1 , -1 \right ]$ & 10 & Neither \\
        $\Sym^3 V_3^*$ & \;\;\;\,$\left[ 1, 1, -2 \right ]$ & 10 & Neither \\
    \end{tabular} \caption{Irreducible representations of $A_2$ of dimension at most 13.}
\end{table} 
\end{center}

Using this notation, the following table classifies all of the conjugacy classes of subalgebras $\mf{h}$ of type $A_2$ in $\mf{g}$ for various $\mf{g}$. 

\newpage

\begin{center}
\begin{longtable}[H]{c|l|c|l}
        $\mf{g}$ & Dynkin & Standard rep.~of & \hspace{1cm}Cartan embedding \vspace{-.5em} \\ 
        & \;\,index & $\mf{g}$ restricted to $\mf{h}$ &   \\ \hline 
        $A_3 $ & \;\;\;\;\,$1'$  & $V_3 \oplus V_1$ & $[x,y,-x-y] \mapsto [x,y,-x-y,0]$  \\ 
         & \;\;\;\;\,$1''$  & $V_3^* \oplus V_1$ & $[x,y,-x-y] \mapsto [-x,-y,x+y,0]$ \\ \hline 
        $B_3$ & \;\;\;\;\,1 & $V_3 \oplus V_3^*$ & $[x,y,-x-y] \mapsto [x,y,-x-y]$ \\ \hline 
        $C_3$ & \;\;\;\;\,2 & $V_3 \oplus V_3^*$ & $[x,y,-x-y] \mapsto [x,y,-x-y]$ \\ \hline 
       $A_4 $ & \;\;\;\;\,1  & \;$V_3 \oplus V_1^{\oplus 2}$ & $[x,y,-x-y] \mapsto [x,y,-x-y,0,0]$  \\ 
         & \;\;\;\;\,1 & \;$V_3^* \oplus V_1^{\oplus 2}$ & $[x,y,-x-y] \mapsto [-x,-y,x+y,0,0]$ \\
\hline
\end{longtable}
\end{center}

\vspace{-.5cm}

\begin{table}[H]
 \caption{Conjugacy classes of subalgebras of type $A_2$ in various $\mf{g}$.} \label{table:A2-subalgebras}
\end{table}

\subsection*{Subalgebras of type $B_2 \simeq C_2$}

We denote the irreducible representations of $B_2 \simeq C_2$ of dimension at most 13 as follows:

\begin{center}
\begin{table}[H]
\begin{tabular}{c|c|c|c|c}
        Label & Highest Weight & Highest Weight & Dimension & Symplectic or \vspace{-.5em} \\ 
         & ($B_2$ coords.) & ($C_2$ coords.) & & Orthogonal? \\ \hline 
        $V_1$ & $[0,0]$ & $[0,0]$ & 1 & Orthogonal \\
        $V_4$ & $\left[ \frac{1}{2}, \frac{1}{2} \right]$ & $[1,0]$ & 4 & Symplectic \\
        $V_5$ & $[1,0]$ & $[1,1]$ & 5 & Orthogonal \\
        $V_{10}$ & $[1,1]$ & $[2,0]$ & 10 & Orthogonal \\
\end{tabular}
\caption{Irreducible representations of $B_2 \simeq C_2$ of dimension at most 13.}
\end{table}
\end{center} 

Using this notation, the following table classifies all of the conjugacy classes of subalgebras $\mf{h}$ of type $B_2 \simeq C_2$ in $\mf{g}$ for various $\mf{g}$.

\small 

\begingroup 
\setlength{\LTleft}{0pt minus 1000pt}
\setlength{\LTright}{0pt minus 1000pt}

\begin{center}
\begin{longtable}[H]{c|l|c|c|c}
        $\mf{g}$ & Dynkin & Standard rep.~of & Cartan embedding & Cartan embedding \vspace{-.5em} \\ 
        & \;\,index & $\mf{g}$ restricted to $\mf{h}$ & ($B_2$ coords.) & ($C_2$ coords.) \\ \hline 
        $A_3 $ & \;\;\;\;\,1  & $V_4$ & $[x,y] \mapsto 
        \left [ \frac{x+y}{2}, \frac{x-y}{2}, -\frac{x+y}{2}, -\frac{x-y}{2} \right ]$  & $[p,q] \mapsto [p,q,-p,-q]$ \\ \hline 
         $B_3$ & \;\;\;\;\,1  & \;\;\,$V_5 \oplus V_1^{\oplus 2}$ & $[x,y] \mapsto [x,y,0]$ & $[p,q] \mapsto [p+q,p-q,0]$ \\ \hline  
        $C_3$ & \;\;\;\;\,1  & \;\;\,$V_4 \oplus V_1^{\oplus 2}$ & $[x,y] \mapsto \left [ \frac{x+y}{2}, \frac{x-y}{2}, 0 \right ]$ & $[p,q] \mapsto [p,q,0]$  \\ \hline 
        $D_3$ & \;\;\;\;\,1  & $V_5 \oplus V_1$ & $[x,y] \mapsto [x,y,0]$ & $[p,q] \mapsto [p+q,p-q,0]$ \\ \hline 
        $A_4$ & \;\;\;\;\,1  & $V_4 \oplus V_1$ & $[x,y] \mapsto \left [ \frac{x+y}{2}, \frac{x-y}{2}, -\frac{x+y}{2}, -\frac{x-y}{2}, 0\right ]$ & $[p,q] \mapsto [p,q,-p,-q,0]$ \\ 
        & \;\;\;\;\,2  & $V_5$ & $[x,y] \mapsto [x,y,-x,-y,0]$ & $[p,q] \mapsto [p+q,p-q,-p-q,q-p,0]$ \\
\end{longtable}
\end{center}

\endgroup

\normalsize

\vspace{-.5cm}

\begin{table}[H]
 \caption{Conjugacy classes of subalgebras of type $B_2 \simeq C_2$ in various $\mf{g}$.} \label{table:B2-subalgebras}
\end{table}

\subsection*{Subalgebras of type $G_2$}

We denote the irreducible representations of $G_2$ of dimension at most 13 as follows:

\begin{center}
\begin{table}[H]
\begin{tabular}{c|c|c|c}
        Label & Highest Weight & Dimension & Symplectic or Orthogonal? \\ \hline
        $V_1$ & $[0,0,0]$ & 1 & Orthogonal \\
        $V_7$ & $\left[ 0,-1,1 \right]$ & 7 & Orthogonal 
\end{tabular}
\caption{Irreducible representations of $G_2$ of dimension at most 13.}
\end{table}
\end{center}

Using this notation, the following table classifies all of the conjugacy classes of subalgebras $\mf{h}$ of type $G_2$ in $\mf{g}$ for various $\mf{g}$.

\begin{center}
\begin{longtable}[H]{c|l|c|l} 
        $\mf{g}$ & Dynkin & Standard rep.~of  & \hspace{1cm}Cartan embedding \vspace{-.5em} \\ 
         & \;\,index & $\mf{g}$ restricted to $\mf{h}$ & \\ \hline  
        $B_3 $ & \;\;\;\;\,1  & $V_7$ & $[x,y,-x-y] \mapsto [x-y, 2x+y, x+2y]$ \\ \hline 
        $B_4 $ & \;\;\;\;\,1  & \;\;\,$V_7\oplus V_1^{\oplus 2}$ & $[x,y,-x-y] \mapsto [x-y, 2x+y, x+2y,0]$ \\ \hline 
        $D_4$ & \;\;\;\;\,1  & $V_7 \oplus V_1$ & $[x,y,-x-y] \mapsto [x-y, 2x+y, x+2y,0]$ \\ \hline 
        $B_5 $ & \;\;\;\;\,1  & \;\;\,$V_7 \oplus V_1^{\oplus 4}$ & $[x,y,-x-y] \mapsto [x-y, 2x+y, x+2y,0,0]$ \\ \hline 
        $D_5 $ & \;\;\;\;\,1  & \;\;\,$V_7 \oplus V_1^{\oplus 3}$ & $[x,y,-x-y] \mapsto [x-y, 2x+y, x+2y,0,0]$ \\ \hline 
        $A_6$ & \;\;\;\;\,2  & $V_7$ & $[x,y,-x-y] \mapsto [x-y, 2x+y, x+2y, y-x, -2x-y, -x-2y, 0]$ \\ \hline 
        $B_6 $ & \;\;\;\;\,1  & \;\;\,$V_7 \oplus V_1^{\oplus 6}$ & $[x,y,-x-y] \mapsto [x-y, 2x+y, x+2y,0,0,0]$ \\ \hline 
        $D_6 $ & \;\;\;\;\,1  & \;\;\,$V_7 \oplus V_1^{\oplus 5}$ & $[x,y,-x-y] \mapsto [x-y, 2x+y, x+2y,0,0,0]$ \\ 
\end{longtable}
\end{center}

\vspace{-.5cm}

\begin{table}[H]
 \caption{Conjugacy classes of subalgebras of type $G_2$ in various $\mf{g}$.} \label{table:G2-subalgebras}
\end{table}

\newpage

\subsection*{Subalgebras of type $A_3 \simeq D_3$}

We denote the irreducible representations of $A_3 \simeq D_3$ of dimension at most 13 as follows:

\begin{center}
\begin{table}[H]
\begin{tabular}{c|l|l|c|c}
        Label & Highest Weight & Highest Weight & Dimension & Symplectic or \vspace{-.5em} \\ 
         & ($A_3$ coords.) & ($D_3$ coords.) & & Orthogonal? \\ \hline 
        $V_1$ & \;$[0,0,0,0]$ & \;\;\;\;\;$[0,0,0]$ & 1 & Orthogonal \\ 
        $V_4$ & \;$\left [ \frac{3}{4}, - \frac{1}{4}, - \frac{1}{4}, - \frac{1}{4} \right ]$ & \;\;\;\;\;$\left [ \frac{1}{2}, \frac{1}{2}, \frac{1}{2} \right ] $ & 4 & Neither \\
        $V_4^*$ & \;$\left [ \frac{1}{4}, \frac{1}{4}, \frac{1}{4}, - \frac{3}{4} \right ]$ & \;\;\;\;\;$\left [ \frac{1}{2}, \frac{1}{2}, - \frac{1}{2} \right ] $  & 4 & Neither \\
      $\bigwedge^2 V_4$  & \;$\left [ \frac{1}{2}, \frac{1}{2}, - \frac{1}{2}, - \frac{1}{2} \right ]$ & \;\;\;\;\;$[1,0,0]$ & 6 & Orthogonal \\ 
       $\Sym^2 V_4$  & \;$\left [ \frac{3}{2}, - \frac{1}{2}, - \frac{1}{2}, - \frac{1}{2} \right ]$ & \;\;\;\;\;$[1,1,1]$ & 10 & Neither \\ 
       $\Sym^2 V_4^*$  & \;$\left [ \frac{1}{2}, \frac{1}{2}, \frac{1}{2}, - \frac{3}{2} \right ]$ & \;\;\;\;\;$[1,1,-1]$ & 10 & Neither 
\end{tabular}
\caption{Irreducible representations of $A_3 \simeq D_3$ of dimension at most 13. (Note that $\bigwedge^2 V_4$ can also be viewed as the standard representation of $D_3$.)}
\end{table}
\end{center}

Using this notation, the following table classifies all of the conjugacy classes of subalgebras $\mf{h}$ of type $A_3 \simeq D_3$ in $\mf{g}$ for various $\mf{g}$.

\small 

\begingroup 
\setlength{\LTleft}{0pt minus 1000pt}
\setlength{\LTright}{0pt minus 1000pt}

\begin{center}
\begin{longtable}[H]{c|l|c|c c}
        $\mf{g}$ & Dynkin & Standard rep.~of & Cartan embedding & \vspace{-.5em} \\ 
        & \;\,index & $\mf{g}$ restricted to $\mf{h}$ & & \\ \hline 
        $B_3 $ & \;\;\;\;\,1  & $\bigwedge^2 V_4 \oplus V_1$ & $[x,y,z,-x-y-z] \mapsto [x-z,y-z,x+y]$ & ($A_3$ coords.) \\ 
         & &  & $[p,q,r] \mapsto [p,q,r]$ & ($D_3$ coords.) \\ \hline 
        $A_4$ & \;\;\;\;\,1 & $V_4 \oplus V_1$ & $[x,y,z,-x-y-z] \mapsto [x,y,z,-x-y-z,0]$ & ($A_3$ coords.) \\
        && & $[p,q,r] \mapsto \frac{1}{2} \left [ -p+q-r, p-q-r, -p-q+r, p+q+r, 0 \right ] $ & ($D_3$ coords.) \\
        & \;\;\;\;\,1 & $V_4^* \oplus V_1$ & $[x,y,z,-x-y-z] \mapsto [-x,-y,-z,x+y+z,0]$ & ($A_3$ coords.) \\
        & &  & $[p,q,r] \mapsto \frac{1}{2} \left [ p-q+r, -p+q+r, p+q-r, -p-q-r, 0 \right ]$  & ($D_3$ coords.) \\
\hline
\end{longtable}
\end{center}

\endgroup
\normalsize

\vspace{-.5cm}

\begin{table}[H]
 \caption{Conjugacy classes of subalgebras of type $A_3 \simeq D_3$ in various $\mf{g}$.} \label{table:A3-subalgebras}
\end{table}

\subsection*{Subalgebras of type $B_3$}

We denote the irreducible representations of $B_3$ of dimension at most 13 as follows:

\begin{center}
\begin{table}[H]
\begin{tabular}{c|c|c|c}
        Label & Highest Weight & Dimension & Symplectic or Orthogonal? \\ \hline
        $V_1$ & $[0,0,0]$ & 1 & Orthogonal \\
        $V_7$ & $\left[ 1,0,0 \right]$ & 7 & Orthogonal 
\end{tabular}
\caption{Irreducible representations of $B_3$ of dimension at most 13.}
\end{table}
\end{center}

Using this notation, the following table classifies all of the conjugacy classes of subalgebras $\mf{h}$ of type $B_3$ in $\mf{g}$ for various $\mf{g}$.

\begin{center}
\begin{longtable}[H]{c|l|c|l} 
        $\mf{g}$ & Dynkin & Standard rep.~of  & Cartan embedding \vspace{-.5em} \\ 
         & \;\,index & $\mf{g}$ restricted to $\mf{h}$ & \\ \hline  
        $B_4 $ & \;\;\;\;\,1  & \;\;\,$V_7 \oplus V_1^{\oplus 2}$ & $[x,y,z] \mapsto [x,y,z,0]$ \\ \hline 
        $D_4$ & \;\;\;\;\,1  & $V_7 \oplus V_1$ & $[x,y,z] \mapsto [x,y,z,0]$ \\ \hline 
        $B_5 $ & \;\;\;\;\,1  & \;\;\,$V_7 \oplus V_1^{\oplus 4}$ & $[x,y,z] \mapsto [x,y,z,0,0]$ \\ \hline 
        $D_5 $ & \;\;\;\;\,1  & \;\;\,$V_7 \oplus V_1^{\oplus 3}$ & $[x,y,z] \mapsto [x,y,z,0,0]$ \\ \hline 
        $A_6$ & \;\;\;\;\,2  & $V_7$ & $[x,y,z] \mapsto [x,y,z,-x,-y,-z,0]$ \\ \hline 
        $B_6 $ & \;\;\;\;\,1  & \;\;\,$V_7 \oplus V_1^{\oplus 6}$ & $[x,y,z] \mapsto [x,y,z,0,0,0]$ \\ \hline 
        $D_6 $ & \;\;\;\;\,1  & \;\;\,$V_7 \oplus V_1^{\oplus 5}$ & $[x,y,z] \mapsto [x,y,z,0,0,0]$ \\ 
\end{longtable}
\end{center}

\vspace{-.5cm}

\begin{table}[H]
 \caption{Conjugacy classes of subalgebras of type $B_3$ in various $\mf{g}$.} \label{table:B3-subalgebras}
\end{table}

\subsection*{Subalgebras of type $C_3$}

We denote the irreducible representations of $C_3$ of dimension at most 13 as follows:

\begin{center}
\begin{table}[H]
\begin{tabular}{c|c|c|c}
        Label & Highest Weight & Dimension & Symplectic or Orthogonal? \\ \hline
        $V_1$ & $[0,0,0]$ & 1 & Orthogonal \\
        $V_6$ & $\left[ 1,0,0 \right]$ & 6 & Symplectic 
\end{tabular}
\caption{Irreducible representations of $C_3$ of dimension at most 13.}
\end{table}
\end{center}

Using this notation, the following table classifies all of the conjugacy classes of subalgebras $\mf{h}$ of type $C_3$ in $\mf{g}$ for various $\mf{g}$.

\begin{center}
\begin{longtable}[H]{c|l|c|l} 
        $\mf{g}$ & Dynkin & Standard rep.~of  & Cartan embedding \vspace{-.5em} \\ 
         & \;\,index & $\mf{g}$ restricted to $\mf{h}$ & \\ \hline  
        $C_4 $ & \;\;\;\;\,1  & \;\;\,$V_6 \oplus V_1^{\oplus 2}$ & $[p,q,r] \mapsto [p,q,r,0]$ \\ \hline 
        $A_5$ & \;\;\;\;\,1  & $V_6$ & $[p,q,r] \mapsto [p,q,r,-p,-q,-r]$ \\ \hline 
        $C_5 $ & \;\;\;\;\,1  & \;\;\,$V_6 \oplus V_1^{\oplus 4}$ & $[p,q,r] \mapsto [p,q,r,0,0]$ \\ \hline 
        $A_6 $ & \;\;\;\;\,1  & $V_6 \oplus V_1$ & $[p,q,r] \mapsto [p,q,r,-p,-q,-r,0]$ \\ \hline 
        $B_6$ & \;\;\;\;\,1  & $V_6 \oplus V_6 \oplus V_1$ & $[p,q,r] \mapsto [p,p,q,q,r,r]$ \\ \hline 
        $C_6 $ & \;\;\;\;\,1  & \;\;\,$V_6 \oplus V_1^{\oplus 6}$ & $[p,q,r] \mapsto [p,q,r,0,0,0]$ \\ 
         & \;\;\;\;\,2  & $V_6 \oplus V_6$ & $[p,q,r] \mapsto [p,p,q,q,r,r]$ \\ \hline 
        $D_6 $ & \;\;\;\;\,$1'$  & $V_6 \oplus V_6$ & $[p,q,r] \mapsto [p,p,q,q,r,-r]$ \\ 
         & \;\;\;\;\,$1''$  & $V_6 \oplus V_6$ & $[p,q,r] \mapsto [p,p,q,q,r,r]$   \\ 
\end{longtable}
\end{center}

\vspace{-.5cm}

\begin{table}[H]
 \caption{Conjugacy classes of subalgebras of type $C_3$ in various $\mf{g}$.} \label{table:C3-subalgebras}
\end{table}

\end{document}